\documentclass[12pt]{article}

\usepackage[left=1in,right=1in,top=1in,bottom=1in]{geometry}

\usepackage{subcaption}
\usepackage{graphicx}
\usepackage{amsfonts}
\usepackage{url}
\usepackage{amsmath,amssymb}
\usepackage{enumerate, placeins}
\usepackage{amsthm,amsmath}
\usepackage{float}

\usepackage{mathtools}

\usepackage[colorlinks]{hyperref}

\newtheorem{thm}{Theorem}

\newtheorem{lemma}{Lemma}

\newtheorem{remark}{Remark}
\newtheorem{prop}{Proposition}

\def\R{\mathbb{R}}
\def\Z{\mathbb{Z}}

\def\P{{\mathbb P}}     
\def\E{{\mathbb E}}

\def\address#1{\gdef\@address{#1}}

\def\affilnum#1{${}^{#1}$}
\def\affil#1{${}^{#1}$}
\address{%
\addr{\affilnum{1}}{Department of Mathematics, Indiana University, Bloomington, IN 47405, USA}
\addr{}{All authors contributed equality to this work.}
}

\title{Constrained Langevin approximation for the  Togashi-Kaneko model of autocatalytic reactions}

\author{
Wai-Tong (Louis) Fan{\affil{*}},
  \and 
  Yifan (Johnny) Yang{{\affil{*}}},
\and
Chaojie Yuan\thanks{Department of Mathematics, Indiana University, Bloomington, IN 47405, USA}
}

\begin{document}

\maketitle

\begin{abstract}
The Togashi Kaneko model (TK model) introduced in \cite{togashi2001transitions} is a  simple stochastic reaction network that displays  discreteness-induced transitions between meta-stable patterns. Here
we study a constrained Langevin approximation (CLA) of this model. The CLA, obtained from  \cite{anderson2019constrained, leite2019constrained}, is an obliquely reflected diffusion process on the positive orthant and hence it respects the constrain that chemical concentrations are never negative. We show that
the CLA is a Feller process, is positive Harris recurrent, and converges exponentially fast to the unique stationary distribution. We also characterize the stationary distribution and show that it has finite moments. In addition, we simulate both the TK model and its CLA in various dimensions. For example, we describe how the TK model switches between meta-stable patterns in dimension 6. Our simulations
suggest that, under the classical scaling, the CLA is a good approximation to the TK model in terms of both the stationary distribution and the transition times between patterns. 
\end{abstract}

\section{Introduction}


In 2001, Togashi and Kaneko \cite{togashi2001transitions} introduced a simple model of autocatalytic reactions 
that displays a peculiar ``switching behavior" in some regions of the parameter space. The system state switches between patterns where a few species are abundant and the remaining species are almost absent, demonstrating  multi-stability at those patterns; see Figures \eqref{fig:traj_2D}-\eqref{fig:traj_6D} for some sample trajectories.
Paraphrasing \cite{bibbona2020stationary}, it is believed that ``the switching is triggered by a single molecule of a previously extinct species that drives the system to a different pattern through a sequence of quick reactions".
The emergence of such multi-stability induced by the small number effect, called \textbf{discreteness-induced transitions (DITs)} in \cite{togashi2001transitions},  has been observed in many complicated models in physics, biology and other scientific fields. For instance, it is reported in catalytic chemical reactions \cite{samoilov2005stochastic,awazu2007discreteness,kobayashi2011connection, biancalani2012noise}, reaction-diffusion systems \cite{togashi2004molecular,butler2011fluctuation}, gene regulatory networks \cite{to2010noise,ma2012small}, cancer tumor evolution \cite{sardanyes2018noisea}, virus replication \cite{sardanyes2018noiseb}, ecology \cite{biancalani2014noise}.

The widespread nature of DIT attracted many theoretical studies on the  model in \cite{togashi2001transitions} and its variants. These studies include analysis for
the switching time \cite{biancalani2014noise,houchmandzadeh2015exact,saito2015theoretical},  stationary distributions \cite{hoessly2019stationary,bibbona2020stationary},  separation of time scale \cite{biancalani2012noise,mcsweeney2014stochastically} and multimodality \cite{plesa2019noise,ali2019multi}. 
In \cite{houchmandzadeh2015exact,saito2015theoretical}, the authors analysed a mass-conserved reaction network, namely the autocatalytic reactions \eqref{auto} below with $d\in\{2,3\}$ species, together with mutations between  species
\begin{align}
A_j \overset{\epsilon}{\underset{}{\rightarrow}} & \;A_i, \qquad \quad  1\leq i\neq j\leq d. \label{muta}
\end{align}
The total number $N$ of molecules among all  species remains constant in time for this model, make it more amenable to analysis. For $d=2$,
by \cite[eqn.(7)]{houchmandzadeh2015exact}, the mean time to move from one boundary state to the other is approximately 
\begin{equation}\label{2D_time}
    \frac{1}{\epsilon} + \frac{2}{\kappa}\frac{N-1}{N}, \qquad \text{as}\quad \frac{\epsilon}{\kappa} \to 0.
\end{equation}
In \cite{saito2015theoretical}, a similar model with $d=3$ species (and more reactions) is studied, where  a noise-induced reversal of chemical current was observed as the total number of molecules decreases.

For the model in \cite{togashi2001transitions},
henceforth called  the \textbf{TK model}, mass is no longer conserved. 
The reaction network of the TK model consists of  a cycle of autocatalytic reactions, together with infow and outflow reactions.
When there are $d$ species $\{A_i\}_{i=1}^d$, the reaction network is
\begin{align}
A_i+A_{i+1} \overset{\kappa}{\rightarrow} & \;2A_{i+1},\qquad i=1,2,\cdots d, \quad \text{where }A_{d+1}=A_1,  \label{auto}\\
\emptyset \overset{\lambda}{\underset{\delta}{\rightleftarrows}} & \;A_i, \qquad \quad i=1,2,\cdots d  \label{inout}
\end{align}

Suppose  $X^i_t$ represents the number of molecules for $A_i$ at time $t$. It is standard to assume that the vector $X_t=(X^i_t)_{i=1}^d\in \Z_+^d$ evolves according to
a \textbf{continuous-time Markov chain (CTMC)} over time, with transition rates specified by  mass-action kinetics. More precisely, we let $X^{d+1}=X^1$ by convention, and construct this stochastic process $X=(X_t)_{t\in\R_+}$ as the solution to the stochastic equation
\begin{align}
X_t=X_0
+ & \sum_{i=1}^d (e_{i+1}-e_i)\,\mathcal{N}^{\kappa}_{i}\left(\kappa\int_0^tX^i_s X^{i+1}_s ds\right)
\label{auto_CTMC} \\
+& \sum_{i=1}^d e_i \,\mathcal{N}^{\lambda}_{i}(\lambda t)
-  \sum_{i=1}^d e_i \,\mathcal{N}^{\delta}_{i}\left(\delta\int_0^t X^i_s ds\right)
,\quad t\in \R_+,\label{inout_CTMC}
\end{align}
where $\{e_i\}_{i=1}^d$ are the standard basis unit vector of $\R^d$, and $\{\mathcal{N}^{\kappa}_i,\mathcal{N}^{\lambda}_i,\mathcal{N}^{\delta}_i\}_{i=1}^d$ are independent Poisson processes with unit rate; see the monograph \cite{anderson2011continuous} for basic properties of this equation.
Note that the sum of all coordinates $\sum_{i=1}^dX^i$ is equal in distribution to 
the process $S$ solving the following equation that no longer depend on  $\kappa$:
\[
S_t=S_0 +\mathcal{N}_1(d\lambda t) - \mathcal{N}_2\left(\delta \int_0^t S_s \,ds\right),\quad t\in \R_+,
\]
where $\mathcal{N}_1$ and $\mathcal{N}_2$ are independent Poisson processes with unit rate. Therefore, 
the total mass $\sum_{i=1}^dX^i_t$ is an immigration-death process that converges, as $t\to\infty$, to the Poisson distribution with mean $\frac{d\lambda}{\delta}$ exponentially fast with rate $\delta$ (see \cite[Chapter 9]{chen2004markov}).

In \cite[Theorem 4.1]{bibbona2020stationary} it was shown that the CTMC  described by \eqref{auto_CTMC}-\eqref{inout_CTMC}  is positive recurrent and converges exponentially fast to its stationary distribution when $\delta>0$ and $\lambda>0$. For the special case when $\delta=\frac{d\,\kappa}{d-1}$,  an explicit form of the stationary distribution is known  in  \cite[Theorem 4.3]{bibbona2020stationary}. Beyond these, not much is proven about \eqref{auto_CTMC}-\eqref{inout_CTMC} in general dimensions. Also, asymptotic formula like \eqref{2D_time} is not known for \eqref{auto_CTMC}-\eqref{inout_CTMC}, and
it is not clear how to directly compare the model \eqref{auto}-\eqref{inout} with the model \eqref{auto}-\eqref{muta},  because the reaction vector at the boundary of $\Z_+^d$ are different for the two models. 


A standard  model reduction technique is to look at the mean field approximation or the diffusion (Langevin) approximation to \eqref{auto_CTMC}-\eqref{inout_CTMC}. However, none of them is a good predictor: ordinary differential equations does not capture DIT in \eqref{auto_CTMC}-\eqref{inout_CTMC} because DIT for this case is due to successive extinction and revival of species, and mean field approximation breaks down when  the abundance of some species are not high. 
Central limit theorems and diffusion approximations \cite{anderson2011continuous,kang2014central} can capture the fluctuation around the deterministic ODE for  stochastic chemical reaction networks, but they can have negative coordinates, which is unrealistic and posses a technical issue: the diffusion process may not remain well-defined when a coordinate becomes negative.

To address these issues,
an obliquely reflected diffusion called \textbf{constrained Langevin approximation (CLA)}
was proposed in \cite{anderson2019constrained, leite2019constrained} as a better diffusion approximation to  CTMC that arise from chemical reaction networks. 
This  obliquely reflected diffusion process has state space in the positive orthant $\R_+^d$ (where $d$ is the number of species) and thus respect the constraint that chemical concentrations are never negative. Intuitively, a CLA behaves like a diffusion process inside the strictly positive orthant $\R^d_{>0}$ and reflects instantaneously at a boundary face in the direction specified by a vector field. Special care need to be taken for  reflection at the intersection of two or more faces. It was demonstrated in \cite{anderson2019constrained}
through numerical studies that, in addition to having the correct support, the stationary distribution for CLA can capture the behavior of the CTMC more accurately than the usual diffusion approximation.

Existing analysis of the TK model and its variants are mostly restricted to models with a small number of species, and they do not cover the analysis of the CLA of the corresponding CTMC.
In this paper, we analyze  the CLA for the TK model
\eqref{auto_CTMC}-\eqref{inout_CTMC} in general dimensions, and we perform a  simulation study for both the CTMC \eqref{auto_CTMC}-\eqref{inout_CTMC} and the CLA. 

\noindent
{\bf Organization of this paper. }
In Section \ref{S:CLA}, we show that the CLA possesses the Feller property, is positive Harris recurrent and converges exponentially fast to the unique stationary distribution $\pi$. We also show that $\pi$ has finite moments and we characterize it in terms of an elliptic partial differential equation. The proofs of these results, heavily based on the  Foster Lyapunov function approach, are presented in Section \ref{S:Proof}. Finally, Section \ref{S:simulation} contains our simulation study for the TK model and its CLA in dimensions $d=2,3$ and higher. 
We demonstrate that, at least in dimensions $d=2$ and 3, the CLA can capture both the stationary distribution and the expected transition time of the TK model when $V$ is large enough.

\medskip




\section{Analytical results for the constrained Langevin approximation}\label{S:CLA}

In this section, we describe and analyze a constrained Langevin approximation (CLA) to the TK model \eqref{auto_CTMC}-\eqref{inout_CTMC}. Precisely, the CLA is the strong solution to \eqref{Reflected_General}. In the 3 subsections below, we first establish wellposedness of equation  \eqref{Reflected_General} and the Feller property of the CLA. We then show that the CLA is positive Harris recurrent and  exponentially ergodic. Finally we characterize the stationary distribution $\pi$ and show that it has finite moments. The proofs of these results are in subsection \ref{SS:Proofs}.


\medskip

Under the classical scaling, the initial  molecule counts are proportional to a scaling parameter $V$, and 
the rate constants are of order $\kappa=O(V^{-1})$, $\lambda=O(V)$ and $\delta=O(1)$ as $V\to\infty$. 
So we let
\begin{equation}\label{classicalScaling}
\kappa=\frac{\kappa'}{V},\quad\delta=\delta'\quad \text{and}\quad \lambda=\lambda'V,
\end{equation}
where $\kappa'$, $\lambda'$ and $\delta'$ are  constants that will emerge in the mean-field approximation as $V\to\infty$.
In chemical reactions, $V$ denotes the Avogadro’s number times the volume of the vessel in which all the reactions take place.


Following the general method in \cite{anderson2019constrained, leite2019constrained},  a CLA for the TK model \eqref{auto}-\eqref{inout} is described by the stochastic differential equation with reflection (SDER)
\begin{equation}\label{Reflected_General}
d Z^{(V)}_t = b(Z^{(V)}_t) \,dt + \frac{1}{\sqrt{V}} {\sigma}(Z^{(V)}_t)\,d{W}_t + \frac{1}{\sqrt{V}}\,{\gamma}(Z^{(V)}_t)\,dL_t,
\end{equation}
where $W$ is a $d$-dimensional Brownian motion, $b:\R_+^d\to \R_+^d$ and $\gamma:\partial\R_+^d\to \R_+^d$ are functions given by
\begin{equation}\label{CLA_coe1}
    b(x) = \sum_{k = 1}^d e_k\Big(\kappa'(x_{k-1} - x_{k+1})x_k + \lambda' - \delta'x_k\Big)\quad \text{ and }\quad\gamma(x) = \frac{b(x)}{|b(x)|},
\end{equation}
where  $\{e_k\}_{k = 1}^n$ is the standard basis in $\R^d$, $x = (x_k)_{k = 1}^d$,
and $|\cdot|$ is the usual Euclidean norm.
The function $\sigma$  is the $d\times d$-matrix-valued function on $\R_+^d$ given by $\sigma(x) = \sqrt{\Gamma(x)}$ where, by \cite[eqn (29)]{anderson2019constrained},
\begin{equation}\label{CLA_coe2}
    \Gamma (x) 
    =\sum_{k = 1}^d e_{kk}\left(\kappa'(x_{k-1}+x_{k+1})x_{k} + \lambda' + \delta'x_k\right) -\sum_{k = 1}^d \kappa'x_kx_{k+1}(e_{k,k+1}+e_{k+1,k}), 
\end{equation}
where  $e_{i,j}\in \R^{d\times d}$ is the matrix whose $(i,j)$-th entry is one and all other entries are zero. Note that the matrix $\Gamma(x)$ is symmetric and strictly positive definite (or uniformly elliptic) for all $x \in \R_+^d$, see subsection \ref{SS:coe} for detail.

\begin{remark}\rm
The square root $\sigma = \sqrt{\Gamma}$ in  \eqref{Reflected_General}  is implicit and need to be calculated in practice. Another, more explicit equation that gives the same process $Z$ in distribution is to use a higher dimensional Brownian motion. Namely, for $d\ge 3$, let $\widetilde{W}$ be a  $2d$-dimensional Brownian motion, and replace ${\sigma}(Z^{(V)}_t)\,d{W}_t$ in \eqref{Reflected_General} by ${\widetilde{\sigma}}(Z_t)\,d\widetilde{W}_t$, where
$\widetilde{\sigma}$ is an explicit $d\times 2d$-matrix-valued function. For example,
for $d=3$, $\widetilde{W}$ is a 6-dimensional Brownian motion and
\[
        \widetilde{\sigma}(x)= 
        \begin{pmatrix}
        -\sqrt{\kappa' x_1 x_2} & 0 &
        \sqrt{\kappa' x_3 x_1} &  \sqrt{\lambda' +\delta x_1} & 0 &
       0
        \\
        \sqrt{\kappa' x_1 x_2}& -\sqrt{\kappa' x_2 x_3}& 0 &
        0& \sqrt{\lambda' +\delta x_2}& 0 \\
        0& \sqrt{\kappa' x_2 x_3} &- \sqrt{\kappa' x_3 x_1} &
        0& 0 &\sqrt{\lambda' +\delta x_3}
        \end{pmatrix}.
\]
For $d=2$, $\widetilde{W}$ is a 3-dimensional Brownian motion and
\[
\widetilde{\sigma}(x)= \,
\begin{pmatrix}
\sqrt{2\kappa' x_1 x_2} & 
\sqrt{\lambda' +\delta x_1} & 0
\\
-\sqrt{2\kappa' x_1 x_2}& 0& \sqrt{\lambda' +\delta x_2}
\end{pmatrix}.
\]
The matrix $\widetilde{\sigma}$ can be obtained from \eqref{auto_CTMC}-\eqref{inout_CTMC} as in \cite[Section 3]{anderson2019constrained},  under the classical scaling \eqref{classicalScaling}.
\end{remark}

\subsection{Path-wise existence and uniqueness of the CLA}

Path-wise existence of obliquely reflected diffusion can fail \cite{harrison1985stationary}. However, for \eqref{Reflected_General}, path-wise existence of solution holds since the reflection angles are nice enough ensured by \cite[Theorem 6.1]{leite2019constrained} or the argument in \cite[Theorem 5.1]{dupuis1993sdes}. 
We cannot find any existing result that can tell us 
whether the solution to \eqref{Reflected_General} is a Feller process or not. So we give a proof of the Feller property.

\begin{thm}[Path-wise solution and Feller property]\label{T:Wellposed}
For each $V\in (0,\infty)$ and initial condition $z_0\in \R_+^d$, there exists a unique path-wise solution to equation \eqref{Reflected_General}. The solutions starting from different points in  $\R_+^d$ form a family of Feller continuous strong Markov processes in $\R_+^d$.
\end{thm}

Let $Z=Z^{(V)}$ be the  solution to \eqref{Reflected_General} from now on, and omit the super script $V$ when there is no ambiguity.
A solution $Z$ to \eqref{Reflected_General} is a good approximation to $\frac{X}{V}$ on any compact time interval, as $V\to\infty$, when initially $Z_0=\frac{X_0}{V}$. More precisely, the CLA was obtained in  \cite{leite2019constrained} as the scaling limit of a sequence of jump-diffusion processes that are believed to be good approximations to the stochastic reaction network. These jump -diffusion processes behave like the standard Langevin approximation in the interior of the positive orthant and like a rescaled version of the CTMC on the boundary of the orthant. Though a rigorous connection between the CLA and the CTMC is still missing, simulation results in \cite{anderson2019constrained, leite2019constrained} demonstrated that the CLA is a remarkably good approximation to the CTMC. 

Our proofs depends heavily on the Foster-Lyapunov function approach \cite{meyn1993stability}. We consider the function
$U:\R^d_+ \to \R_+$ defined by
\begin{align}\label{lyapunov_function}
    U(x) = \left(|x|_1 - \frac{d\lambda'}{\delta'}\right)^2,
\end{align}
where $|x|_1 = \sum_{i = 1}^d |x_i|$. Importantly, this function is compatible with the reflection field on the boundary (Lemma \ref{lem:local_time_out}) and leads  to Lyapunov inequalities (Lemma \ref{lm:lyapunov}), which enables
our stability analysis for $Z$ that is needed to establish  Feller property, positive Harris recurrence, and exponential ergodicity of $Z$.

\subsection{Positive Harris recurrence and exponential ergodicity of the CLA}

In \cite[Theorem 4.1]{bibbona2020stationary} it was shown that the discrete TK model \eqref{auto_CTMC}-\eqref{inout_CTMC}  is positive Harris recurrent and converges exponentially fast to its stationary distribution. Here we obtain the analogous results for the CLA. 
Recall from \cite[Sections 3-4]{meyn1993stability} that  $Z$ is called \textbf{Harris recurrent} if there exists  a sigma-finite measure $\mu$ on $\R^d_+$ such that whenever $\mu(A)>0$,  we have
$\P_x(\tau_A<\infty)=1$  for all $x\in \R_+^d$, where $\tau_A=\inf\{t\in\R_+:\,Z_t\in A\}$ is the hitting time of a Borel set $A$.
If, furthermore, the invariant measure is finite, then $Z$ is called \textbf{positive Harris recurrent}. 


\begin{thm}[Positive recurrence]\label{thm:stationaryExist}
The solution $Z$ to \eqref{Reflected_General} is positive Harris recurrent and it has a unique stationary distribution $\pi$.
Furthermore, all moments of $\pi$ are finite.
\end{thm}

Next, we consider rate of convergence to stationarity. We say that $Z$ is \textbf{$f$-exponentially ergodic} for a  function $f$ if  the law of $Z_t$ converges to $\pi$ exponentially fast in the following sense: there exists a constant $\beta<1$ and a function $B:\R_+^d\to \R_+$ such that
\begin{equation}\label{f_expo}
    \| P^t(x,\cdot) - \pi \|_f \le B(x)\,\beta^t\quad \forall t \ge 0, \, x \in \R_+^d,
\end{equation}
where the $\|\cdot\|_f$-norm is defined as $ \|\mu\|_f := \sup_{g:\,|g| \le f} \left|\int_{\R_+^d} g d\mu\right|$, and
the supremum is taken over the space of Borel measurable functions $g$ on $\R_+^d$ with $|g(x)|\leq f(x)$ for all $x\in \R_+^d$. 
\begin{thm}[Exponential ergodicity]\label{thm:stationaryconv}
The solution $Z$ to \eqref{Reflected_General} is $f$-exponentially ergodic with $f=U+1$, where $U$ is defined in \eqref{lyapunov_function}.
\end{thm}

Theorem \ref{thm:stationaryconv} implies that
\eqref{f_expo} remains true if we replace $\|\cdot\|_{f}$ by the total variation distance $\|\cdot \|_{\rm TV}$. This is because $\|\mu_1-\mu_2\|_{\rm TV}\leq \|\mu_1-\mu_2\|_{f}$ when $f\geq 1$.

Positive recurrence can fail for reflected diffusion on unbounded domains (e.g. the reflected Brownian motion with a positive drift on $[0,\infty)$). Hence
suitable conditions on the state-dependent coefficients and the reflection vector field are needed. In \cite{atar2001positive}, the authors consider a reflected diffusion 
on a convex polyhedral cone $G\subset \R^d$ with vertex at the origin, 
the  reflection vector $v_i$ is assumed to be constant on each face $G_i$ of the cone. Let $\mathfrak{C}$ be the cone spanned by $\{-v_i\}_{i}$.
The main result \cite[Theorem 2.2]{atar2001positive} asserts that
the reflected diffusion is positive recurrent and has a unique invariant distribution if there is a bounded set $A\subset G$ such that the vector $b(x)$ is in the cone $\mathfrak{C}$ and uniformly away from the boundary of $\mathfrak{C}$, for all $x\in G\setminus A$.

Unfortunately, the result in \cite{atar2001positive,budhiraja2007long} cannot be applied to our CLA  directly since the reflection vector field for $Z$ is state-dependent. On other hand, the papers \cite{dupuis1993sdes,kang2014characterization,leite2019constrained} consider state-dependent reflection vector field on non-smooth domains, but these results are not concerned with positive recurrence. 
We shall prove Theorems \ref{thm:stationaryExist} and \ref{thm:stationaryconv} by the Foster-Lyapunov function approach in this paper.

\subsection{Characterization of the stationary distribution of the CLA}

Characterization of stationary distributions of a general class of reflected diffusions is given in \cite{kang2014characterization}. It was shown
that a stationary distribution, should it exists, must satisfy an adjoint linear elliptic partial differential equation with oblique derivative boundary conditions. This equation is called the basic adjoint relationship (BAR) for reflected Brownian motion in \cite{harrison1987brownian,dai1992reflected}.
Here, in subsection \ref{SS:Proofs}, we verify the conditions in  \cite[Theorems 2 and 3]{kang2014characterization} and apply those results to our CLA. 

Following \cite{kang2014characterization}, we let  $C_c^2(\R^d_+)$ be the space of twice   continuously differentiable functions on $\R_+^d$ with compact support. We consider the space of functions 
\begin{align}
    {\cal H} := \left\{f \in C^2_c\left(\R_+^d\right) \oplus \R:\, \langle \nabla f(x), \gamma(x)\rangle \ge 0,\, \forall x \in \partial \R_+^d \right\},
\end{align}
where  $C_c^2(\R^d_+)\oplus \R$ denotes the space of functions in $C_c^2(\R^d_+)$ plus a constant in $\R$. We also let ${\cal L}$ be the differential operator:
\begin{align}\label{generator}
    {\cal L}f(x) = \frac{1}{2V}\sum_{i,j = 1}^d \Gamma_{i,j}(x) \frac{\partial^2 f}{\partial x_i \partial x_j}(x) +\sum_{i = 1}^d b_i (x) \frac{\partial f}{\partial x_i}(x).
\end{align}

\begin{prop}\label{L:stationary1}
A probability measure $\pi$ on $\R_+^d$ is a stationary distribution of $Z$ if and only if $\pi(\partial \R_+^d)=0$ and 
    $\int_{\R_+^d} {\cal L}f(x)\, \pi(dx) \le 0$ for all $f \in {\cal H}$.
\end{prop}

\medskip

A more explicit way to characterize the stationary distribution is through a partial differential equation as in \cite[Theorem 3]{kang2014characterization}. For this we recall that the adjoint operator ${\cal L}^*$ of
${\cal L}$ is
\begin{align}\label{adjoint_generator}
    {\cal L}^*f(x) = \frac{1}{2V}\sum_{i,j = 1}^d \frac{\partial^2 }{\partial x_i \partial x_j}\left(\Gamma_{i,j}(x) f(x)\right) -\sum_{i = 1}^d \frac{\partial }{\partial x_i}\left(b_i (x) f(x)\right).
\end{align}
\begin{prop}\label{L:stationary2}
Suppose there exists a nonnegative integrable function $p \in {\cal C}^2(\R_+^d)$ that satisfies the following three relations:
\begin{enumerate}
    \item ${\cal L}^*p(x) = 0$ for all $x \in \R_+^d$;
    \item 
    For each $i\in\{1,2,\cdots,d\}$ and $x \in \left\{x_i = 0\right\}\cap \left\{x_j >0,\,\forall j \in \{1,2,\cdots,d\}\backslash{\{i\}}\right\}$, 
    \begin{align*}
        -2p(x)\lambda' + \lambda'\frac{\partial p}{\partial x_i} -\lambda' \nabla \cdot p(x)\gamma(x) + \frac{\partial p}{\partial x_i}\lambda' + p(x) \left(\kappa'(x_{i-1}+x_{i+1}) + \delta'\right) = 0;
    \end{align*}
    \item for each $1 \le i\neq j \le d$ and $x \in  \{x_i = 0\}\cap  \{x_j = 0\} \cap \partial \R_+^d$, $p(x) = 0$.
\end{enumerate}
Then the probability measure on $\R_+^d$ defined by 
\begin{align*}
    \pi(A) := \frac{\int_A p(x)dx}{\int_{\R_+^d} p(x)dx},\qquad A\in \mathcal{B}(\R_+^d),
\end{align*}
is a stationary distribution for the process  $Z$.
\end{prop}

The probability measure $\pi$ in Proposition \ref{L:stationary1} exists and is unique, by Theorem \ref{thm:stationaryExist}. 
We do not know if the function $p \in {\cal C}^2(\R_+^d)$ in Proposition \ref{L:stationary2} exists or not.


\section{Simulation study for the TK model and the CLA}\label{S:simulation}

In this section, we present simulation results for the CTMC  \eqref{auto_CTMC}-\eqref{inout_CTMC} and its associated CLA \eqref{Reflected_General}. In particular, dynamical properties of the CLA, including its stationary distribution $\pi$ (guaranteed in Theorem \ref{thm:stationaryExist}) and its finite time trajectory, are compared with those of the  CTMC in various dimensions. To begin, in Figure \ref{fig:finite_traj} we show some sample trajectories of the CTMC $X$ that solves \eqref{auto_CTMC}-\eqref{inout_CTMC} in dimensions $d=2,3,4,5,6$, under the parameter $\lambda = 1/4, \delta = 1/64$ and $\kappa = 1/16$ in \eqref{auto}-\eqref{inout}. Under this choice of rate constants, all processes tend to spend most of the time on the boundary, i.e., some species are almost extinct while others are abundant.

\begin{figure}[!htbp]
\begin{subfigure}{0.5\textwidth}
\includegraphics[width=\linewidth, height=5cm]{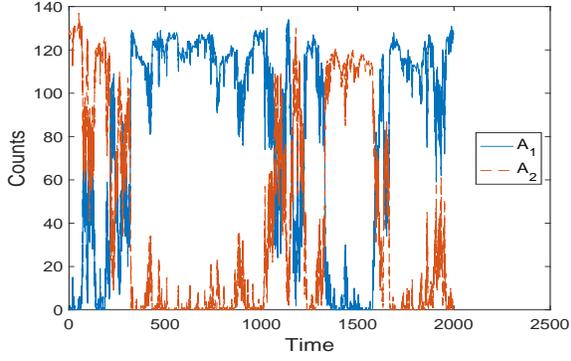}
\caption{$d=2$}
\label{fig:traj_2D}
\end{subfigure}
\begin{subfigure}{0.5\textwidth}
\includegraphics[width=\linewidth, height=5cm]{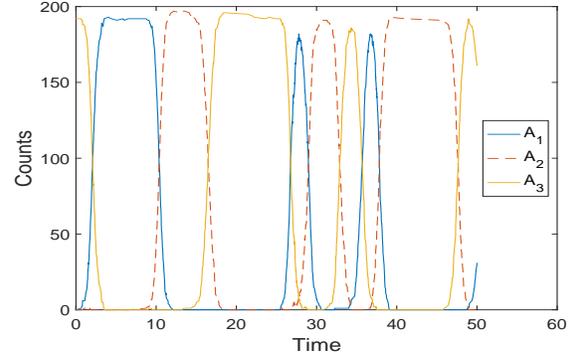}
\caption{$d=3$}
\label{fig:traj_3D}
\end{subfigure}

\begin{subfigure}{0.5\textwidth}
\includegraphics[width=1\linewidth, height=5cm]{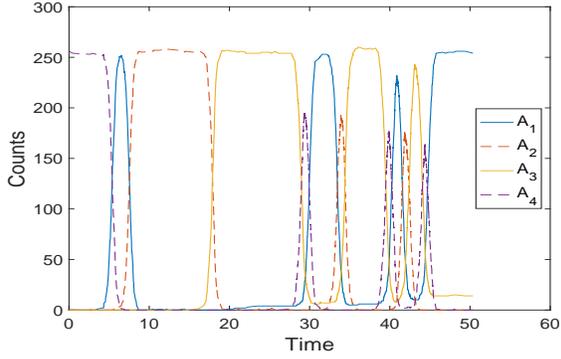}
\caption{$d=4$}
\label{fig:traj_4D}
\end{subfigure}
\begin{subfigure}{0.5\textwidth}
\includegraphics[width=1\linewidth, height=5cm]{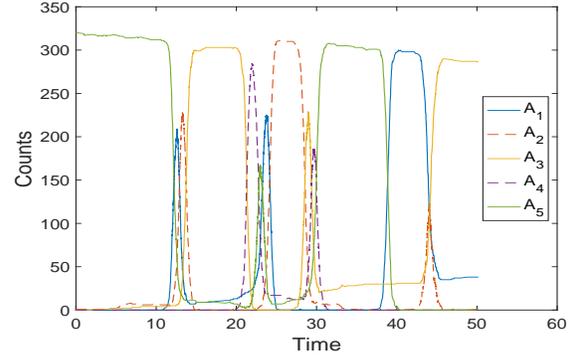}
\caption{$d=5$}
\label{fig:traj_5D}
\end{subfigure}
\begin{subfigure}{0.5\textwidth}
\includegraphics[width=1\linewidth, height=5cm]{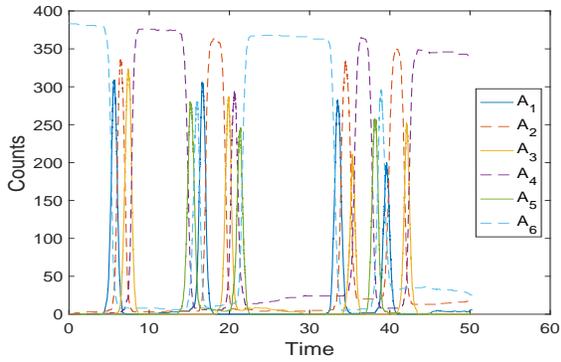}
\caption{$d=6$}
\label{fig:traj_6D}
\end{subfigure}
\caption{CTMC trajectories of standard TK model in \eqref{auto_CTMC}-\eqref{inout_CTMC} is plotted for different dimensions under the parameter $\lambda = 1/4, \delta = 1/64$ and $\kappa = 1/16$. The initial condition is $(0,\cdots,\,0,\,16d)$, i.e. $X^d_0 = 16d$ and zero for all other species. }
\label{fig:finite_traj}
\end{figure}


In Section \ref{sec:2DTK}, we simulate the CTMC and the CLA for $d=2$. Our simulations suggest that the stationary distributions and the hitting time distributions of CTMC are well approximated by those of the CLA in dimension $d=2$. Similar simulations and suggestions are obtained for $d=3$ in Section \ref{sec:3DTK}, there we also emphasize a discrepancy of finite trajectory property of TK models between $d=2$ and $d=3$. In section \ref{sec:HDTK}, we focus on higher dimensional TK models. For example, we 
give a description to the switching behaviors between meta-stable patterns of the 6-dimensional CTMC.

\smallskip

\noindent
{\bf Simulation schemes. }All CTMCs are simulated using the Gillespie algorithm \cite{gillespie1977exact}. Special care need to be taken  when simulating CLA, when the trajectory is near the boundary. Here all CLAs  are simulated via the modified Euler Maruyama method proposed in \cite{bossy2004symmetrized}. An alternative simulation scheme for the CLA may also be developed by a suitable discrete version of the local time \cite{fan2016discrete,chen2017hydrodynamic}. 

\FloatBarrier

\subsection{Simulation results for the 2-dimensional TK model}
\label{sec:2DTK}


In dimension $d=2$, our
simulations indicate that the CLA  \eqref{Reflected_General} nicely captures the stationary distribution of CTMC \eqref{auto_CTMC}-\eqref{inout_CTMC}. Furthermore, we consider  the time between extinction events of the two species (called the switching time). We demonstrate that the switching time distribution of the CLA capture that of the CTMC when the volume $V$ is large enough. 


For the 2-dimensional CTMC, explicit expression of the stationary distribution was derived in \cite{bibbona2020stationary}.  
Under the classical scaling \eqref{classicalScaling} with 
$$D = \lambda' = \delta',$$ 
the system exhibits different stationary behavior for different choices of $D$: when $D > 2/V$, the stationary distribution is unimodal; whereas when $D=2/V$, stationary distribution conditioning on the level sets $x+y=n$ is uniform; and when $D < 2/V$, stationary distribution is heavily concentrated on both boundaries where one species is almost extinct. Such behavior can be visualized in row 1 of Figure \ref{fig:stationary_2d}, where the stationary distributions are plotted for different choice of $D$.

\begin{figure}[!ht]
    \begin{subfigure}{0.32\textwidth}
    \includegraphics[width=\linewidth, height=4cm]{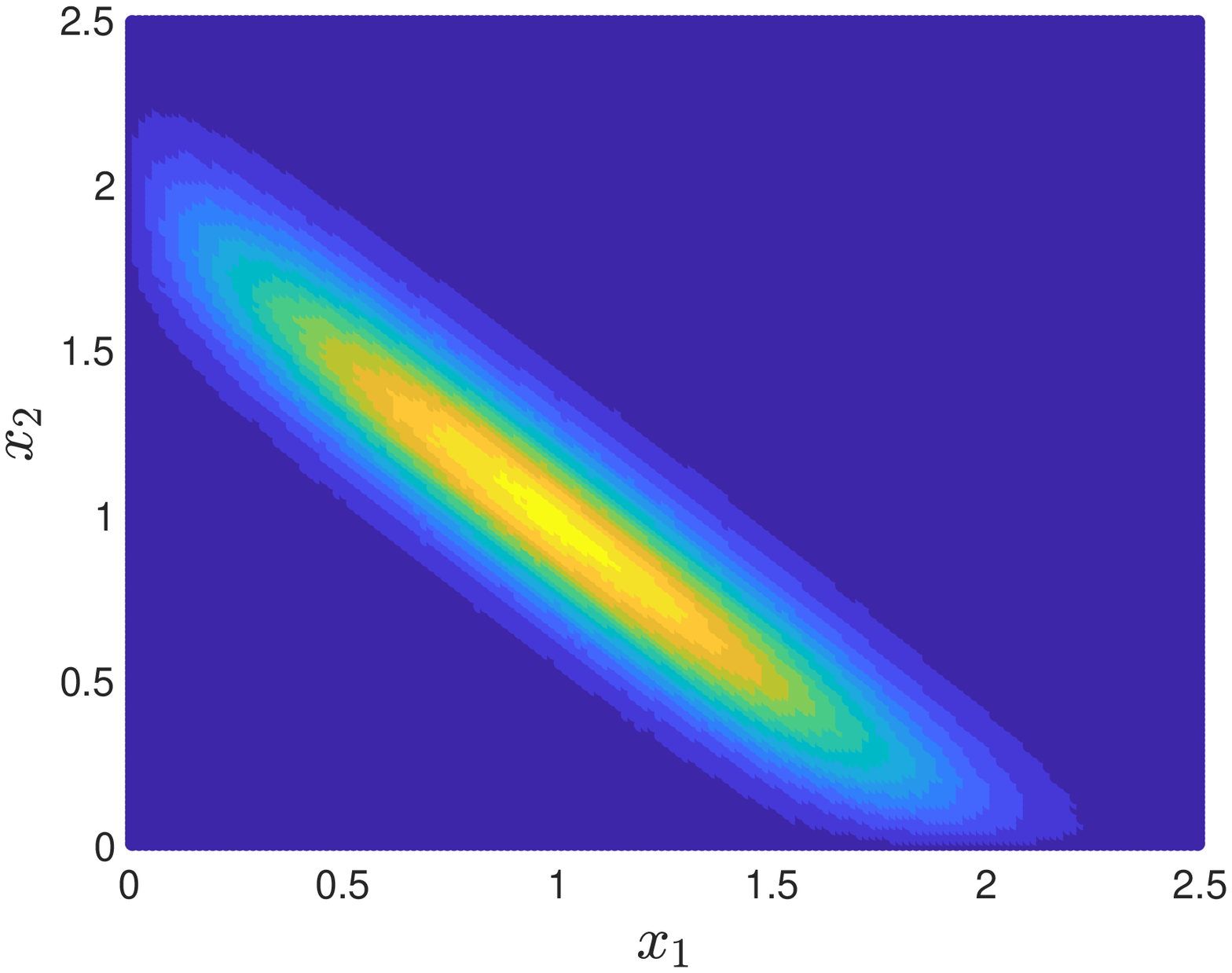}
    \end{subfigure}
    \begin{subfigure}{0.32\textwidth}
    \includegraphics[width=\linewidth, height=4cm]{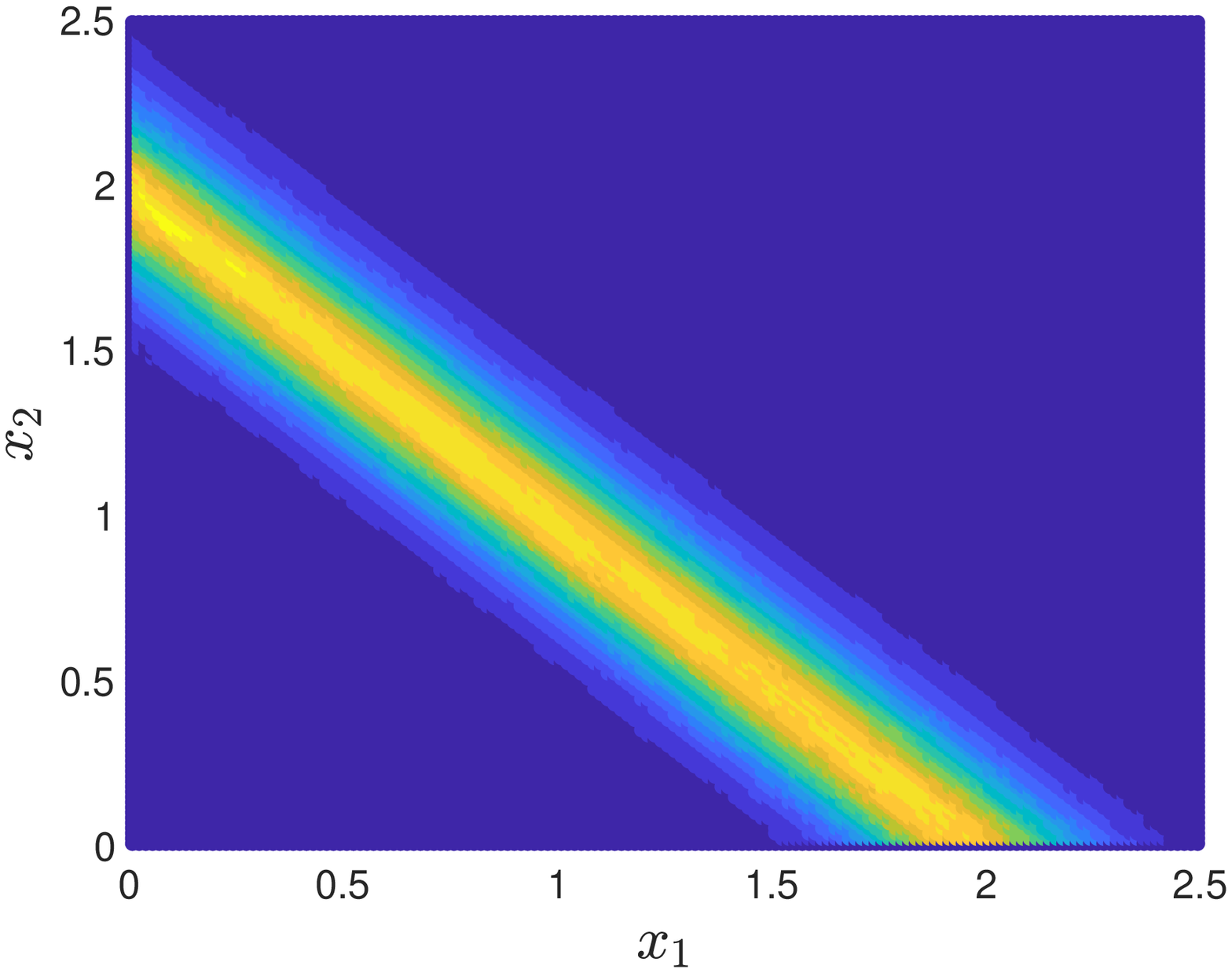}
    \end{subfigure}
    \begin{subfigure}{0.36\textwidth}
    \includegraphics[width=\linewidth, height=4cm]{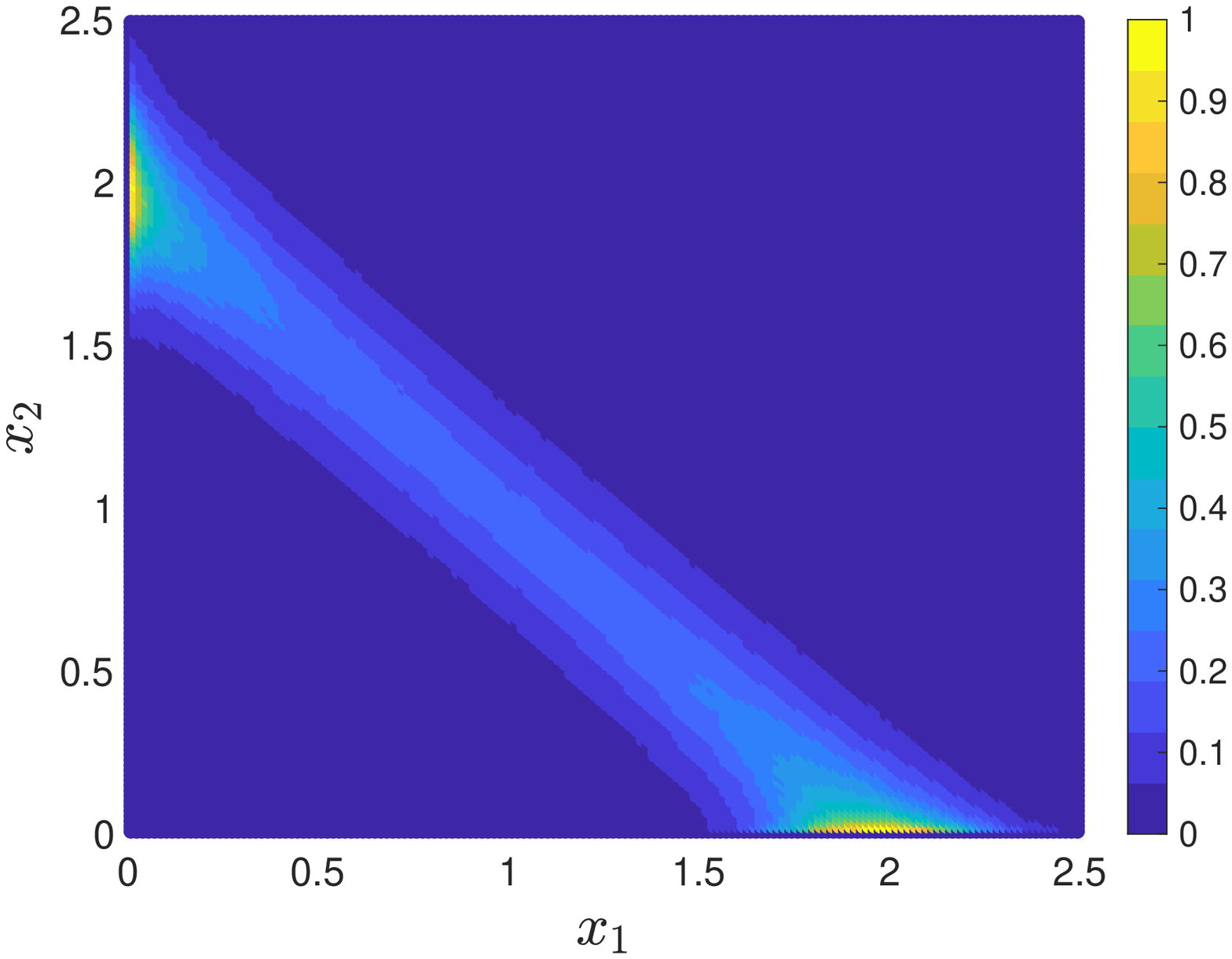}
    \end{subfigure}
    
    \begin{subfigure}{0.32\textwidth}
    \includegraphics[width=\linewidth, height=4cm]{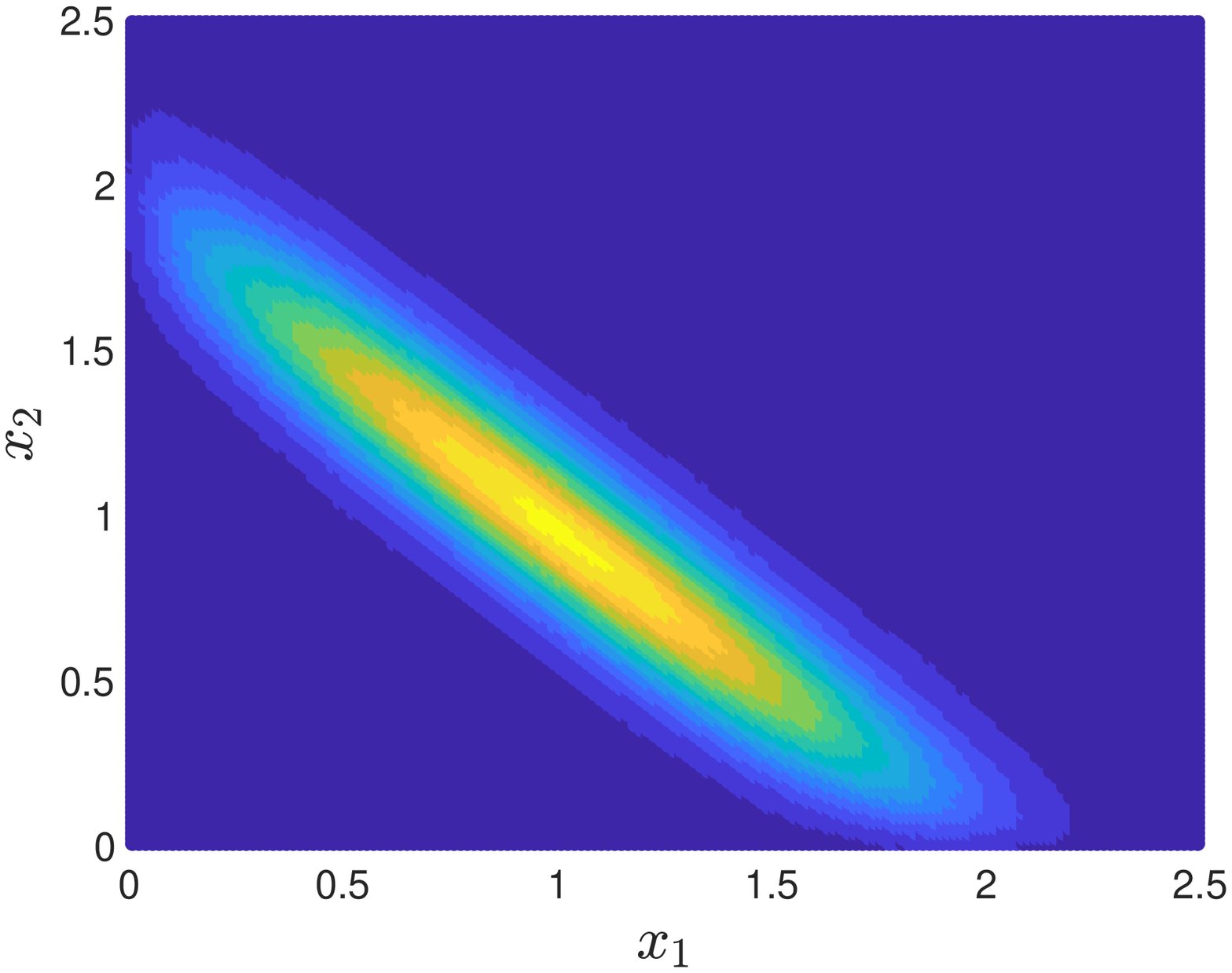}
    \end{subfigure}
    \begin{subfigure}{0.32\textwidth}
    \includegraphics[width=\linewidth, height=4cm]{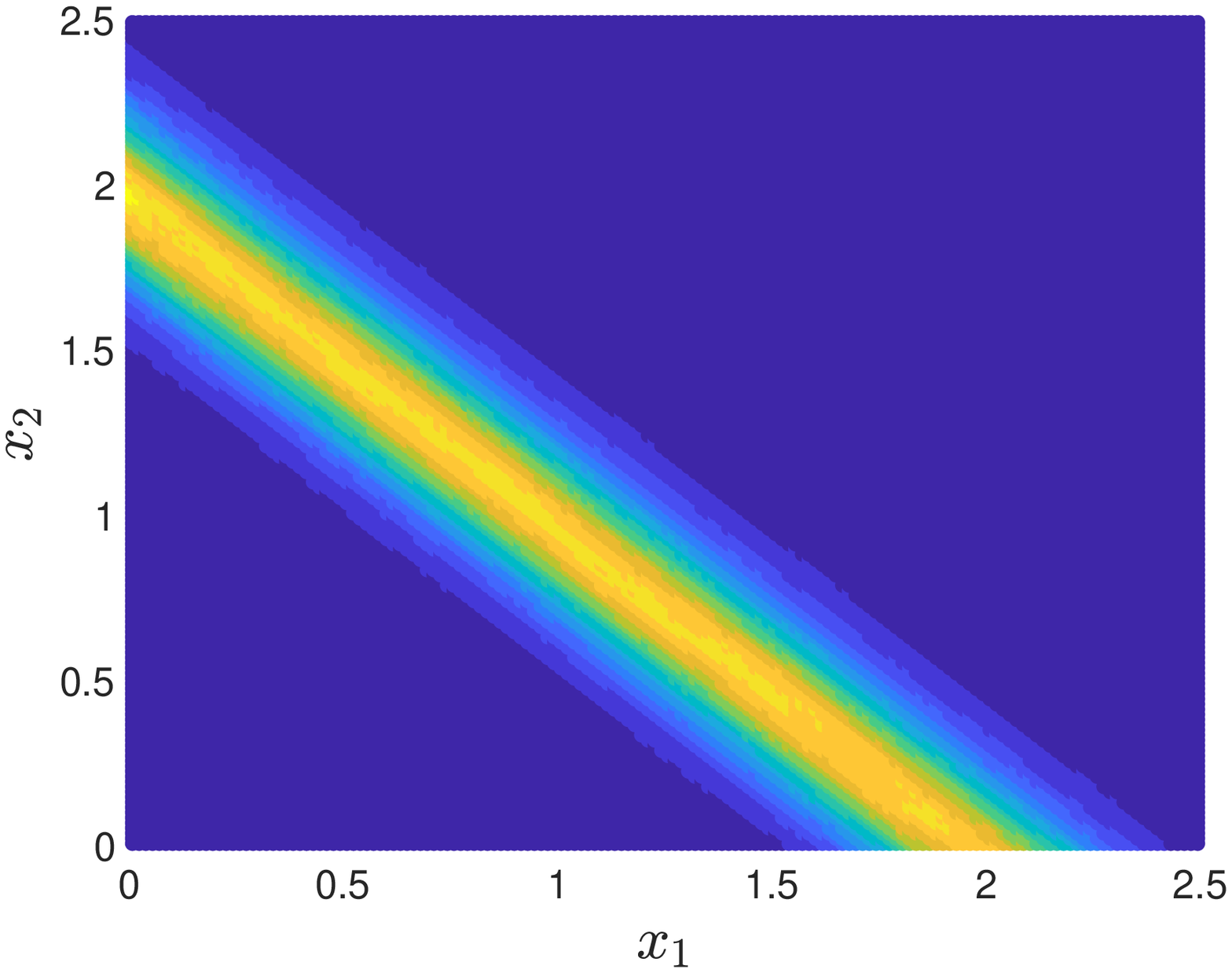}
    \end{subfigure}
    \begin{subfigure}{0.36\textwidth}
    \includegraphics[width=\linewidth, height=4cm]{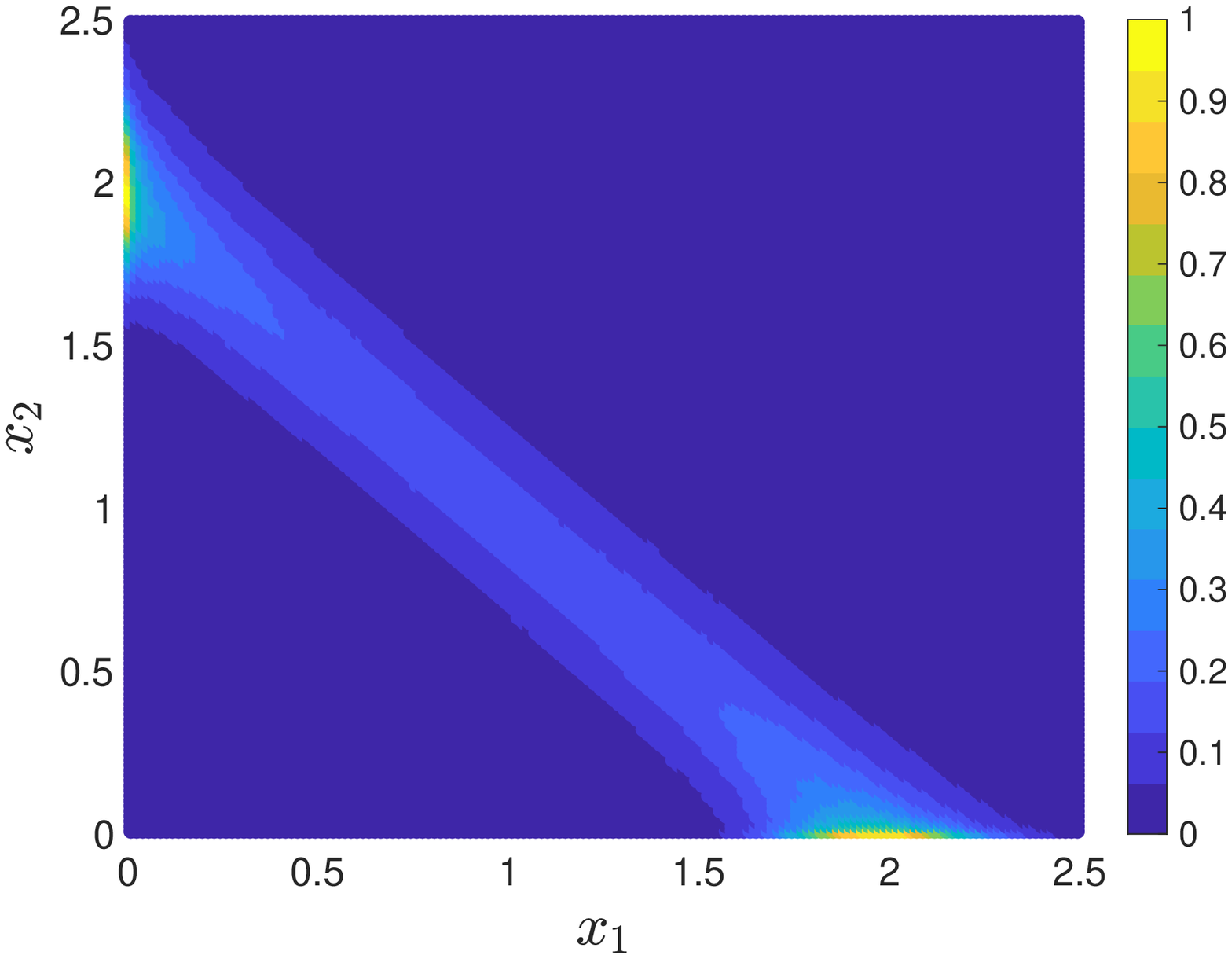}
    \end{subfigure}
    \caption{Stationary distribution of of CTMC (row 1) in \eqref{auto_CTMC}-\eqref{inout_CTMC} and CLA (row 2) in \eqref{Reflected_General} for $d=2$. All stationary quantities in Figure \ref{fig:stationary_2d} are obtained via time averaging over long time trajectory. The trajectory is simulated until $T=10^6$ with $V = 64, \kappa'= 1$, and $D = \lambda' = \delta'$ is given by $1/16, 1/32, 1/64$ for each column, from left to right respectively.
    }
    \label{fig:stationary_2d}
\end{figure}

In row 2 of Figure \ref{fig:stationary_2d}, stationary densities, obtained via time averaging over long-time trajectories of the CLA \eqref{Reflected_General}, are plotted for the same parameters as the associated CTMC. In all three cases, the CLA in \eqref{Reflected_General} accurately captures the stationary behavior of the CTMC in row 1 of Figure \ref{fig:stationary_2d}.

As $D<2/V$, finite trajectory of 2-dimensional TK model can be observed in Figure \ref{fig:traj_2D}, where the process spends most of its time on the boundary, i.e., one species is almost extinct. Hence switching time between boundaries is an important metric describing finite time dynamics of TK models. For simplicity, we will define \textbf{switching time} as the first time $X^2$ reaches 0, assuming the process starts initially with no $A_1$, i.e. $(X^1_0,X^2_0) = (0,2V)$. Note that finite time properties, including switching time distributions, can not be extracted by analyzing the stationary distributions.  

In an attempt to obtain explicit formula for the switching time, we consider a 1-dimensional approximation of  the 2-dimensional CLA. Roughly, we assume that the total mass evolves in a much slower time-scale than that of the differences between the two species, so that $Z^1_t + Z^2_t\approx Z^1_0 + Z^2_0 = n$ for a long time period (under the timescale for $Z^2$). Hence we can approximate $Z^2_t$ by $n-Z^1_t$ and the equation for $Z^1$ is given by
\begin{align}\label{eq:CLA1D_V_2}
    d S_t = S_0 + (\lambda' -\delta' S_t)dt + \frac{1}{\sqrt{V}}\sqrt{2\kappa' S_t(n - S_t) +\lambda' + \delta' S_t} dW'_t +m(S_t)dL_t, 
\end{align}
where $L_t$ is the local time of $S_t$ on the boundary of the interval $[0,n]$ and $m$ is the inward normal vector at the boundary of the interval (so $m(0)=1$ and $m(n)=-1$). Then for any interval $[I_-,I_+]$ and  an arbitrary subinterval $J \subset [I_-,I_+]$, we let $\tau^{J} =\inf\{t\geq 0:\,S_t \notin J\}$ be the first time process $S_t$ exits the subinterval $J$. The expectation $\E_x\left[\tau^{J}\right]$ solve the boundary value problem as in \cite{karlin1981second}, 
\begin{equation}
\label{HittingE_in}
\begin{cases}
    \mathcal{L}f(x)&=-1 \qquad\qquad \text{if}\quad x\in (I_-,I_+)\\
    f'(I_-)&=f(I_+)=0
\end{cases},
\end{equation}
where $\mathcal{L}$ is the generator of  \eqref{eq:CLA1D_V_2}, namely, suppose $f \in C^2([0,n])$, then 
\begin{align*}
    {\cal L} f = \frac{1}{2V}\left(2\kappa' x(n-x) + \lambda'+\delta' x\right) f'' +(\lambda' - \delta' x)f'.
\end{align*}



Hence the expected hitting time of $n$ can be expressed explicitly,
\begin{align}
\E_0\left[\tau^{[0,n)}\right] =  \int_0^n \frac{1}{{\cal I}(x)} \int_0^x \varphi(y) {\cal I}(y)dydx, 
\end{align}
where
\begin{align*}
    \phi(x) = \frac{2V\left(\lambda' - \delta' x\right)}{2\kappa' x(n-x) +\lambda' +\delta' x};\quad \varphi(x) =  \frac{2V}{2\kappa' x(n-x) +\lambda' +\delta' x}, \quad {\cal I}(x) = \exp\left\{ \int_0^x \phi(z)dz      \right\}.  
\end{align*}

In Figure \ref{fig:sensitivity_switch_all}, mean switching time is plotted against different parameters in \eqref{classicalScaling}, while fixing other parameters whose value can be found within the caption. Switching time is computed for CTMC \eqref{auto_CTMC}-\eqref{inout_CTMC} as well as CLA 2D in \eqref{Reflected_General}, and the one dimensional reflected approximation in \eqref{eq:CLA1D_V_2}, which are termed as CLA 1D in Figure \ref{fig:sensitivity_switch_all}.

As $V\rightarrow \infty$ in Figure \ref{fig:sensitivity_V_2d}, switching time of CTMC can be better approximated by CLA 2D in \eqref{Reflected_General}. Such trend can also be observed when $\kappa<1$ in figure \ref{fig:sensitivity_kappa_2d}, whereas CLA 1D in \ref{eq:CLA1D_V_2} become better approximations as $\kappa\rightarrow \infty$. Heuristically as $\kappa \rightarrow \infty$, dynamics of CTMC is dominated by autocatalytic reactions \eqref{auto}, which coincides with the motivation of CLA 1D in \eqref{eq:CLA1D_V_2} where the total mass evolves in a much slower time-scale.

\FloatBarrier

\begin{figure}[!ht]
\centering
\begin{subfigure}{0.45\textwidth}
\includegraphics[width=\linewidth, height=5cm]{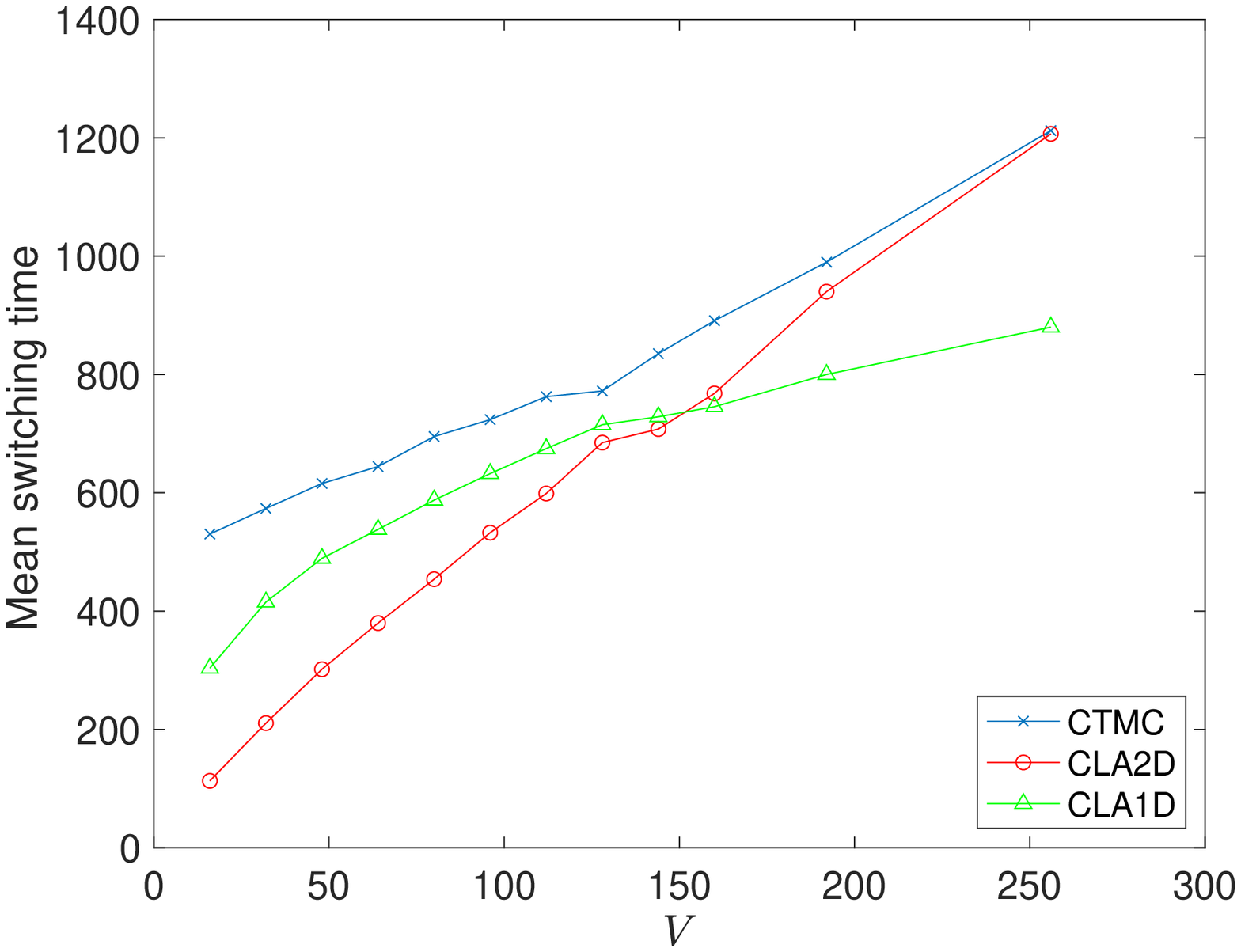}
\caption{$\kappa' = 1$ and $\lambda' = \delta' = 1/64$.}
\label{fig:sensitivity_V_2d}
\end{subfigure}
\begin{subfigure}{0.45\textwidth}
\includegraphics[width=\linewidth, height=5cm]{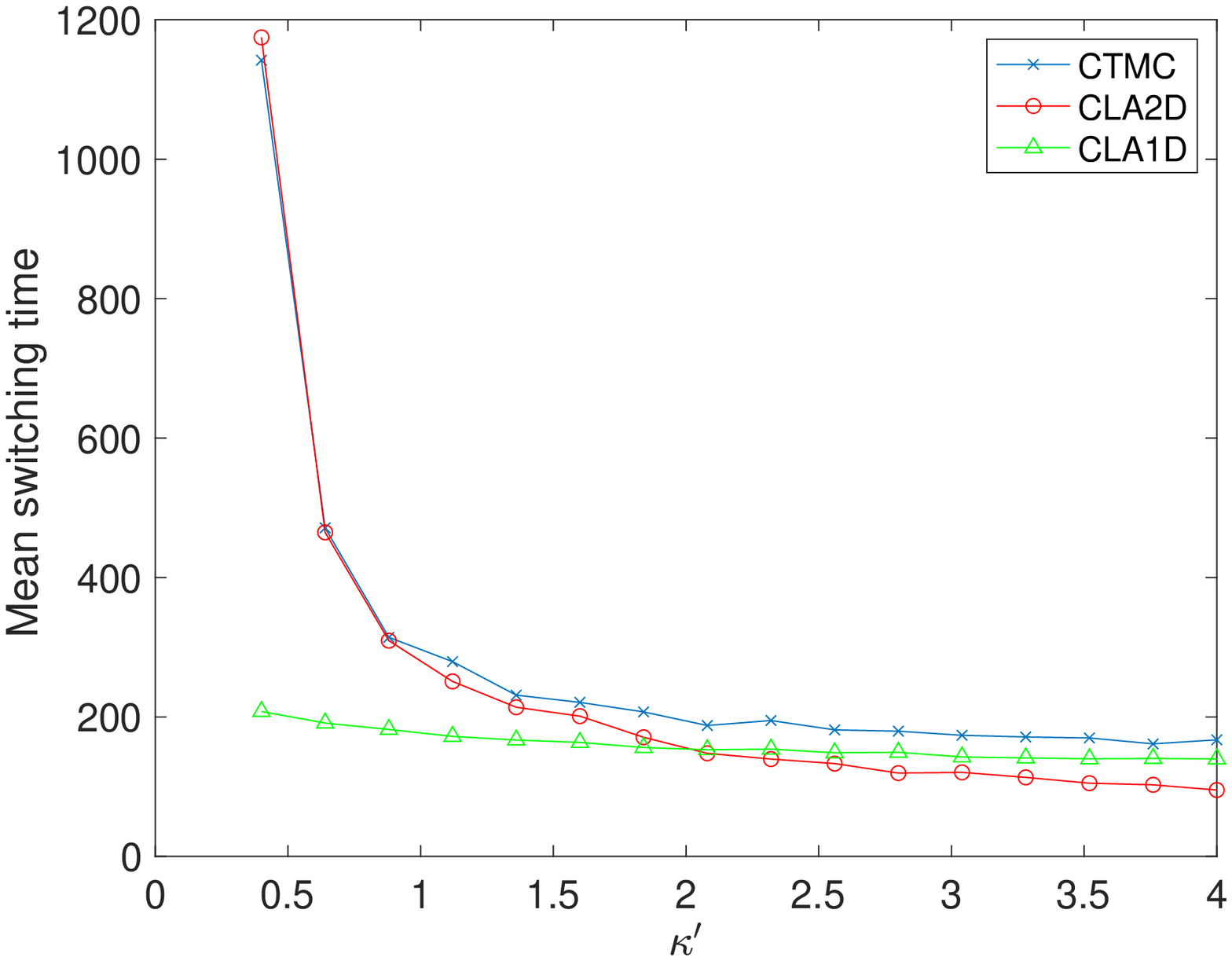}
\caption{$V = 64$ and $\lambda' = \delta' = 1/64$.}
\label{fig:sensitivity_kappa_2d}
\end{subfigure}

\begin{subfigure}{0.45\textwidth}
\includegraphics[width=\linewidth, height=5cm]{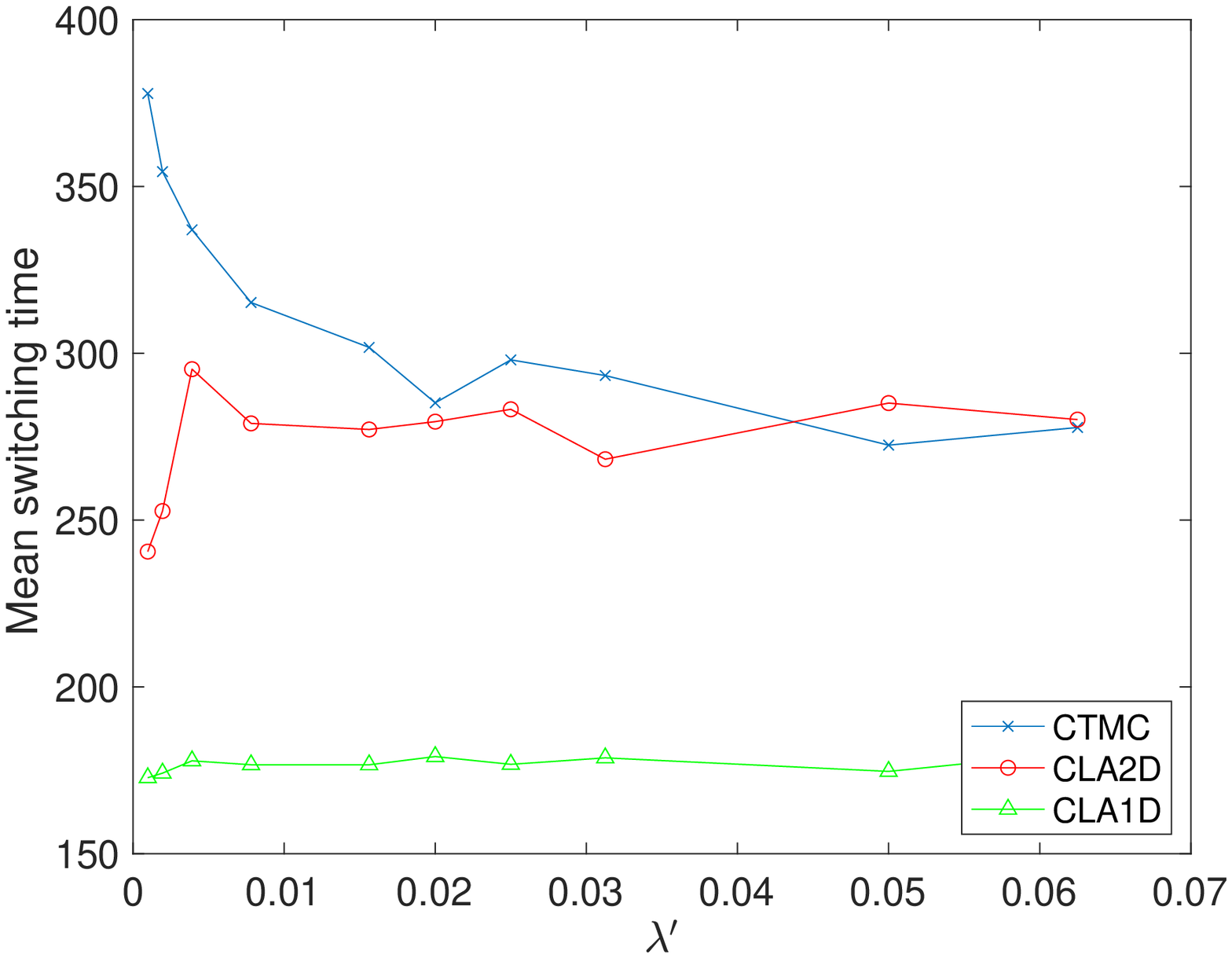}
\caption{$V = 64, \kappa' = 1$ and $\delta' = 1/64$.}
\label{fig:sensitivity_lambda}
\end{subfigure}
\begin{subfigure}{0.45\textwidth}
\includegraphics[width=\linewidth, height=5cm]{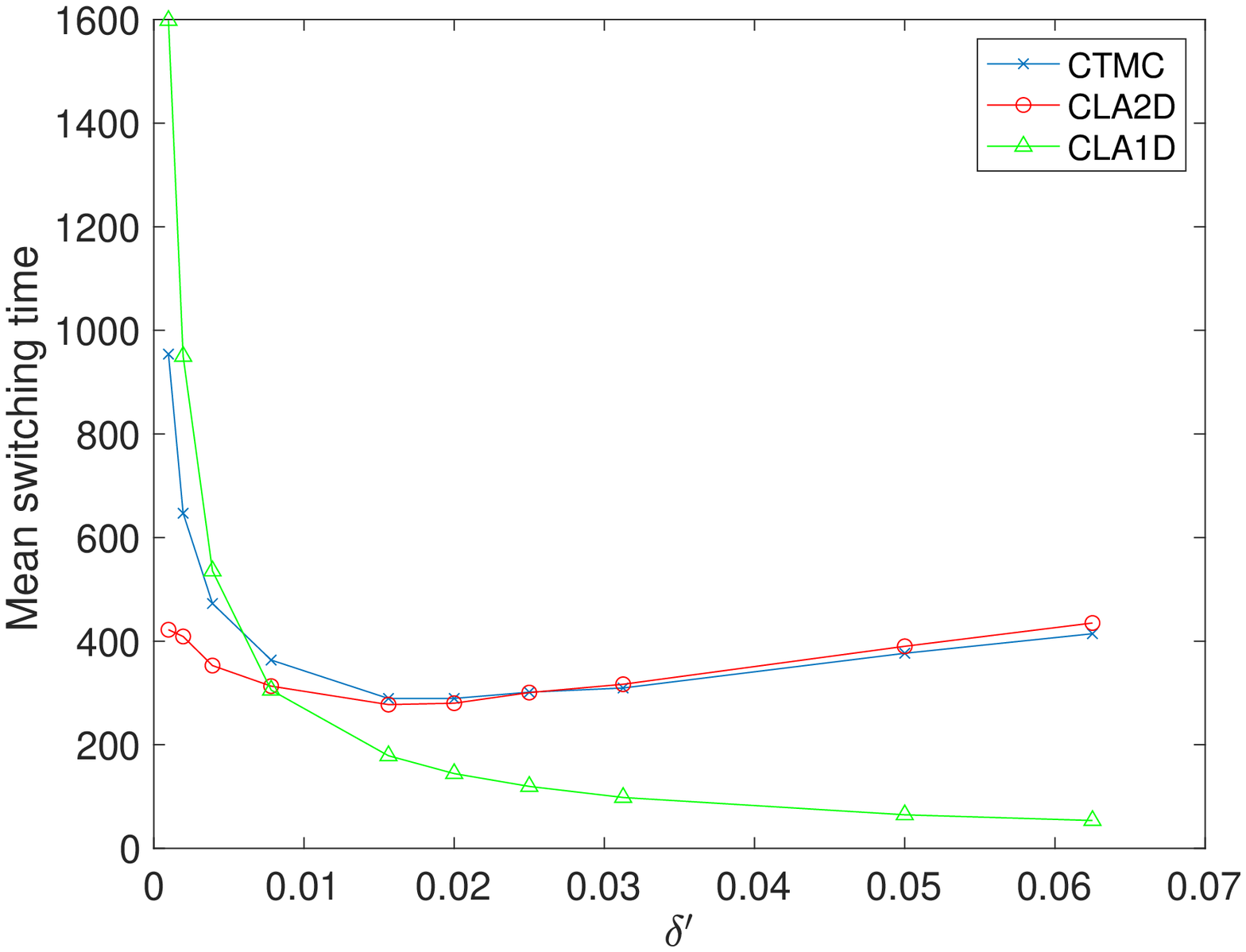}
\caption{$V = 64, \kappa' = 1$ and $\lambda' = 1/64$.}
\label{fig:sensitivity_delta}
\end{subfigure}
\caption{Mean switching time for $d=2$ is plotted against different choice of parameters in \eqref{classicalScaling} for CTMC (\eqref{auto_CTMC}-\eqref{inout_CTMC}), 2D CLA \eqref{Reflected_General} as well as 1D CLA \eqref{eq:CLA1D_V_2}. 
Throughout all simulations, the initial condition is chosen as $(X^1_0,X^2_0) = (0,2V)$. The \textbf{switching time} for each trajectory is then defined to as the first time $X^2$ becomes 0. Mean switching time is then computed via averaging over 1000 trajectories.
}
\label{fig:sensitivity_switch_all}
\end{figure}



\FloatBarrier

We compare the switching time distributions between the CTMC and the CLA via the histogram in Figure \ref{fig:hist_stoch_CLA}. Parameters of the simulation are given by  $V= 64, \kappa'= 1$ in \eqref{classicalScaling} for all four figures.  In addition, in Figures \ref{fig:2d_uniform_switch_stoch}-\ref{fig:2d_uniform_switch_CLA2d}, $\lambda' = \delta' = 1/32$, as the CTMC in \eqref{auto_CTMC}-\eqref{inout_CTMC} possess uniform distribution when conditioned on the level sets; in Figures \ref{fig:2d_bimodal_switch_stoch}-\ref{fig:2d_bimodal_switch_CLA2d},   $\lambda' = \delta' = 1/64$, as the CTMC in \eqref{auto_CTMC}-\eqref{inout_CTMC} possess bimodal stationary distribution as mass is concentrated on the meta-stable states near the boundary. In both cases, mean and variance of switching time are quite close, and the shape of the distributions is nicely recovered. 

\begin{figure}[!htbp]
\centering
\begin{subfigure}{0.45\textwidth}
\includegraphics[width=\linewidth, height=5cm]{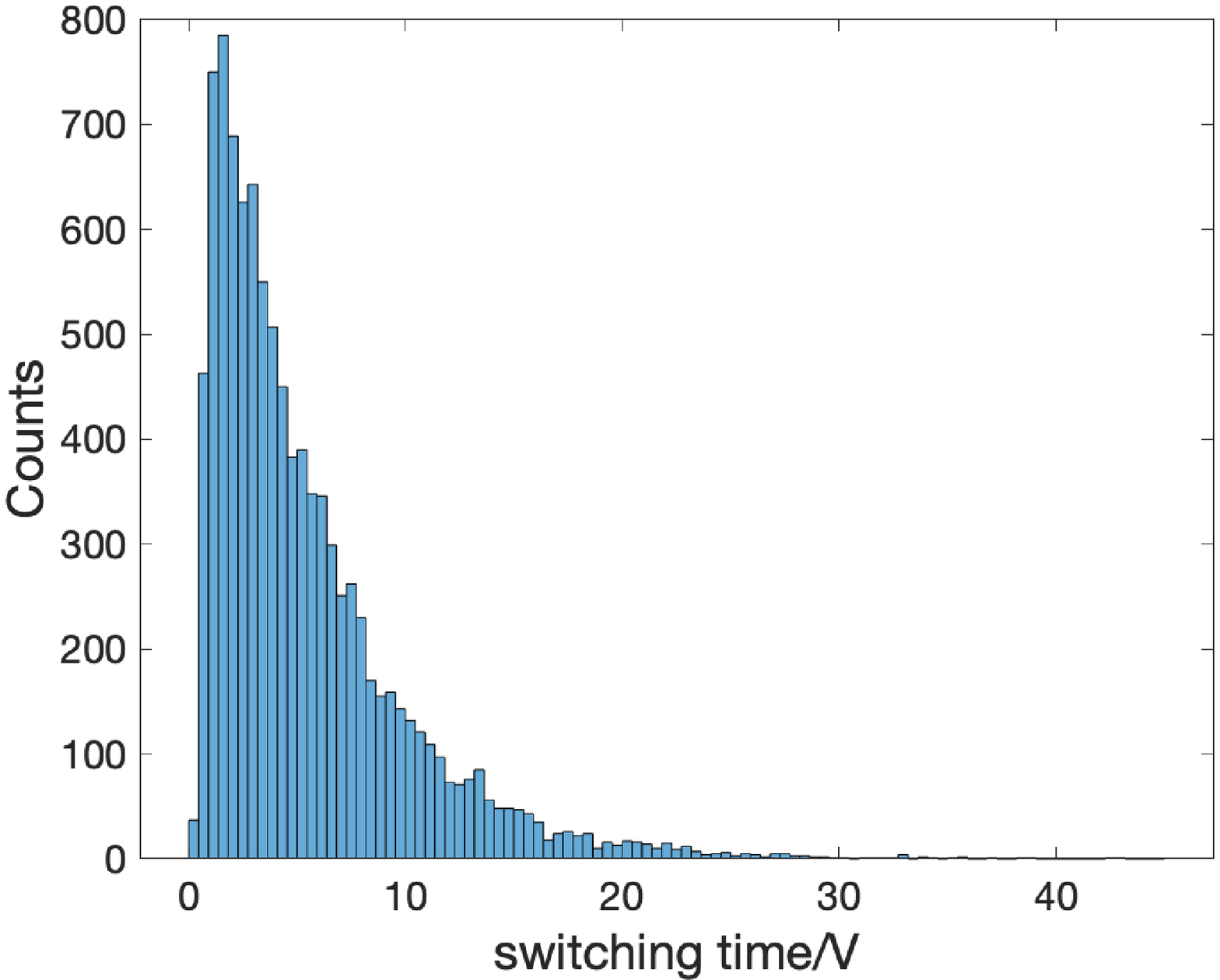}
\caption{Mean = 5.42, variance = 21.54.}
\label{fig:2d_uniform_switch_stoch}
\end{subfigure}
\begin{subfigure}{0.45\textwidth}
\includegraphics[width=\linewidth, height=5cm]{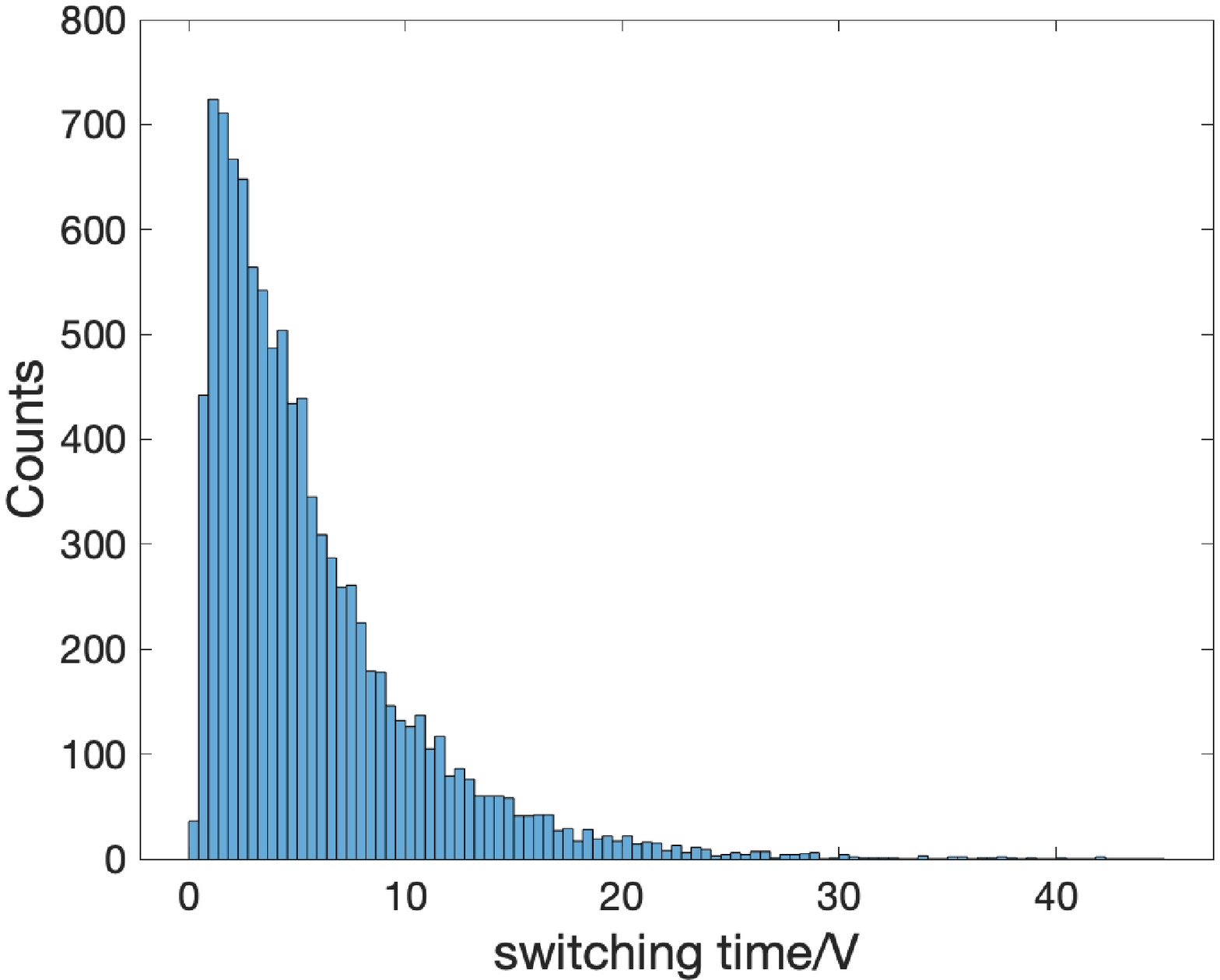}
\caption{Mean = 5.62, variance = 23.33.}
\label{fig:2d_uniform_switch_CLA2d}
\end{subfigure}

\begin{subfigure}{0.45\textwidth}
\includegraphics[width=\linewidth, height=5cm]{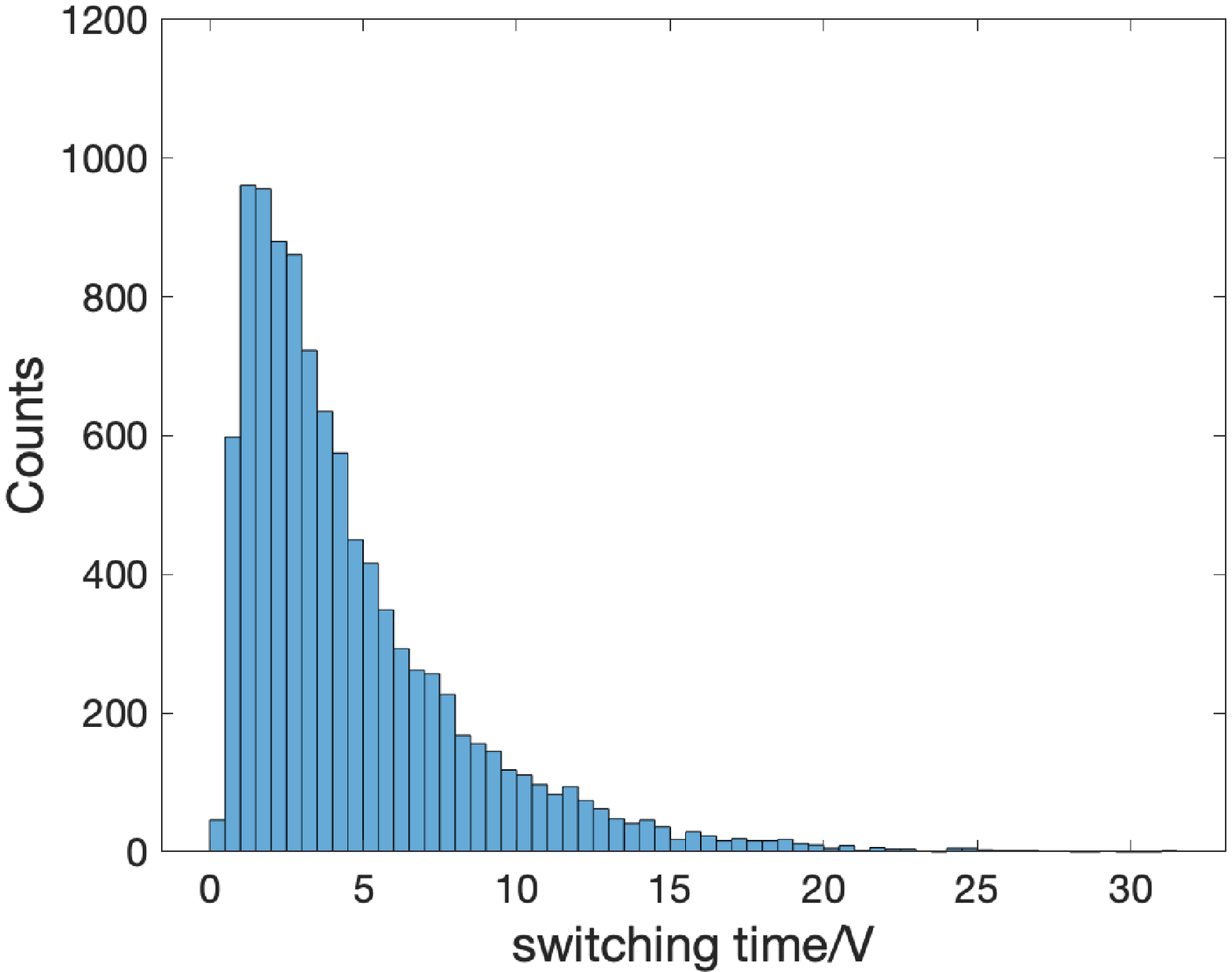}
\caption{Mean = 4.63, variance = 14.31.}
\label{fig:2d_bimodal_switch_stoch}
\end{subfigure}
\begin{subfigure}{0.45\textwidth}
\includegraphics[width=\linewidth, height=5cm]{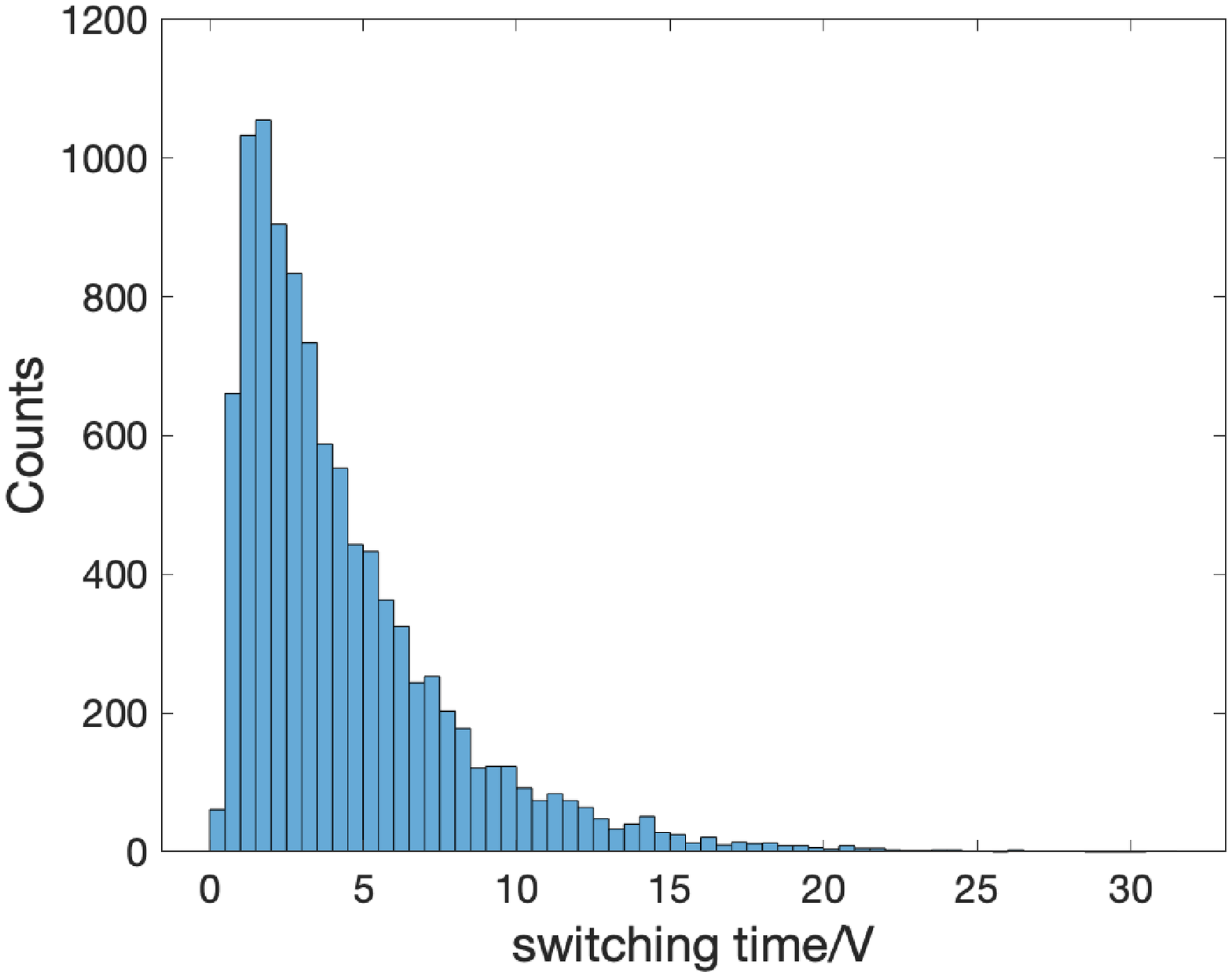}
\caption{Mean = 4.37, variance = 12.76.}
\label{fig:2d_bimodal_switch_CLA2d}
\end{subfigure}
\caption{Histogram of switching times for the 2D CTMC \eqref{auto_CTMC}-\eqref{inout_CTMC} and the 2D CLA \eqref{Reflected_General} are plotted, in left and right column respectively, for 1000 trajectories. Parameters of the simulation are given by  $V= 64,  \kappa'= 1$ in \eqref{classicalScaling} for all four figures.  In addition, in Figures \ref{fig:2d_uniform_switch_stoch}-\ref{fig:2d_uniform_switch_CLA2d}, $\lambda' = \delta' = 1/32$;  in Figures \ref{fig:2d_bimodal_switch_stoch}-\ref{fig:2d_bimodal_switch_CLA2d},   $\lambda' = \delta' = 1/64$. 
}
\label{fig:hist_stoch_CLA}
\end{figure}

However, as $\lambda'$ or $\delta'\rightarrow 0$, the switching time estimates using CLA 2D \eqref{Reflected_General} are no longer close to the CTMC, as shown in Figure \ref{fig:sensitivity_lambda} and \ref{fig:sensitivity_delta}. In figure \ref{fig:2d_extbimodal_switch}, such cases were investigated for $V= 64, \lambda' = \delta' = 1/256, \kappa'= 1$, where switching time distribution of CTMC \eqref{auto_CTMC}-\eqref{inout_CTMC} is approximated by the CLA 1D \eqref{eq:CLA1D_V_2}. This approximation yields much better result comparing to CLA 2D in \eqref{Reflected_General}, however, the approximation via  CLA 1D \eqref{eq:CLA1D_V_2} still underestimates the mean and variance of switching time for the exact CTMC model.

\begin{figure}
\centering
\begin{subfigure}{0.45\textwidth}
\includegraphics[width=\linewidth, height=5cm]{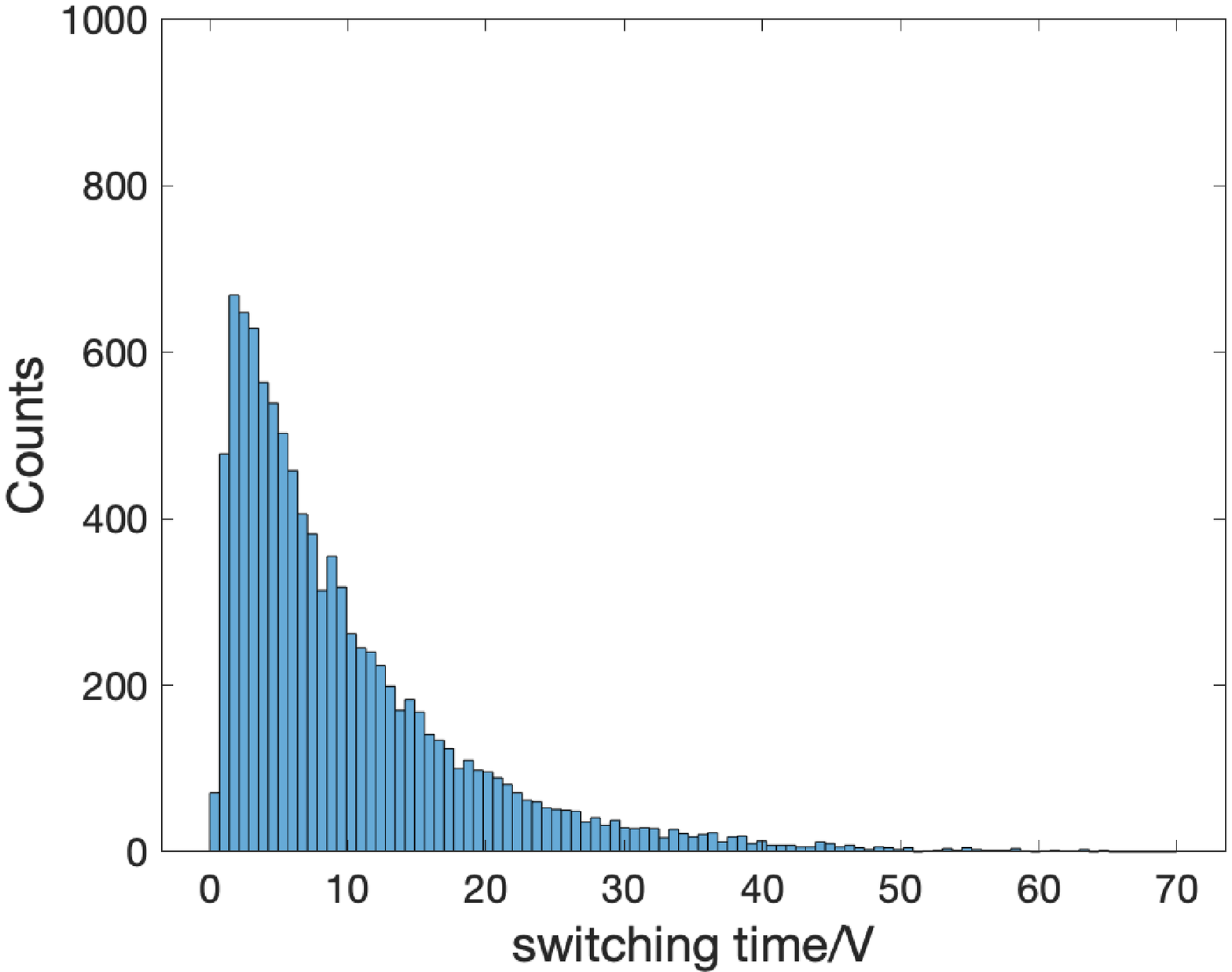}
\caption{Mean = 10.05, variance = 86.79.}
\label{fig:2d_extbimodal_switch_stoch}
\end{subfigure}
\begin{subfigure}{0.45\textwidth}
\includegraphics[width=\linewidth, height=5cm]{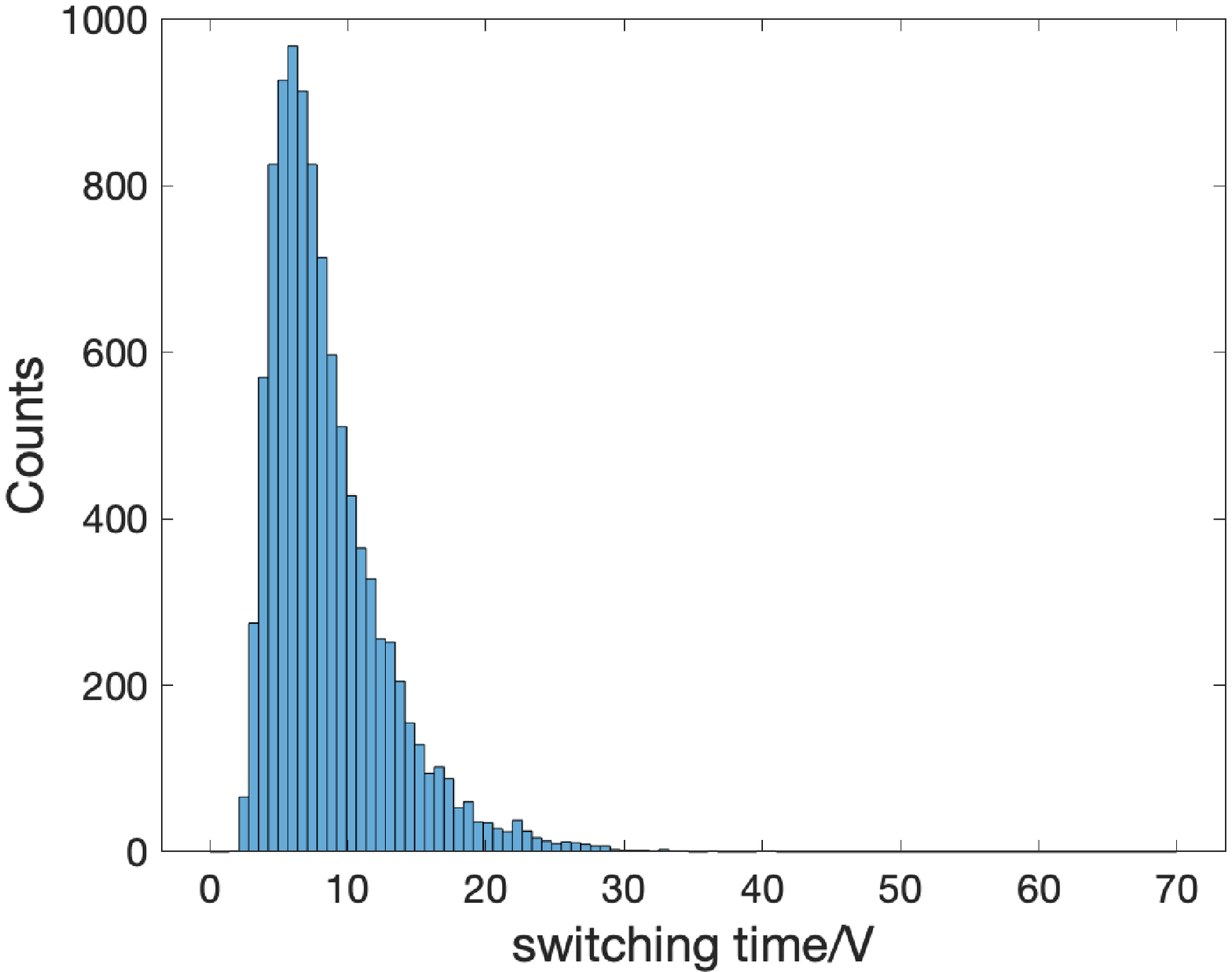}
\caption{Mean = 8.52, variance = 18.39.}
\label{fig:2d_extbimodal_switch_CLA1D}
\end{subfigure}
\caption{Histogram of switching times for the 2D CTMC \eqref{auto_CTMC}-\eqref{inout_CTMC} and the CLA 1D \eqref{eq:CLA1D_V_2} are plotted, in left and right column respectively, for 1000 trajectories. Parameters for the simulation are given by $V= 64, \lambda' = \delta' = 1/256, \kappa'= 1$ in \eqref{classicalScaling}. 
}
\label{fig:2d_extbimodal_switch}
\end{figure}

\FloatBarrier

\subsection{Simulation result for the 3-dimensional TK model }
\label{sec:3DTK}

In this section, we first give an explanation on discrepancy between finite time dynamics of 2-dimensional and 3-dimensional TK models. Then we show via simulation that CLA in \eqref{Reflected_General} recovers the stationary distribution of the associated CTMC in \eqref{auto_CTMC}-\eqref{inout_CTMC}. Furthermore we propose cycling time, which is analogous to switching time when $d=2$, and the distribution of cycling time of the CTMC can be well approximated via its associated CLA when $V$ is large. 

The main contrast between $d=2$ and $d=3$ for the CTMC, as observed in Figure \ref{fig:finite_traj}, is that the transition time between meta-stable states for $d=2$ is much longer than that for $d=3$. A heuristic explanation to this is because for 2-dimensional TK model, the two autocatalytic reactions \eqref{auto} move in opposite directions while sharing the same intensity, hence making it difficult for the process to move between meta-stable states; whereas for higher dimensional TK models ($d\geq 3$), autocatalytic reactions move the process in a cyclic direction, namely $A_1\rightarrow A_2 \rightarrow \cdots \rightarrow A_d \rightarrow A_1$. More specifically when $A_1$ is abundant, dynamics of $A_1$ is mainly driven by autocatalytic reactions $A_d+A_1\rightarrow 2A_1$ and  $A_1+A_2\rightarrow 2A_2$. However firing of both reactions tend to decrease the rate of $A_d+A_1\rightarrow 2A_1$ and increase the rate of $A_1+A_2\rightarrow 2A_2$, leading to an imbalance toward gaining of $A_2$, hence most $A_1$ will be changed into $A_2$.

Similar to 2-dimensional TK models, positive recurrence are established for $d=3$ in \cite{bibbona2020stationary}, however explicit form of stationary distribution is only derived when $\delta = \frac{3}{2}\kappa$. In figure \ref{fig:stationary_3d}, stationary distribution of the CTMC in \eqref{auto_CTMC}-\eqref{inout_CTMC} is plotted on the hyperplane $\{ x+y+z=3V\}$ for $\delta > \frac{3}{2}\kappa$, $\delta = \frac{3}{2}\kappa$ and $\delta < \frac{3}{2}\kappa$ in row 1. Similar multi-modality of stationary distributions are observed when $\delta < \frac{3}{2}\kappa$. In row 2 of figure \ref{fig:stationary_3d}, stationary distribution of the associated CLA in \eqref{Reflected_General} is plotted via densities near the hyperplane, namely $\left\{ |x+y+z-3|\leq \frac{1}{128}\right\}$. In all three cases, the CLA captures the stationary behavior of the exact CTMC model. 

\begin{figure}[!ht]
\centering
\begin{subfigure}{0.3\textwidth}
\includegraphics[width=\linewidth, height=3.5cm]{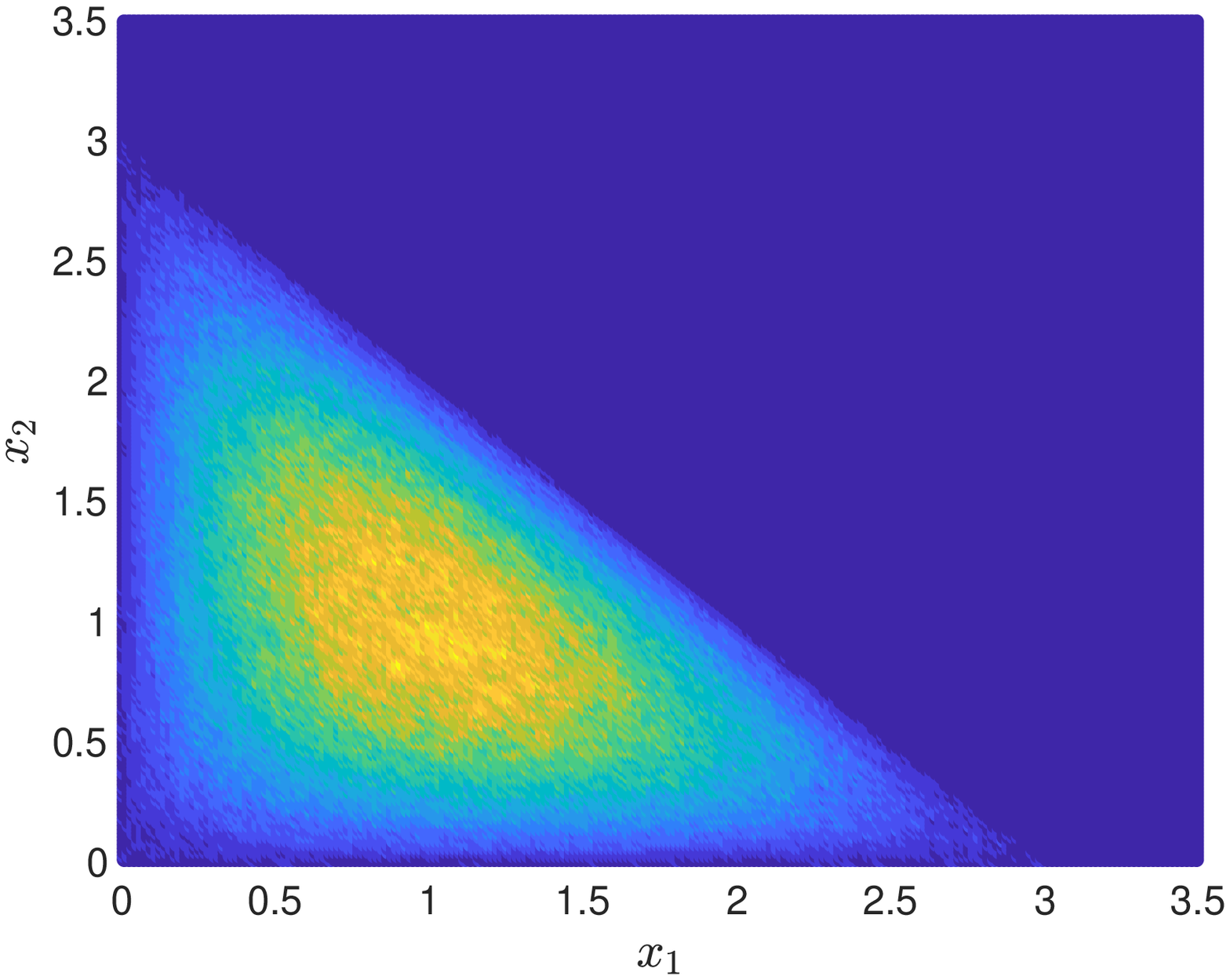}
\end{subfigure}
\begin{subfigure}{0.3\textwidth}
\includegraphics[width=\linewidth, height=3.5cm]{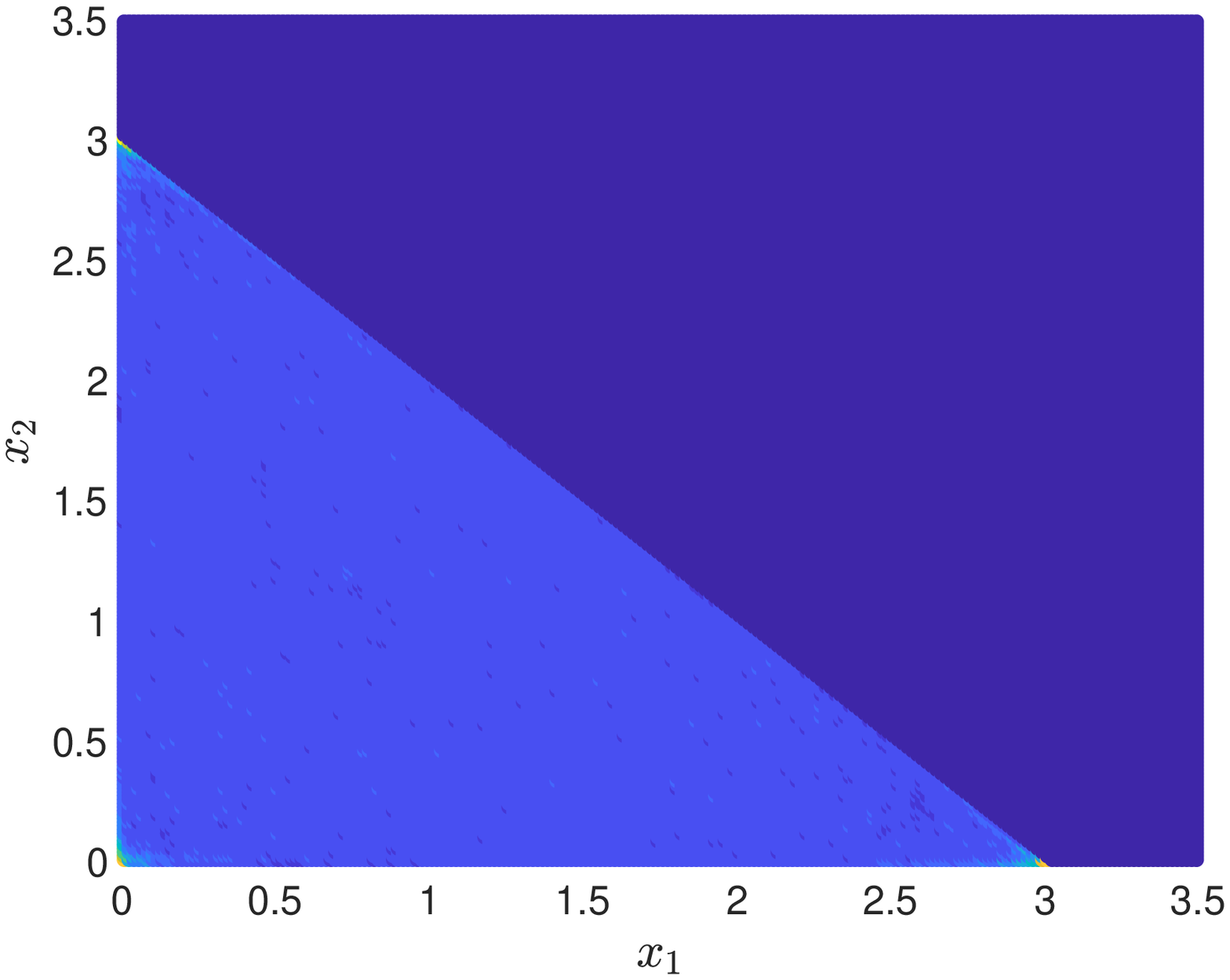}
\end{subfigure}
\begin{subfigure}{0.3\textwidth}
\includegraphics[width=\linewidth, height=3.5cm]{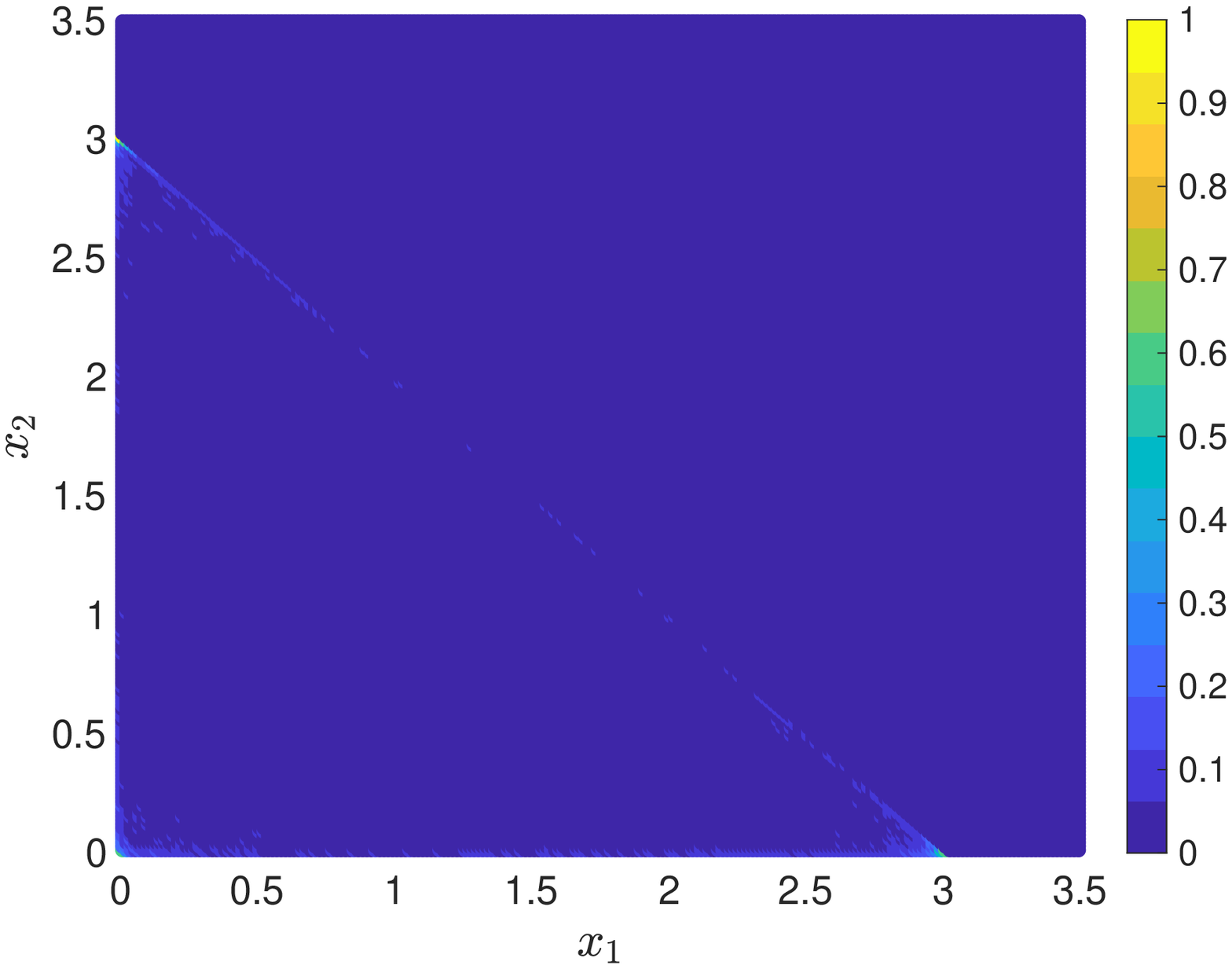}
\end{subfigure}

\begin{subfigure}{0.3\textwidth}
\includegraphics[width=\linewidth, height=3.5cm]{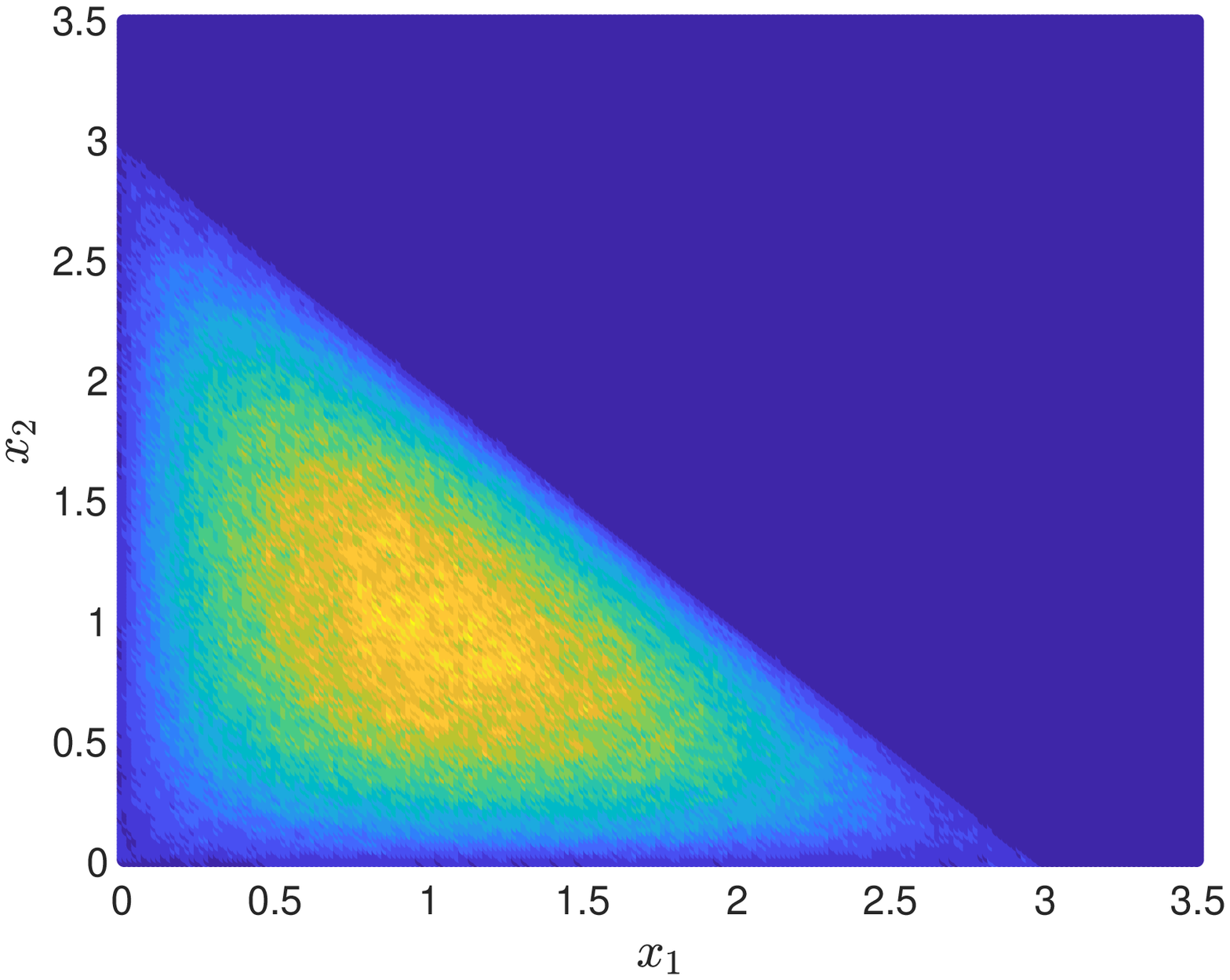}
\end{subfigure}
\begin{subfigure}{0.3\textwidth}
\includegraphics[width=\linewidth, height=3.5cm]{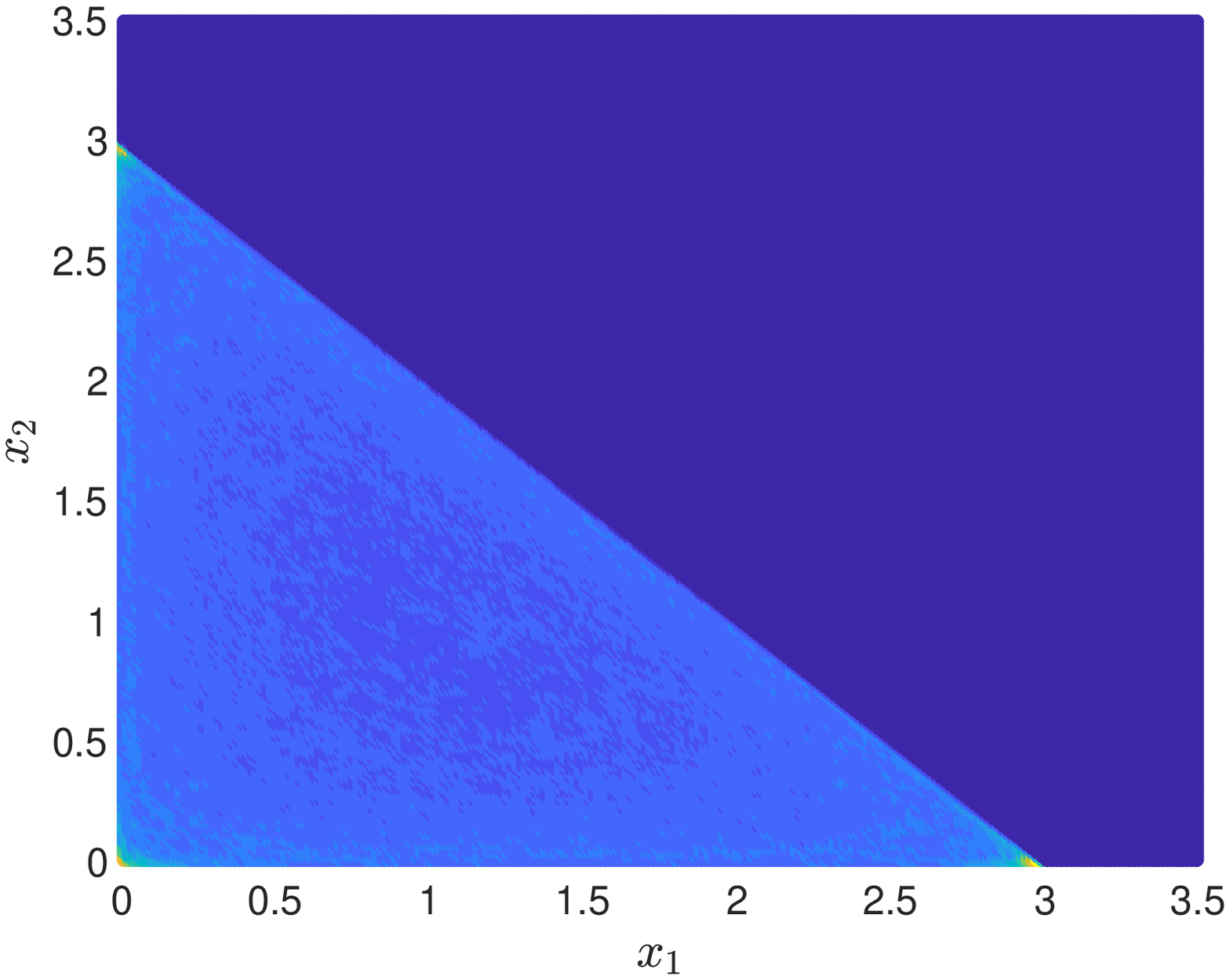}
\end{subfigure}
\begin{subfigure}{0.3\textwidth}
\includegraphics[width=\linewidth, height=3.5cm]{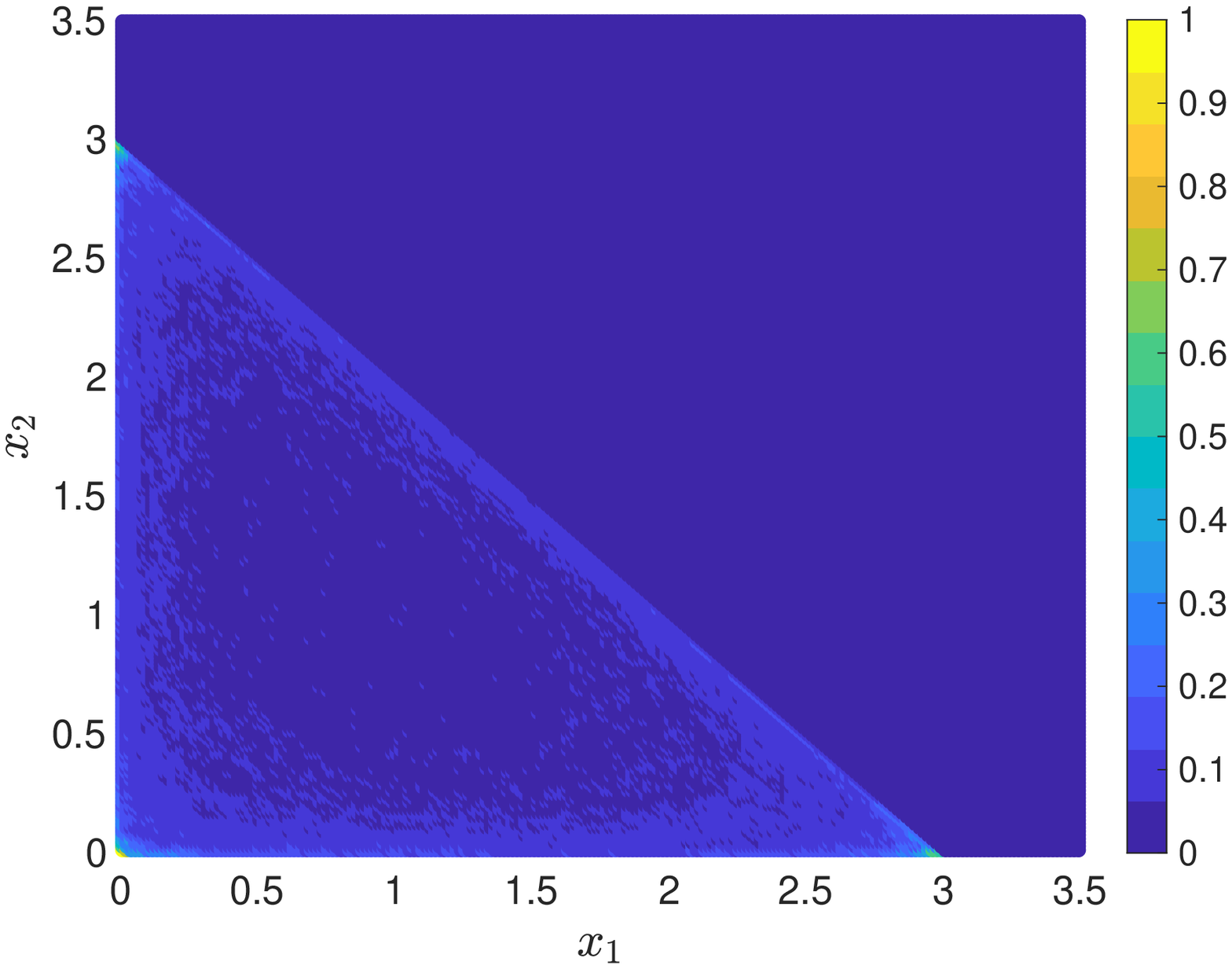}
\end{subfigure}
\caption{Stationary distribution of the CTMC in \eqref{auto_CTMC}-\eqref{inout_CTMC} conditioning on $\{x+y+z = 3V\}$, and CLA in \eqref{Reflected_General} conditioning on $\{|x+y+z-3|\leq 1/128\}$ with dimension $d=3$. All stationary quantities in Figure \ref{fig:stationary_3d} are obtained via time-averaging over long time trajectory. The trajectory is simulated until $T = 10^6$ with parameter $V = 64, \kappa'= 1$, and the value of $D$ is given by $3/64, 3/128, 1/64$ from left to right respectively. }
\label{fig:stationary_3d}
\end{figure}

As discussed in the different transition time between meta-stable states for $d=2$ and $d=3$, the dominant species when $d=3$ form a cycle, $A_1\rightarrow A_2\rightarrow A_3\rightarrow A_1$. Note that such cyclic behavior is not exclusive to parameter regimes when stationary distributions are concentrated near the the boundaries. In figure \ref{fig:3d_finitetraj_nontri}, trajectories of 3-dimensional TK model is plotted for $V = 256, \kappa'= 1$, and $D = \lambda' = \delta'$ is given by 1/32 and 3/512 respectively. The stationary distribution is unimodal or uniform when conditioning on the hyperplane $\{x+y+z = 3V\}$ respectively. In both cases, cyclic behavior still persists within the trajectory, hence we propose \textbf{cycling time} as the analogous quantity in describing finite time dynamics for $d\geq 3$, which can be defined as the second time $X^3$ reaches peak abundance when initial condition is given by  $(X^1_0,X^2_0,X^3_0) = (0,0,3V)$. Note that the switching time is also well defined in 3-dimensional TK models, however the switching could end up in different regions, i.e. if initially there is only species $A_3$, it could move to the boundary with only species $A_1$ as well as the boundary with only species $A_2$. In comparison cycling time is more straightforward and it consistently captures the average behavior of 3-dimensional TK model.

\begin{figure}[!ht]
\centering
\includegraphics[width=0.6\linewidth, height=6cm]{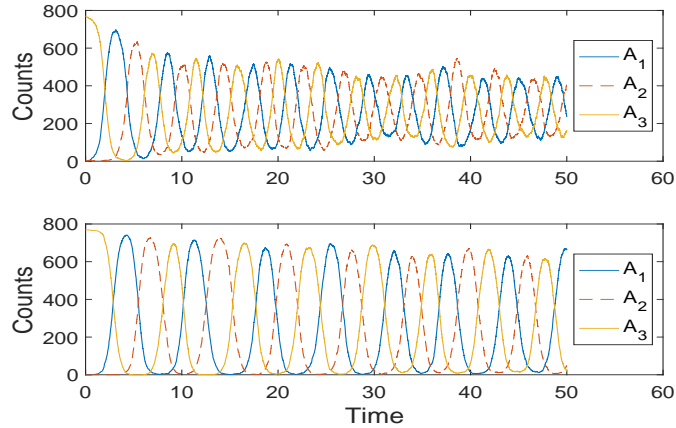}
\caption{Finite time trajectory of 3-dimensional TK model is plotted for $V = 256, \kappa'= 1$ and $D = \lambda' = \delta'$ is given by 1/32 and 3/512 respectively.}
\label{fig:3d_finitetraj_nontri}
\end{figure}

We will simulate and compare cycling time distribution for the CTMC in \eqref{auto_CTMC}-\eqref{inout_CTMC} and its associated CLA in \eqref{Reflected_General}. In figure \ref{fig:sensitivity_switch_all_3d}, mean switching time is plotted against different parameters in \eqref{classicalScaling}, while fixing other parameters whose value can be found within the caption. Cycling time is computed for CTMC as well as CLA 3D in \eqref{Reflected_General}.  Similar to 2D results, mean cycling time is better captured when $V\rightarrow \infty$, or when $\kappa\rightarrow 0$. 

\begin{figure}[!ht]
\centering
\begin{subfigure}{0.45\textwidth}
\includegraphics[width=\linewidth, height=5cm]{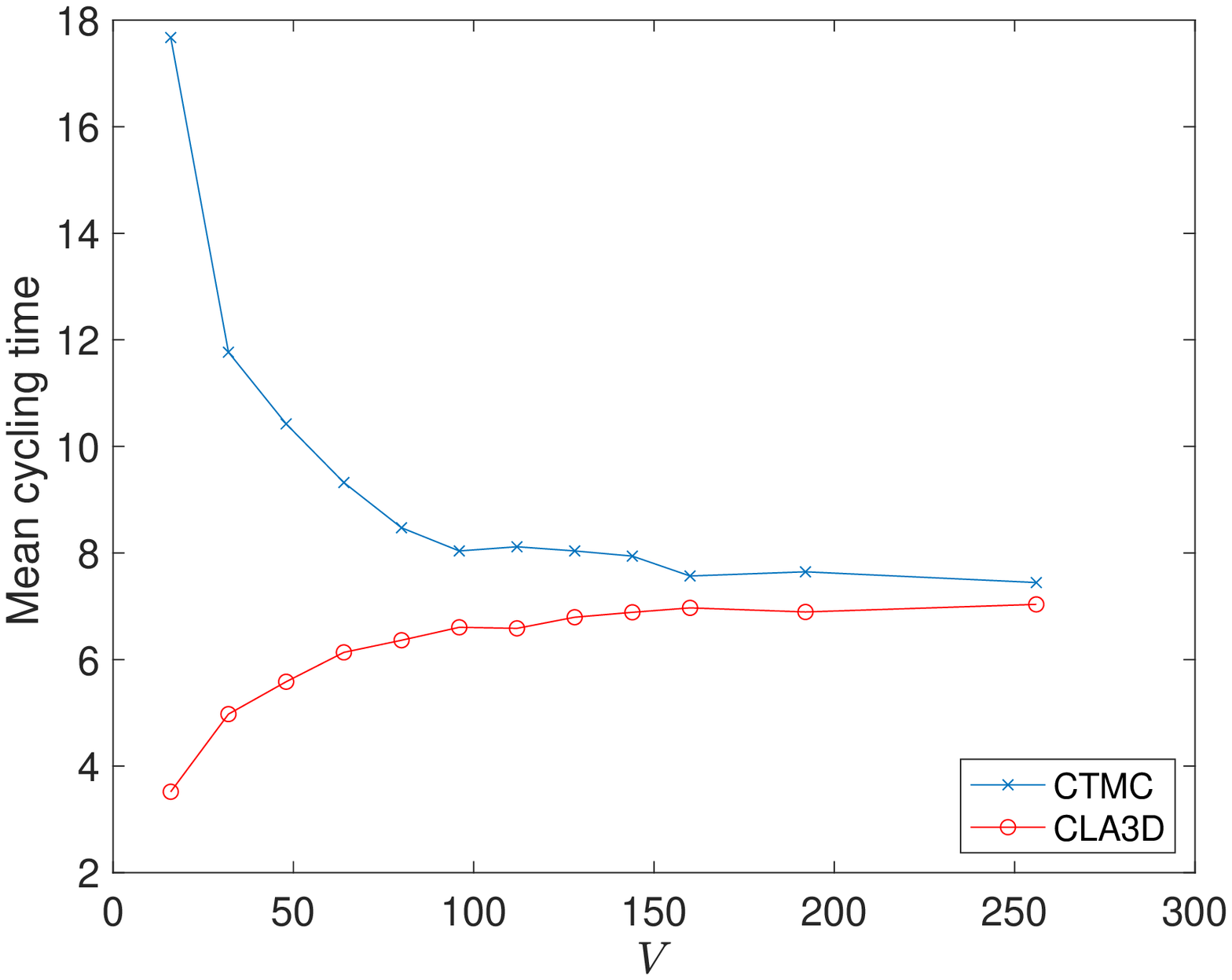}
\caption{$\kappa' = 1$ and $\lambda' = \delta' = 1/64$.}
\label{fig:sensitivity_V_3d}
\end{subfigure}
\begin{subfigure}{0.45\textwidth}
\includegraphics[width=\linewidth, height=5cm]{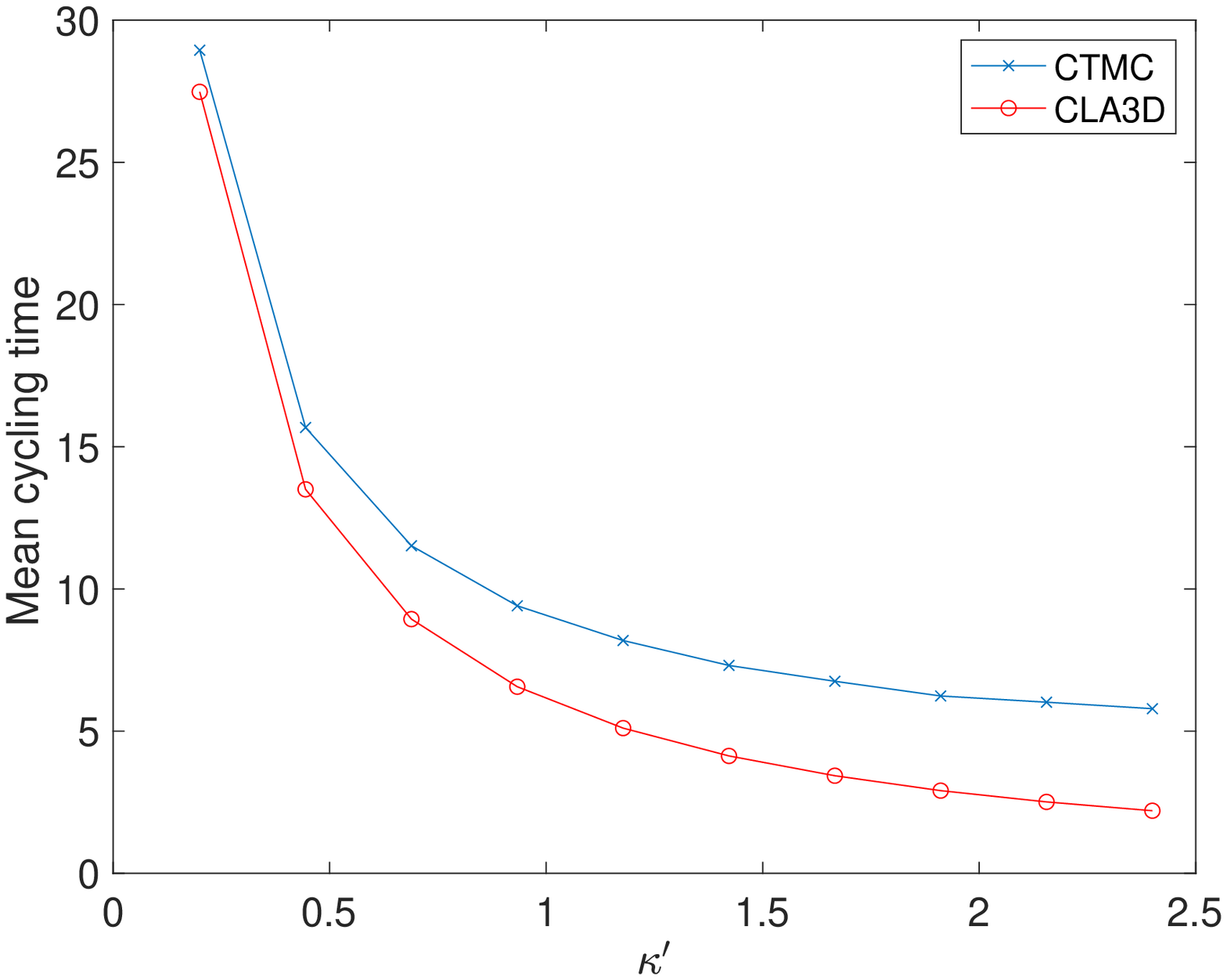}
\caption{$V = 64$ and $\lambda' = \delta' = 1/64$.}
\label{fig:sensitivity_kappa_3d}
\end{subfigure}

\begin{subfigure}{0.45\textwidth}
\includegraphics[width=\linewidth, height=5cm]{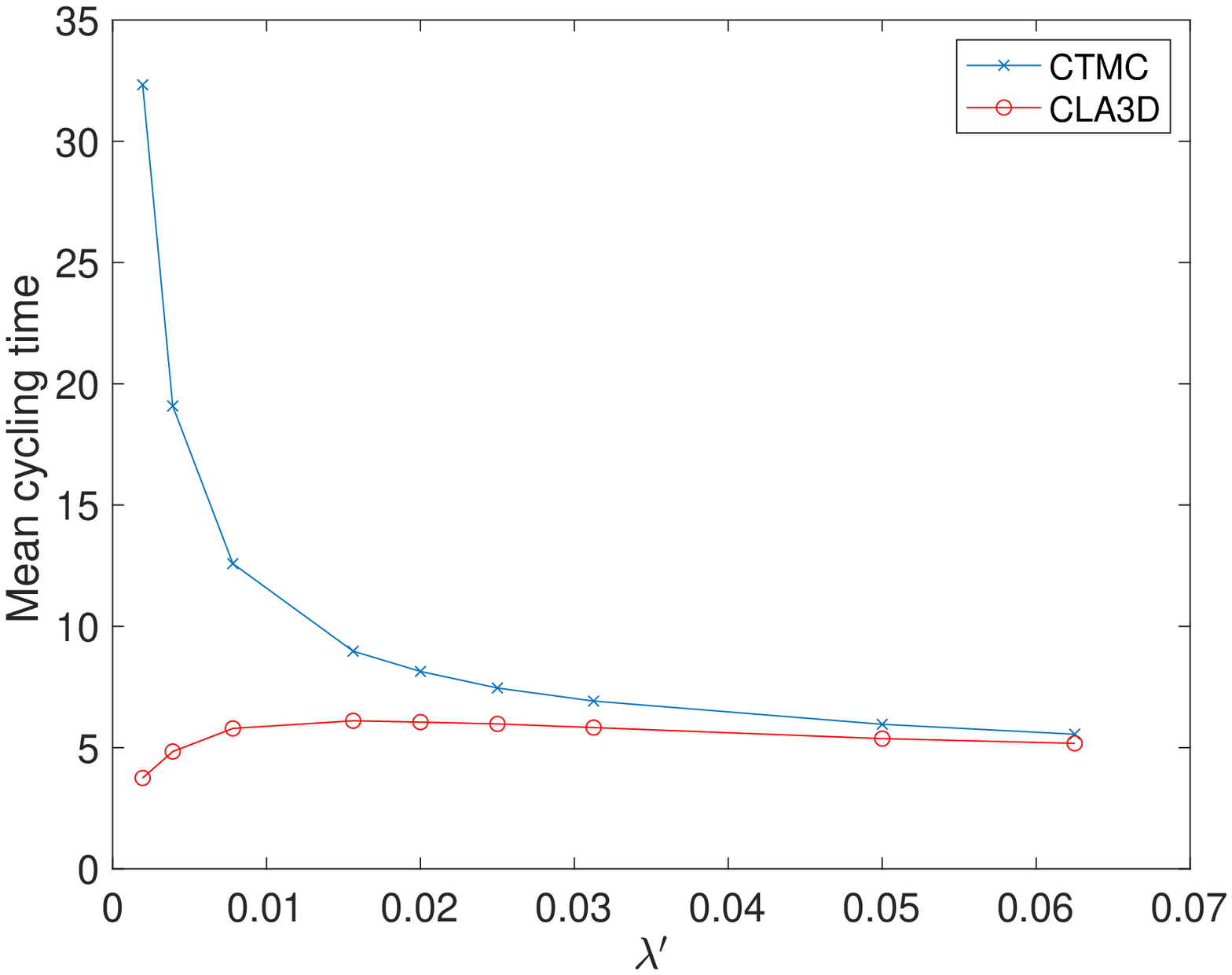}
\caption{$V = 64, \kappa' = 1$ and $\delta' = 1/64$.}
\label{fig:sensitivity_lambda_3d}
\end{subfigure}
\begin{subfigure}{0.45\textwidth}
\includegraphics[width=\linewidth, height=5cm]{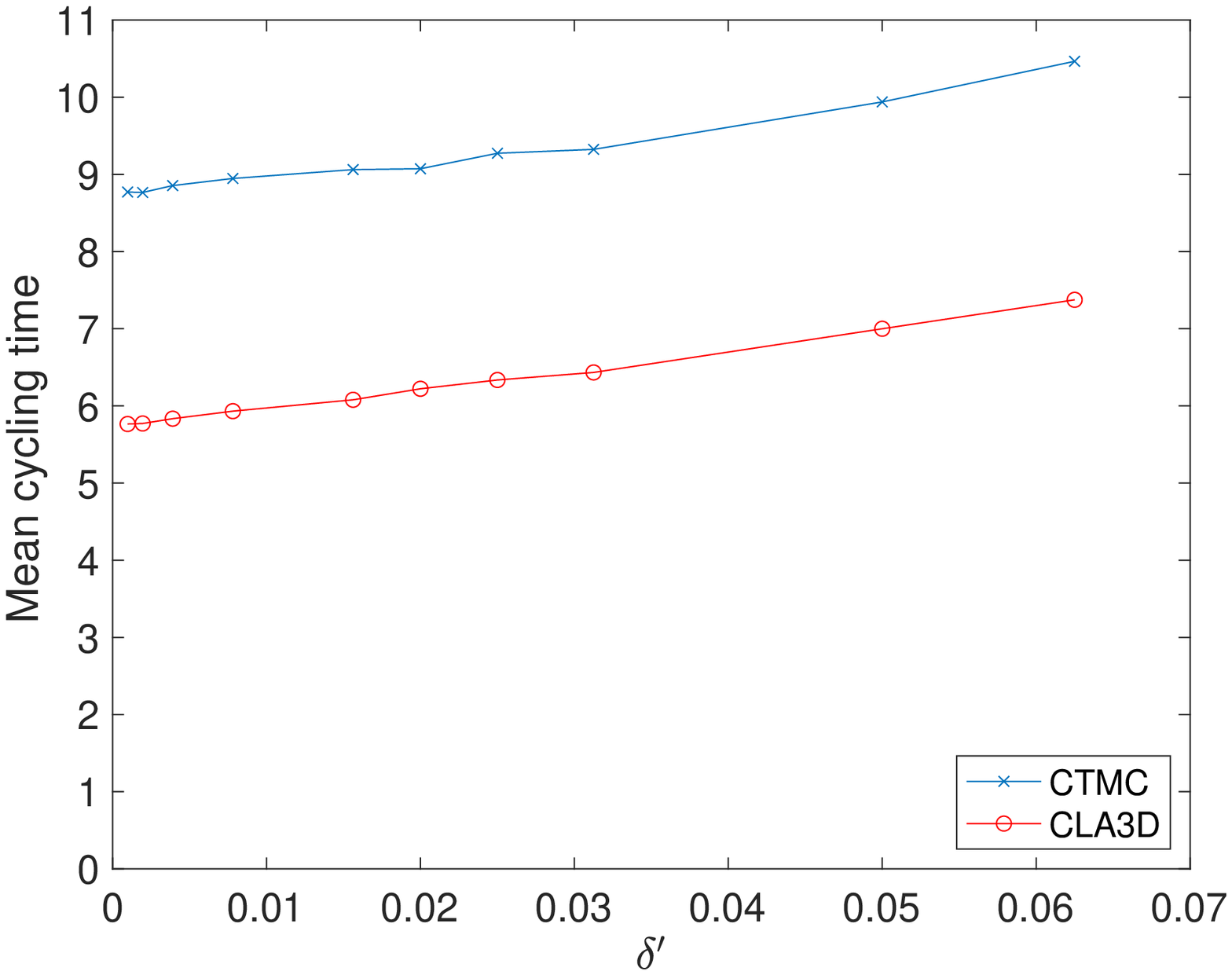}
\caption{$V = 64, \kappa' = 1$ and $\lambda' = 1/64$.}
\label{fig:sensitivity_delta_3d}
\end{subfigure}
\caption{Mean cycling time is plotted for the CTMC \eqref{auto_CTMC}-\eqref{inout_CTMC} and the CLA \eqref{Reflected_General} in dimension $d=3$ are plotted against different choice of parameters. 
Throughout all simulations, the initial condition is chosen as $(X^1_0,X^2_0,X^3_0) = (0,0,3V)$. The cycling time for each trajectory is then obtained as the first time $X^3$ reaches peak abundance. Mean switching time is then computed via averaging over 1000 trajectories. 
}
\label{fig:sensitivity_switch_all_3d}
\end{figure}

In figure \ref{fig:3d_hist_switch}, cycling time distribution for the CTMC is compared with its associated CLA \eqref{Reflected_General}. Simulation of cycling time in figure \ref{fig:3d_unimodal_switch_stoch}-\ref{fig:3d_unimodal_switch_CLA3d} is obtained using $V = 256,\kappa' = 1, \lambda' = \delta' = 1/32$, which corresponds to a unimodal stationary distribution. In this case the cycling time distribution of CTMC,  as well as its mean and variance, is nicely recovered by its associated CLA. 

In figure \ref{fig:3d_uniform_switch_stoch}-\ref{fig:3d_uniform_switch_CLA3d}, We use the same $V$ and $\kappa'$, while choosing $\lambda' = \delta' = 3/512$, which corresponds to stationary distribution is uniform when conditioning on the hyperplane. However cycling time distribution associated with CLA underestimates the mean and variance of the cycling time associated to the original CTMC. 

\begin{figure}[!htbp]
\centering
\begin{subfigure}{0.45\textwidth}
\includegraphics[width=\linewidth, height=5cm]{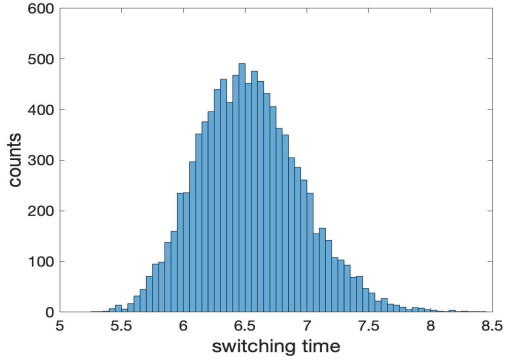}
\caption{Mean = 6.54, variance = 0.18.}
\label{fig:3d_unimodal_switch_stoch}
\end{subfigure}
\begin{subfigure}{0.45\textwidth}
\includegraphics[width=\linewidth, height=5cm]{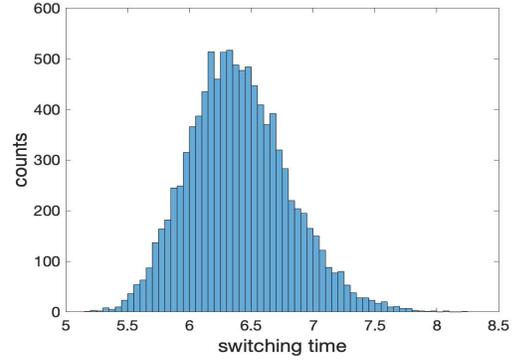}
\caption{Mean = 6.4, variance = 0.16.}
\label{fig:3d_unimodal_switch_CLA3d}
\end{subfigure}

\begin{subfigure}{0.45\textwidth}
\includegraphics[width=\linewidth, height=5cm]{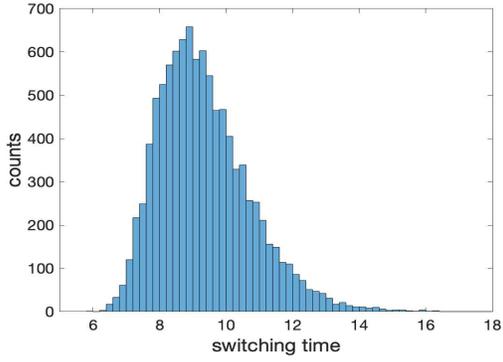}
\caption{Mean = 9.34, Variance = 1.92.}
\label{fig:3d_uniform_switch_stoch}
\end{subfigure}
\begin{subfigure}{0.45\textwidth}
\includegraphics[width=\linewidth, height=5cm]{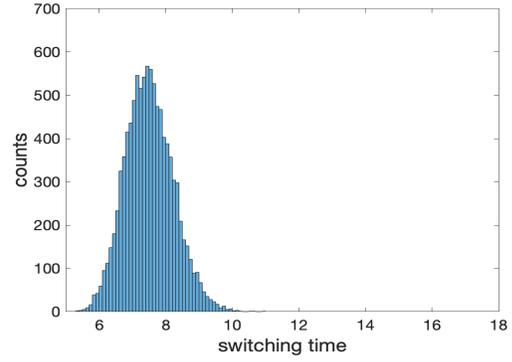}
\caption{Mean = 7.51, variance = 0.52.}
\label{fig:3d_uniform_switch_CLA3d}
\end{subfigure}
\caption{Histogram of cycling time for the CTMC \eqref{auto_CTMC}-\eqref{inout_CTMC} and the CLA 3D \eqref{Reflected_General} in $d=3$ are plotted in the left and right column respectively. Parameters of the simulation are given by $V= 256,  \kappa'= 1$ in \eqref{classicalScaling} for all four figures. In addition, in Figures \ref{fig:3d_unimodal_switch_stoch}-\ref{fig:3d_unimodal_switch_CLA3d}, $\lambda' = \delta' = 1/32$;  in Figures \ref{fig:3d_uniform_switch_stoch}-\ref{fig:3d_uniform_switch_CLA3d},  $\lambda' = \delta' = 3/512$. 
}
\label{fig:3d_hist_switch}
\end{figure}

\FloatBarrier

\subsection{Simulation results for the general d-dimensional TK model}
\label{sec:HDTK}

In this section, we perform stochastic simulation for higher dimensional TK models. We first give a precise description of 6-dimensional TK model in terms of two switchings: a slow switching between even and odd species, as well as fast switching between dominant regions near the boundary. We then investigate the effect of dimension $d$ on the mean cycling time.

Dynamics of 4-dimensional CTMC \eqref{auto_CTMC}-\eqref{inout_CTMC} is elaborated in the original publication of \cite{togashi2001transitions}. In particular, when the system volume $V$ is chosen appropriately (niether too small nor too large) , extinction of species slowly switches between odd and even species in auto-catalytic reaction loop. More precisely, the model switches between states that are abundant in odd species and states that are abundant in even species, during which the other species are almost extinct. Moreover within a temporal domain of abundant odd species, there are fast switches between species $A_1$ and $A_3$ with large imbalance between the two species,  either $X^1\gg X^3$, or $X^1\ll X^3$.

To describe the finite time dynamics for $6-$dimensional TK model, we simulate trajectories of the CTMC \eqref{auto_CTMC}-\eqref{inout_CTMC}, with the parameters $V = 64, \lambda' = \delta' = 1/256,  \kappa' = 1$. To investigate its switching behavior, we define the process $B(t)$ below that captures the disparity between odd speices and even species: 
\begin{align}\label{eq:defz}
    B(t) \coloneqq  \frac{1}{6V}\sum_{i=1}^{3} (X^{2i-1}(t) - X^{2i}(t)),
\end{align}
with initial condition $X^i_0 = V$ for all $i=1,2,\cdots 6$. Exponential convergence to the stationary distribution of CTMC, established in \cite{bibbona2020stationary}, guarantees the denominator $6V$ is approximately the average total population when $T$ is large. Similar expression is considered in \cite{togashi2001transitions} for 4-dimensional TK model to identify regions of $V$ where discreteness-induced transitions persist. 


In figure \ref{fig:6D_btraj}, finite time trajectory of $B(t)$ is plotted, and $B(t)$ switches between regions near $B(t) = 1$ and $B(t) = -1$, which are regions with abundant odd or even species respectively. Stationary distribution of $B(t)$ is plotted in Figure \ref{fig:6D_bdensity} via time averaging, which also yields a bimodal distribution having peaks near $-1$ and $1$.

\begin{figure}[!ht]
\begin{subfigure}{0.45\textwidth}
\includegraphics[width=\linewidth, height=5cm]{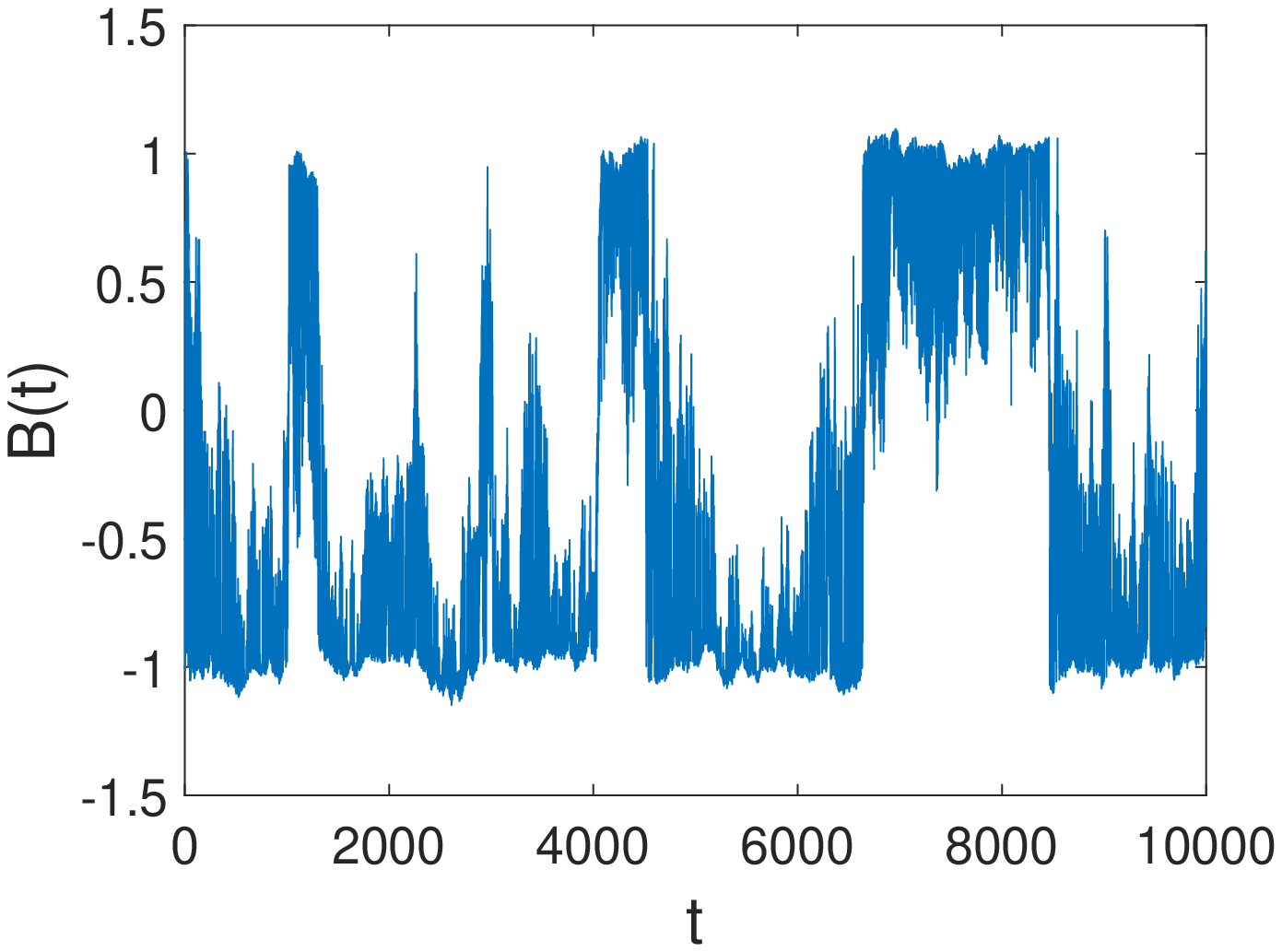}
\caption{Trajectory of $B(t)$. }
\label{fig:6D_btraj}
\end{subfigure}
\begin{subfigure}{0.45\textwidth}
\includegraphics[width=\linewidth, height=5cm]{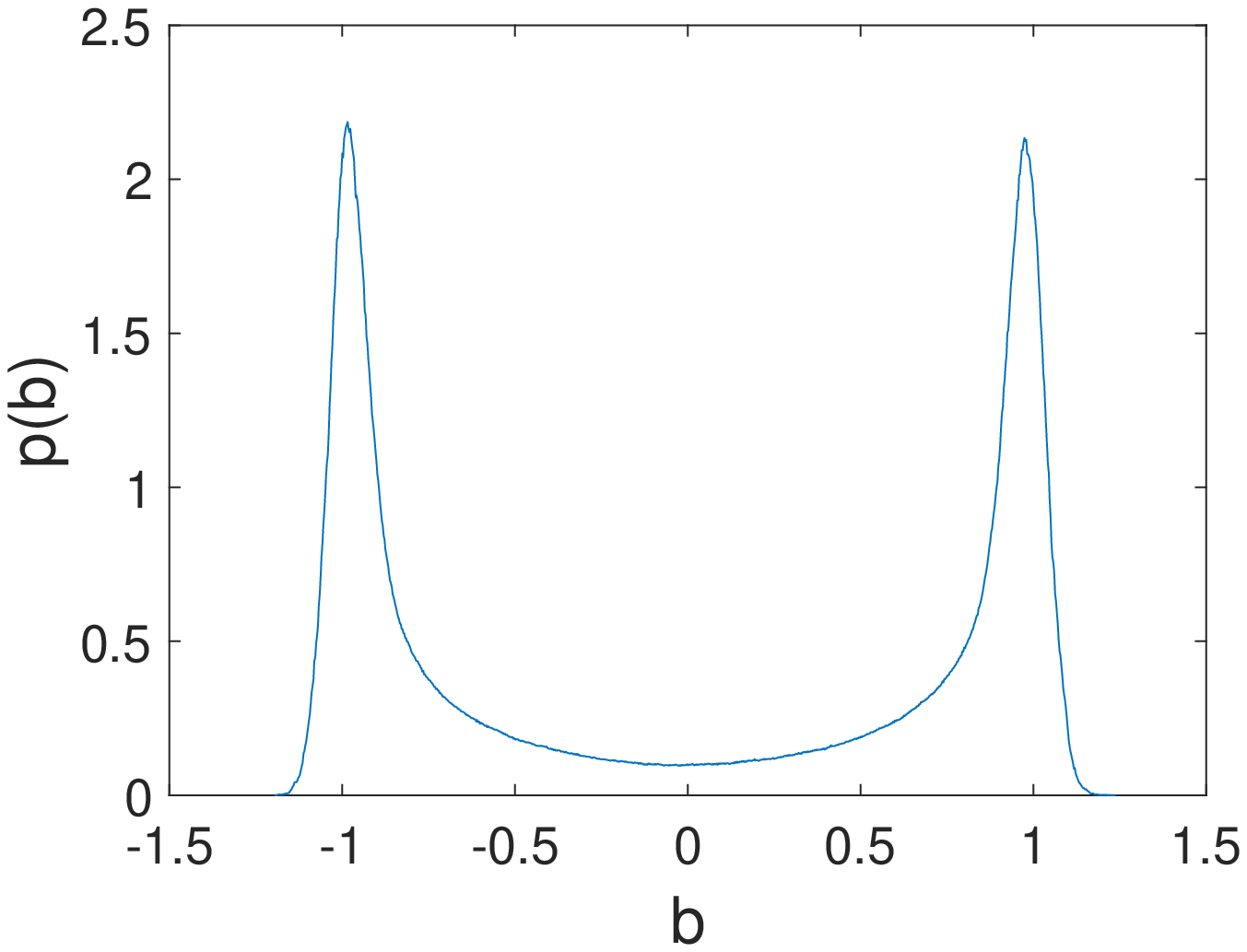}
\caption{Probability density of $B$ under stationarity.}
\label{fig:6D_bdensity}
\end{subfigure}
\caption{Finite trajectory of $B(t)$ as well as stationary distribution of $B(t)$ is obtained using a trajectory 6-dimensional TK model, whose parameters are given under classical scaling \eqref{classicalScaling} with  $V=64, \delta'= \lambda' = 1/256$ and $\kappa' = 1$. }
\end{figure}

Next, we investigate the \textit{dynamics in between slow switches} by analyzing trajectories with abundant odd species. In particular, joint distribution of 
\begin{equation*}
    \left(\rho_1(t),\, \rho_3(t)\right) = \left(\frac{X^1_t}{X^1_t+X^3_t+X^5_t},\; \frac{X^3_t}{X^1_t+X^3_t+X^5_t}\right),
\end{equation*}
conditioning on odd species being abundant, is plotted in figure \ref{fig:6D_cond_odddensity} under stationarity via time averaging. Throughout the simulation, the abundance condition is approximated by $B(t)\geq 0.95$. The joint distribution is concentrated within three boundary regions, which can be specified by $\Omega_i = \{ \rho_i \approx 0\}$ for $i = 1,3,5$, as each boundary implies one odd species is almost extinct while 2 other odd species are abundant. Moreover these regions are not symmetric with respect to the two dominant species, in the sense that  $(\rho_1, \rho_3)$ is on average $(0.4,0.6)$ when $\rho_5 \approx 0$. Such asymmetry exists since the states with more $A_1$ is sensitive with respect to the birth of $A_2$. More specifically, gaining $A_2$ (hence losing $A_1$) and losing $A_2$ (hence gaining $A_3$) is characterized by the autocatalytic reactions $A_1+A_2\rightarrow 2A_2$ and $A_2+A_3\rightarrow 2A_3$, rates of which are completely determined by the relative counts of $A_1$ and $A_3$. For states with $\rho_1\geq \rho_3$, higher counts of $A_1$ speed up the gaining of $A_2$, hence the transition $A_1\rightarrow A_2 \rightarrow A_3$ into states with $\rho_1\leq \rho_3$; on the other hand, states with $\rho_1\leq \rho_3$ is likely to remain unchanged, since $A_2$ is more likely to be exhausted.

In addition to asymmetry, dynamic of odd species, conditioning on odd species being abundant, moves between three dominant regions $\{\Omega_i\}_{i=1,3,5}$ in a clockwise manner as plotted in figure \ref{fig:6D_cond_odddensity}, $\Omega_5\rightarrow \Omega_3 \rightarrow \Omega_1 \rightarrow \Omega_5$. More specifically in the region $\Omega_5$ where $A_5$ is almost extinct, birth of $A_4$ would transition $A_3$ into $A_5$, which moves the process into the region $\Omega_3$. Birth of all other species can not change the dominant molecules of $A_1$ and $A_3$. Similar movement applies recursively leading to a clockwise cycle in figure \ref{fig:6D_cond_odddensity}.

\begin{figure}[!ht]
\centering
\begin{subfigure}{0.45\textwidth}
\includegraphics[width=\linewidth, height=6cm]{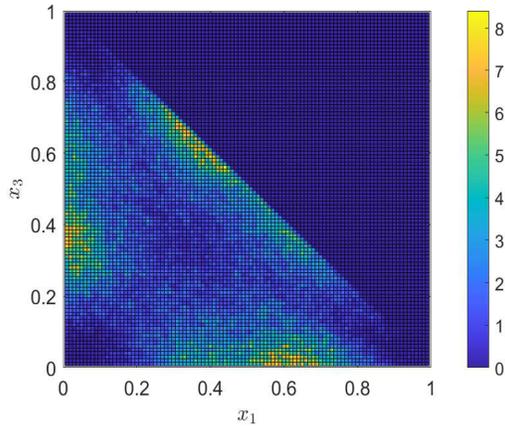}
\caption{Joint distribution of $(\rho_1,\rho_3)$ under stationarity.}
\label{fig:6D_cond_odddensity}
\end{subfigure}
\begin{subfigure}{0.45\textwidth}
\includegraphics[width=.9\linewidth, height=6cm]{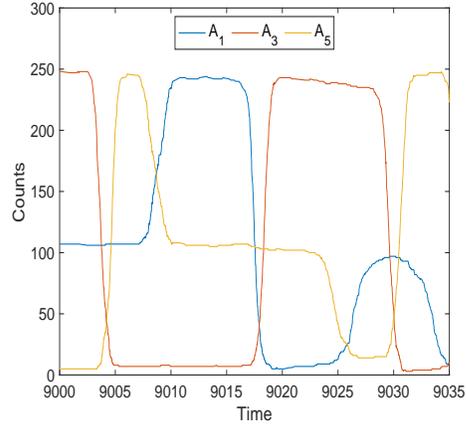}
\caption{Sample trajectory of odd species.}
\label{fig:6d_cyclicity_traj}
\end{subfigure}
\caption{Trajectory of 6D TK model is simulated until $T=10^6$ under classical scaling \eqref{classicalScaling} with  $V=64, \delta'= \lambda' = 1/256$ and $\kappa' = 1$, and initial condition $(0,\cdots,0,6V)$. Conditioning on $B\geq 0.95$, Joint distribution of $(\rho_1,\rho_3)$ under stationarity is plotted in Figure \eqref{fig:6D_cond_odddensity} along with a sample trajectory of odd species \ref{fig:6d_cyclicity_traj} when even species are almost extinct.}
\label{fig:6d_dynamic}
\end{figure}

To support our claims, a sample trajectory of odd species is plotted in figure \ref{fig:6d_cyclicity_traj}, during which we only have abundant odd species. The process stays within the region $\Omega_5$ at time $t= 9000$, which switches into the region $\Omega_3$ at the time $t=9010$, then proceeds to $\Omega_1$ around $t= 9020$. At time $t = 9025$ the process returns to $\Omega_5$ and completed a clockwise cycle.

We summarize the dynamics for $d=6$. Two types of switches can be utilized in describing finite time dynamics of $6-$dimensional TK model when discreteness-induced transitions persist. In particular, abundant molecule species switches between odd and even on a slower time scale. In between these switches, the dynamics dominated by odd species will cycle between boundary regions as $\Omega_5\rightarrow \Omega_3 \rightarrow \Omega_1 \rightarrow \Omega_5$, or $\Omega_6\rightarrow \Omega_4 \rightarrow \Omega_2 \rightarrow \Omega_6$ when the dynamics is dominated by even species.



\begin{figure}
    \centering
     \includegraphics[width=0.6\linewidth, height=7cm]{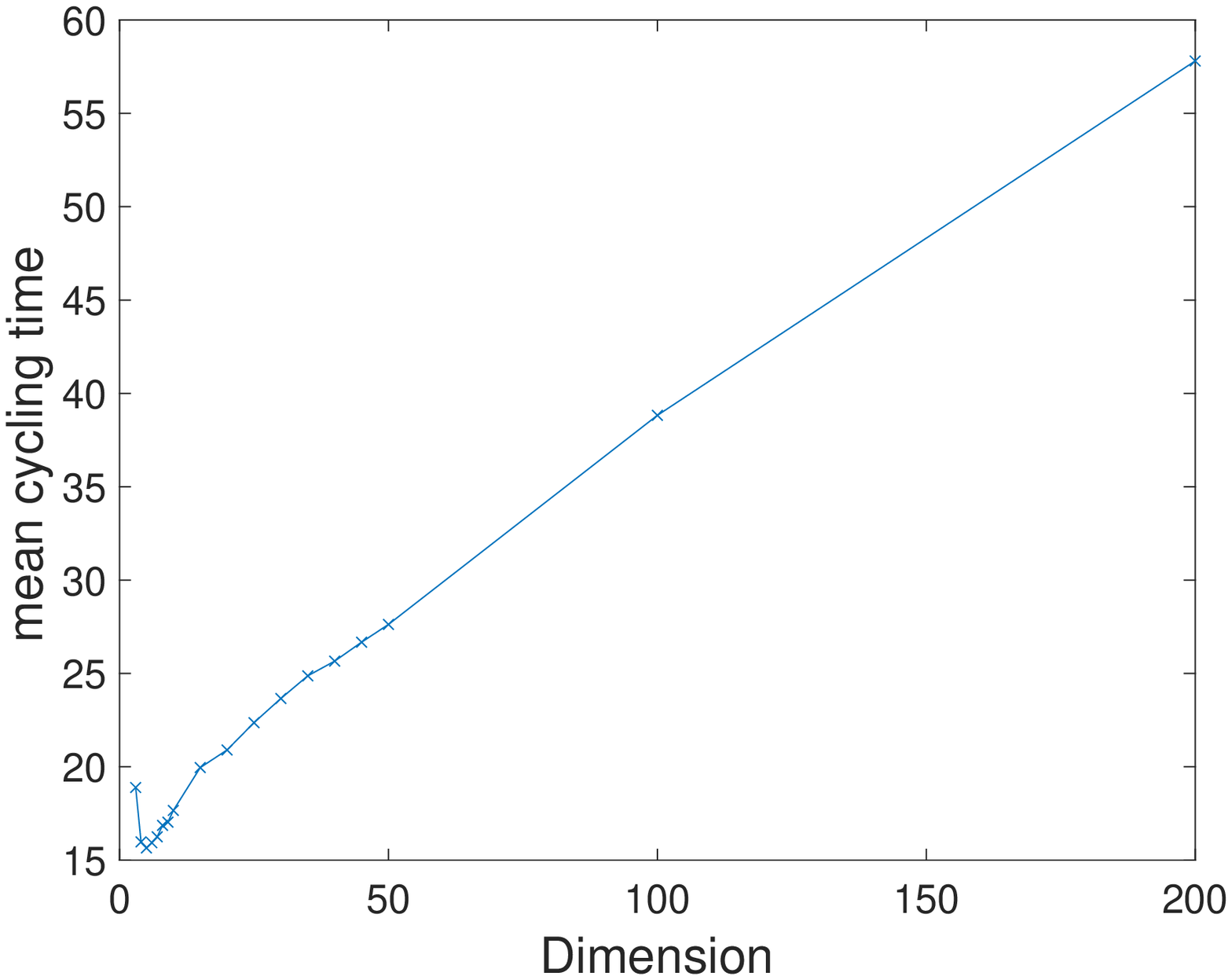}
    \caption{Mean cycling time is plotted against dimension $d$, for $d$ in the set $\{3,4,5,6,7,8,9,10,15,20,25,30,35,40,45,50,100,200\}$. Mean cycling time is averaged over 1000 samples when $d\leq 50$ (and over 100 samples for $d=100,200$), and it is determined for each CTMC trajectory with the parameters given by $V=64$, $\kappa'=1$ and $\lambda' = \delta' = 1/256$. Throughout all simulations, the initial condition is chosen as $X^d_0 = dV$ and zero for all other species. The cycling time for each trajectory is defined as the second time $X^d$ reaches peak abundance.}
    \label{fig:HD_cycle}
\end{figure}

\smallskip

Last but not least,  mean cycling time is plotted against some dimension $d$ in Figure \ref{fig:HD_cycle}. As in previous sections, cycling time is computed for each trajectory as the second time $X^d_t$ reaches peak abundance (initially there are only species $d$), and mean cycling time are averaged over independent samples. Despite longer cycles as $d$ increases, the cycling time decreases for smaller $d$. This is because if the birth of species $A_2$ occurs before the birth of species $A_1$, the switches from $A_d\rightarrow A_1$ and $A_1\rightarrow A_2$ would happen simultaneously, as it can be observed in Figure \ref{fig:traj_4D}-\ref{fig:traj_6D}). Such simultaneous switches speed up the cycling and hence reduce the cycling time, whereas these switches are more likely to occur in higher dimensions. However as $d$ increases further, such effect is properly averaged out, and mean cycling time approximately increases linearly in $V$. 

\FloatBarrier

\section*{Acknowledgments}
This research is initiated during the American Institute of Mathematics (AIM) workshop ``Limits and control of stochastic reaction networks" in July 2021. The authors gratefully acknowledge the support of AIM and the organizers of the workshop. The authors are indebted to the stimulating discussion during the monthly TK group meetings with Lea Popovic, Ruth Williams, Grzegorz Rempala, Hye Won Kang, Enrico Bibbona, Siri Paola, Wasiur Khuda Bukhsh and 
Felipe Campos Vergara.
This research is partially supported by NSF awards DMS 1855417 and DMS 2152103 and ONR grant TCRI N00014-19-S-B001 to W.T. Fan.

\section*{Conflict of interest}
The authors declare no conflicts of interest.

\bibliographystyle{alpha}
\bibliography{CLA_TK}

\section{Proofs}\label{S:Proof}

This section contains the proofs of the  results stated in Section \ref{S:CLA}. The proofs will be collected in subsection \ref{SS:Proofs}, after we develop  some preliminary estimates for the CLA in subsections \ref{SS:coe} and \ref{SS:Lya}.

\subsection{Estimates for the coefficients of CLA}\label{SS:coe}

Let $\Gamma,\sigma,b,\gamma$ be defined as in \eqref{CLA_coe1}-\eqref{CLA_coe2}, note that $b$ is globally Lipschitz for dimension $d=2$, but only locally Lipschitz for $d\geq 3$ due to the terms $\sum_{k = 1}^d e_k(\kappa'(x_{k-1} - x_{k+1})x_k$. Precisely,
\begin{align*}
&|b(x)-b(y)| \\
\leq&\, \kappa'\Big|\sum_{k = 1}^d e_k\Big((x_{k-1} - x_{k+1})x_k -(y_{k-1} - y_{k+1})y_k\Big)\Big| + \delta'\Big|\sum_{k = 1}^d e_k\,(x_k-y_k)\Big|\\
= &\,\kappa'\Big|\sum_{k = 1}^d e_k\Big((x_{k-1}-y_{k-1})x_k + y_{k-1}(x_k-y_k)\,+\,(x_{k+1}-y_{k+1})x_k + y_{k+1}(x_k-y_k)\Big)\Big| + \delta'|x-y|\\
\leq &\,2\kappa'|x-y|_{\infty}(|x|+|y|)\,+\, \delta'|x-y|\\
\leq &\,|x-y|\,\Big(\delta'+2\kappa'(|x|+|y|)\Big),
\end{align*}
where $|x|=\sqrt{\sum_{i=1}^dx_i^2}$ is the Euclidean norm of $x$, and $|x|_{\infty}:=\max_{1\leq i\leq d}|x_i|$.

Note also that $\Gamma(x)$ is symmetric and positive, for all $x \in \R^d$. Precisely,  for all $x \in \R^d_+$ and $\theta \in \R^d$, 
\begin{align}
    \langle \theta, \Gamma(x) \theta\rangle &= \sum_{1 \le k \le d} \left( \lambda'+\delta'x_k\right)\theta_k^2 + \sum_{1 \le k \le d} \left(\theta_k^2 - \theta_k\theta_{k-1} \right)x_{k}x_{k-1} + \sum_{1 \le k \le d} \left(\theta_k^2 - \theta_{k}\theta_{k+1}\right)x_kx_{k+1}\notag\\
    &=\sum_{1 \le k \le d} \left( \lambda'+\delta'x_k\right)\theta_k^2 + \sum_{1 \le k \le d} \left(\theta_k^2 - \theta_k\theta_{k-1} \right)x_{k}x_{k-1} + \sum_{1 \le k \le d} \left(\theta_{k-1}^2 - \theta_{k-1}\theta_{k}\right)x_{k-1}x_{k} \notag\\
    &=\sum_{1 \le k \le d} \left( \lambda'+\delta'x_k\right)\theta_k^2 + \sum_{1 \le k \le d} \left(\theta_k - \theta_{k-1}\right)^2 x_{k-1}x_k  \notag\\ 
    &\ge \lambda'|\theta|^2. \label{EllipticGamma}
\end{align}
We can show that $\sigma$ is locally Lipschitz  and grows linearly.
We note that 
\begin{align*}
    \|\Gamma(x)\|_\infty := \max_{1 \le i,j \le d} |\Gamma_{i,j} (x)| &\le \max_{1 \le i \le d}\kappa ' (|x_{i-1}| + |x_{i+1}|)|x_i|+ \lambda' + \delta'|x_i|\\
    &\le 2 \kappa' |x|_1^2 +\delta' |x|_1 +\lambda'.
\end{align*}
So the operator norm on the space of $d\times d$ matrix is of quadratic growth:
\begin{align*}
    \|\Gamma(x)\| := \sup_{y \in \R^d,|y| = 1} |\Gamma(x)y| \le C\left( 2 \kappa' |x|_1^2 +\delta' |x|_1 +\lambda'\right)
\end{align*}
for some $C > 0$ by the equivalence of norms in $\R^n$. Then we have 
\begin{align*}
    \|\sigma(x)\|_\infty \le C\|\sigma(x)\|  =C\sqrt{\|\Gamma(x)\|} \le C' \sqrt{\left( 2 \kappa' |x|_1^2 +\delta' |x|_1 +\lambda'\right)}.
\end{align*}
Locally Lipschitz of $\sigma(\cdot)$ is given by (\ref{EllipticGamma}) and the Powers-Stormer inequality \cite[Lemma 4.2]{powers1970free}:
\begin{align*}
    \| \sqrt{\Gamma(x)} - \sqrt{\Gamma(y)}\|_{HS} \le \frac{1}{\sqrt{\lambda'}} \| \Gamma(x) - \Gamma(y)\|_{HS},
\end{align*}
where $\|\cdot\|_{HS}$ denotes the Hilbert-Schmidt norm. Since all norms are equivalent in finite dimensional vector space and each entry in $\Gamma(\cdot)$ is of second order polynomial, $\Gamma(x)$ is locally Lipschitz, hence $\sqrt{\Gamma(\cdot)} = \sigma(\cdot)$ is also locally Lipschitz.
\subsection{Lyapunov inequalities and moment estimates}\label{SS:Lya}

Recall the Lyapunov function $U$ defined in \eqref{lyapunov_function} and the differential operator ${\cal L}$ in \eqref{generator}. We also write $U^p(x) := \left(U(x)\right)^p$ for simplicity.
\begin{lemma}\label{lem:local_time_out}
Let $\gamma$ be as in (\ref{CLA_coe1}), then we have $\nabla U^p(x) \cdot \gamma(x) \le 0$ for all $x \in \R_+^d$ and  $p\in\mathbb{N}$.
\end{lemma}
\begin{proof}
    Observe that for all $x \in \R^d_+$,
    \begin{align*}
        \nabla U^p(x) = \sum_{i = 1}^d 2p\left(|x|_1 - \frac{d\lambda'}{\delta'}\right)^{2p-1}e_i
    \end{align*}
    where $\{e_i\}_{i = 1}^d$ is the standard basis in $\R^d$, since $\gamma(x) = \frac{b(x)}{|b(x)|}$, we only need to check $\nabla U^p(x) \cdot b(x) \le 0$: 
    \begin{align*}
        \nabla U^p(x) \cdot b(x) &= 2p\frac{1}{\delta'} \sum_{i = 1}^d (\delta' |x|_1 - d\lambda')^{2p-1}(\kappa'(x_{i - 1}- x_{i+1})x_i - \delta'x_i +\lambda')\\
        &= \frac{-2p}{\delta'}\left(\delta'|x|_1 - d \lambda'\right)^{2p} \le 0.
    \end{align*}
\end{proof}

\begin{lemma}[Lyapunov inequalities]\label{lm:lyapunov}
Let $U$ and ${\cal L}$ be defined in (\ref{lyapunov_function}) and (\ref{generator}) respectively. Then for all  $p\in\mathbb{N}$, there exist constants $c_p,c_p' >0$ such that 
\begin{align}\label{lyapunov}
    {\cal L}U^p(x) \le - c_pU^p(x) +c_p'
\end{align}
for all $x \in \R_+^d$. Furthermore, there is a compact set $K \subset \R_d^+$ and $f: \R^d_+\to [1,\infty)$ such that 
\begin{align}\label{lyapunov_2}
    {\cal L} U(x) \le - c_1' f(x) + 1_K c_2'
\end{align}
for some positive constants $c_1'$ and $c_2'$. 
\end{lemma}

\begin{proof}
   For $x = (x_1,\cdots,x_d) \in \R^d_+$ and $1 \le i,j \le d, x\in \R_+^d$, we have that  
    \begin{align*}
        \frac{\partial}{\partial x_i}U^p(x) = p\left( |x|_1 - \frac{d\lambda'}{\delta'}\right)^{p-1};        \quad \frac{\partial^2}{\partial x_i \partial x_j}U^p(x) = 2p(2p-1)\left( |x|_1 - \frac{d\lambda'}{\delta'}\right)^{2p-2} .
    \end{align*}
    Applying the differential operator ${\cal L}$ to $U^p(x)$, 
    \begin{align*}
        {\cal L} U^p(x) &= \frac{1}{2V}\sum_{i,j = 1}^d \Gamma_{i,j}(x)\frac{\partial^2}{\partial x_i \partial x_j}U^p(x) +\sum_{i = 1}^d b_i(x) \frac{\partial }{\partial x_i} U^p(x)\\
        &= \frac{p(2p-1)}{V}\left( |x|_1 - \frac{d\lambda'}{\delta'}\right)^{2p-2} \sum_{i,j = 1}^d \Gamma_{i,j}(x) + p\left(|x|_1 - \frac{d\lambda'}{\delta'}\right)^{2p-1}\sum_{i=1}^d b_i(x)\\
        &=\frac{p(2p-1)}{V}\left( |x|_1 - \frac{d\lambda'}{\delta'}\right)^{2p-2} (\delta'|x|_1 +d \lambda') + p\delta'\left(|x|_1 - \frac{d\lambda'}{\delta'}\right)^{2p-1}\left(\frac{d\lambda'}{\delta'} - |x|_1\right)\\
        &=\frac{p(2p-1)}{V}\left( |x|_1 - \frac{d\lambda'}{\delta'}\right)^{2p-2} (\delta'|x|_1 +d \lambda') - p\delta'\left(|x|_1 - \frac{d\lambda'}{\delta'}\right)^{2p}\\
        &=p\delta'\left(|x|_1 - \frac{d\lambda'}{\delta'}\right)^{2(p-1)} \left(\frac{(2p-1)}{V}\left(|x|_1 + \frac{d\lambda'}{\delta'}\right) -  \left( |x|_1 - \frac{d\lambda'}{\delta'}\right)^2\right)
    \end{align*}
    Let $c_p = \frac{p\delta'}{2} $ then 
    \begin{align*}
        {\cal L}U^p(x) + c_pU^p(x)&= p\delta' \left(|x|_1 - \frac{d\lambda'}{\delta'}\right)^{2(p-2)} \left( \frac{2p-1}{V}\left(|x|_1 + \frac{d\lambda'}{\delta'}\right) - \frac{1}{2}\left(|x|_1 - \frac{d\lambda'}{\delta'}\right)^2\right)
    \end{align*}
    which is bounded above in $\R^d_+$ by some positive constant $c_p'$, due to the fact that 
    \[
    	\left( \frac{2p-1}{V}\left(|x|_1 + \frac{d\lambda'}{\delta'}\right) - \frac{1}{2}\left(|x|_1 - \frac{d\lambda'}{\delta'}\right)^2\right)
    \] 
    is  quadratic with leading coefficient being a negative number. This proves (\ref{lyapunov}).
    
    To prove (\ref{lyapunov_2}), we let $p = 1$ and $f(x) := V(x) + 1$ and let $c_1' = \min\{\delta',\frac{d\lambda'}{2V}\}$, then 
    \begin{align}
        {\cal L}U(x) + c_1'U(x) = \left(\frac{d \lambda'}{V} - c_1'\right) + \frac{\delta'}{V}|x|_1 - (2\delta'- c_1') U(x).
    \end{align}
    Note that by the choice of $c_1'$, both $\frac{d \lambda'}{V} - c_1'$ and $(2\delta'- c_1')$ are positive, and by the basic facts of quadratic functions, ${\cal L}U(x) + c_1'U(x)$ is only positive on a compact set, call it $K$, and since it is also continuous, it is uniformly bounded by some $c_2' > 0$. So (\ref{lyapunov_2}) holds.
\end{proof}

\begin{lemma}\label{lem:sup_bound}
    Let $x \in \R^d_+$ and let $Z$ be the solution to  (\ref{Reflected_General}) with $Z_0 = x$, then 
    \begin{align}\label{sup_bound}
        \E_x\left[\sup_{s \in [0,t]} \left|Z_s\right|^p\right] < \infty,\quad \forall p \in \mathbb{N} \text{ and } t \in\R_+.
    \end{align}
\end{lemma}
\begin{proof}
    Suppose $Z$ is the process that solves the CLA (\ref{Reflected_General}) that starts at $x \in \R_+^d$ and let $U$ be defined as (\ref{lyapunov_function}), then by Ito's formula, we have 
    \begin{align}\label{Ito}
        U(Z_t) = U(x) + \int_0^t {\cal L}U(Z_s)ds + M_t + \int_0^t\nabla U(Z_s) \cdot \gamma(Z_s) dL_s,\quad a.s.\, \forall t\ge 0,
    \end{align}
    where $M_t$ is the local martingale term that has the following explicit expression:
    \begin{align}\label{lcoal_mtg}
        M_t = \frac{2}{\sqrt{V}}\sum_{j = 1}^d \int_0^t \left(|Z_s|_1 - \frac{d\lambda'}{\delta'}\right)\left(\sum_{i = 1} \sigma_{i,j}(Z_s) \right)dW^{(j)}_s,
    \end{align}
    where $(\sigma_{i,j})_{i,j = 1}^d = \sigma = \sqrt{\Gamma}$ is the dispersion matrix and $\{W^{(j)}\}_{j =1}^d$ are independent one dimensional Brownian motions. By Lemma \ref{lem:local_time_out}, the process $\int_0^t\nabla U(Z_s) \cdot \gamma(Z_s) dL_s$ is non-positive and decreasing since $\gamma(x) = \frac{b(x)}{\|b(x)\|}$. Therefore we have the following almost sure inequality for all $t \ge 0$:
    \begin{align*}
         U(Z_t) &\le U(x) + \int_0^t {\cal L}U(Z_s)ds + M_t\\
         &= U(x) + \int_0^t \frac{1}{V} \left(d\lambda' +\delta'|Z_s|_1\right) - 2\delta' U(Z_s)ds +M_t.
    \end{align*}
    Note that $\sum_{i,j = 1}^d \Gamma_{i,j}(x) =  d\lambda' +\delta' |x|_1 = \delta'\sqrt{U(x)} +2 d\lambda' $ and $\delta' U(x) \ge 0$ for all $x$, therefore we can rewrite the above inequality as 
    \begin{align}\label{pre_bound}
        U(Z_t) \le U(x) + \frac{1}{V}\int_0^t \delta'\sqrt{U(Z_s)} + 2d\lambda'ds + M_t,\quad \text{ a.s. for all }t \ge 0.
    \end{align}
    For each $m\in\mathbb{N}$, we let $\{\tau_m'\}_{m = 1}^\infty$  be the family of stopping times defined by $\tau_m' = \inf\{t\geq 0: |Z_t|_1 \ge m\}$ and let $\tau_m = \tau'_m \wedge m$. Define the stopped process $\widetilde{Z}^{(m)}$ by $\widetilde{Z}^{(m)} = (\widetilde{Z}^{(m),i}_\cdot)_{1 \le i \le d} =  (Z^{i}_{\cdot \wedge \tau_m})_{1 \le i \le d} = Z_{\cdot \wedge \tau_m}$. Let $M_t^{(m)} = M_{t\wedge \tau_m}$ and replace $Z_t$ and $M^t$ in (\ref{pre_bound}) by $\widetilde{Z}^{(m)}_t$ and $M_t^{(m)}$ respectively, then raise both sides  by a power of $2p$, by Jessen's inequality we get that, for all $T >0$ and $t \in [0,T]$, there is a $C$ depend only on $T$ and $p$ such that:
    \begin{align*}
        U^{2p}\left(\widetilde{Z}^{(m)}_t\right) \le C \left( U^{2p}(x) +\int_0^t \delta'^{2p}U^\frac{2p}{2}\left(\widetilde{Z}^{(m)}_s \right) + (2d\lambda')^{2p} ds +\left(M^{(m)}_t\right)^{2p} \right).
    \end{align*}
    Now take the sup over time and taking expectation  to get $\E_x\left[\sup_{s \in [0,t]} U^{2p}(\widetilde{Z}_s^{(m)})\right] $ is bounded by
    \begin{align}\label{pre_bound_m}
         C\left( U^{2p}(x) + \int_0^t \delta'^{2p} \E_x\left[\sup_{\tau \in [0,s]}\left(U^\frac{2p}{2}(\widetilde{Z}_\tau^{(m)})\right)\right] + (2d\lambda')^{2p}  \,ds + \E_x\left[\sup_{s \in [0,t]} (M_s^{(m)})^{2p}\right]\right).
    \end{align}
    We wish to obtain an inequality of Gronwall type, note that the square root function is  concave on $\R_+$, so by Jensen's inequality we get
    \begin{align}\label{grownwall_1}
        \E_x\left[\left(\sup_{\tau \in [0,s]}U\left(\widetilde{Z}_\tau^{(m)}\right)\right)^\frac{2p}{2}\right]\le \sqrt{\E_x\left[\sup_{\tau \in [0,s]}U^{2p}\left(\widetilde{Z}_\tau^{(m)}\right)\right]}.
    \end{align}
    
    Denote $\langle M\rangle_\cdot$ as the quadratic variation of a process $M_\cdot$. By BDG  inequality, there is an absolute constant $C > 0$ that depends only on $p$ such that
    \begin{equation}\label{BDG}
         \E_x\left[\sup_{s \in [0,t]} (M_s^{(m)})^{2p}\right] \le C \,\E_x\left[\langle M^{(m)}\rangle_t^p\right], \quad t\in\R_+.
    \end{equation}
Furthermore, for all $T>0$, there exists constants $C_1,\,C_2$ (depending on $T$ and $p$) such that for $t\in[0,T]$, 
    \begin{align*}
       \E_x\left[\langle M^{(m)}\rangle_t^p\right]&= \frac{1}{V}\E_x \left(\int_0^t U\left(\widetilde{Z}^{(m)}_s\right) \sum_{i,j = 1}^d \Gamma_{i,j}\left(\widetilde{Z}^{(m)}_s\right) ds\right)^p\\ 
        &\le \frac{C_1}{V}\int_0^t \E_x\left[ \left(U\left(\widetilde{Z}_s^{(m)}\right) \sum_{1 \le i,j \le d}\Gamma_{i,j}(\widetilde{Z}_s)\right)^p \right]ds\\ 
        &= \frac{C_1}{V}\int_0^t \E_x\left[ \left(U(\widetilde{Z}_s^{(m)}) \left(d \lambda' + \delta' |\widetilde{Z}_s^{(m)}|_1\right)\right)^p\right] ds\\
        &= \frac{C_1}{V}\int_0^t \E_x\left[ \left(U(\widetilde{Z}^{(m)}_s)\left(\delta'\sqrt{U(\widetilde{Z}^{(m)}_s)}+2d\lambda'\right)\right)^p\right]ds\\
        &\le \frac{C_2}{V}\int_0^t \delta'^p\E_x\left[\sup_{\tau \in [0,s]} U^{\frac{3}{2}p}\left(\widetilde{Z}^{(m)}_\tau\right)\right] + (2d\lambda')^p\E_x\left[ \sup_{\tau \in [0,s]} U^p\left(\widetilde{Z}^{(m)}_\tau\right)\right]ds
    \end{align*}
    where the above positive constants $C_2$ depends only on $p,T$ and is again independent of $m$. Finally, by Jensen's inequality again, we see there is $C_0 >0$ depends on $T,p,V$ such that the following inequality holds for all $t \in [0,T]$ 
    \begin{align}\label{grownwall_2}
        \E_x\left[\sup_{s \in [0,t]} (M_s^{(m)})^{2p}\right] \le C_0 \int_0^t \delta'^p \left(\E_x\left[ \sup_{\tau \in [0,s]} U^{2p}(\widetilde{Z}_\tau^{(m)})\right]\right)^\frac{3}{4}+(2d\lambda')^p \left(\E_x\left[\sup_{\tau \in [0,s]} U^{2p}(\widetilde{Z}^{(m)})\right]\right)^\frac{1}{2}\,ds
    \end{align}
    We let $y(t):=\E_x\left[\sup_{s \in [0,t]} U^{2p}(Z_s^{(m)})\right]$ and combine the inequalities (\ref{pre_bound_m}), (\ref{grownwall_1}) and (\ref{grownwall_2}) to get the following: 
    \begin{align}\label{grown_pre}
        y(t) \le C_0\left( U^{2p}(x) + \int_0^t (\delta'^{2p} +(2d\lambda')^p)\sqrt{y(s)}+(2d\lambda')^{2p} +  \delta'^p \left(y(t)\right)^\frac{3}{4}ds \right).
    \end{align}
    We let $\omega(s) = (\delta'^{2p} +(2d\lambda')^p)\sqrt{s} + (2d\lambda')^{2p} + \delta'^p s^\frac{3}{4}$, which is positive and strictly increasing on $[0,\infty)$. Therefore, $\Phi:[0,\infty)\to (0,\infty)$  defined by
 $\Phi(t) = \int_0^t \frac{1}{\omega(s)}ds$  is strictly increasing and continuous. Hence $\Phi^{-1}$ exists and is continuous. Also, note that $\E_x[\sup_{s \in [0,t]}U^{2p}(\widetilde{Z}^{(m)}_s)]$ is continuous in $t$ , so by the Gronwall-type inequality \cite[Theorem 4, p3]{dragomir2002some} with $C_0U^{2p}(x)$ and $ C_0$ in the place of $M$ and $\Psi$, we have 
    \begin{align}\label{grown_result}
         \E_x\left[\sup_{s\in [0,t]} U^{2p}\left(\widetilde{Z}^{(m)}_s\right)\right] \le \Phi^{-1}\left(\Phi\left(C_0 U^{2p}(x)\right) + C_0 \,t\right) \quad\text{for }t\in[0,T] \text{ and }x\in\R_+^d.
    \end{align}
    Note that the right hand side does not depend on $m$, so by taking $m\to\infty$ on the left hand side and invoke Fatou's lemma, we see $\E_x[\sup_{s \in [0,t]}U^{2p}(Z_s)]$ is finite for all $t \ge 0$ which implies (\ref{sup_bound}).
\end{proof}

\begin{prop}
\label{prop:feller}
The solutions to the (\ref{Reflected_General}) starting from different $x\in \R_+^d$ form a Feller process.
\end{prop}
\begin{proof}
    Let $\tau^x_M = \inf\{t \ge 0:|Z^x_t|_1 \ge M\}$ for $M > M^* =\frac{d \lambda'}{\delta'}+1$ as in \cite[(4.25)]{leite2019constrained}, then $\tau^x_M$ is a stopping time, where we note that $u$ in \cite[(4.25)]{leite2019constrained} is equal to $(1,1,\cdots,1)$ in our case. We will show that for all $f \in C_b(\R^d_+)$ and $t \ge 0$ with $x \in \R_+^d$, 
    \begin{align}\label{to_feller}
        \E\left[\left|f\left(Z^x\right)-f\left(Z^y\right)\right|\right] \to 0\quad \text{ as } y \to x,
    \end{align}
    where $Z^x,Z^y$ is the strong solution to the CLA (\ref{Reflected_General}) that start at $x,y \in \R_+^d$ respectively. This would imply Feller property since the function $x \mapsto \E[f(X_t^x)]$ is bounded.
    
    Fixing $x \in \R^d_+$, we first note that for any $\epsilon > 0, \exists M_\epsilon> M^*>0$ such that for all $M \ge M_\epsilon$ we have 
    \begin{align}
        \sup_{y \in B(x,1)} \P(\tau_M^y \le t) < \epsilon.
    \end{align}
    Indeed, by Markov inequality, we have 
    \begin{align*}
        \sup_{y \in B(x,1)} \P(\tau_M^y \le t)= \sup_{y \in B(x,1)} \P\left[\sup_{s \in [0,t]} |Z^y_s| \ge M\right] & \le \frac{\sup_{y \in B(x,1)}\E\left[\sup_{s\in [0,t]}|Z^y_t|\right]}{M}
    \end{align*}
    and by (\ref{grown_result}) and continuity of $\Phi$ in the proof of Lemma \ref{lem:sup_bound}, the map $\R_+^d\ni y \mapsto\E\left[\sup_{s\in [0,t]}|Z^y_t|^4\right]$ is uniformly bounded on compact set, hence the numerator on the right hand side is bounded by some $C_x >0$ depending only on $x$. So the right hand side goes to zero as $M \to \infty$.
    
    Assume $y \in B(x,1)$ and $M > M_\epsilon$, then we have the following decomposition of (\ref{to_feller}):
    \begin{align}\label{last_feller}
        \E\left[\left|f\left(Z^x_t\right)-f\left(Z^y_t\right)\right|\right]  &\le  \E\left[\left|f\left(Z^x_t\right)-f\left(Z^y_t\right)\right|1_{\tau^x_M\wedge \tau^y_M \ge t}\right] + 2\|f\|_\infty \left(\P(\tau_M^x \le t) + \P(\tau^y_M \le t)\right)\nonumber \\
        &\le \E\left[\left|f\left(Z^{x,(M)}_t\right)-f\left(Z^{y,(M)}_t\right)\right|\right] +4 \|f\|_\infty \epsilon,
    \end{align}
    where $Z^{x,(M)}_\cdot$ denote the stopped process $Z^x_{t \wedge\tau_M^x}.$ Then following the proof of Theorem 6.1 in \cite{leite2019constrained}, with its modification to the proof of Theorem 5.1 of \cite{dupuis1993sdes} to the stopped process and obtain the following inequality similar to \cite[equation (6.4)]{leite2019constrained}: for each $T>0$, there is a constant $C>0$ such that for $t\in[0,T]$,
    \begin{align}\label{ruth_unique}
        \E\left[\sup_{s \in [0,t]}\left|Z^{x,(M)}_s - Z^{y,(M)}_s\right|^2 \right] \le C \left(|x - y|^2 +\int_0^t \E\left[\sup_{\tau \in [0,s]}\left|Z^{x,(M)}_\tau - Z^{y,(M)}_\tau\right|^2 \right]ds \right).
    \end{align}
   By  Gronwall's inequality,  that as $B(x,1) \ni y \to x$, the left hand side of (\ref{ruth_unique}) goes to zero. Therefore, from (\ref{last_feller}) we see that 
   \begin{align*}
       \lim_{B(x,1)\ni y \to x}\E\left[\left|f\left(Z_t^x\right) - f\left(Z_t^y\right)\right|\right] \le 4\|f\|_\infty \epsilon,
   \end{align*}
   and since $\epsilon > 0$ is arbitrary, we see that (\ref{to_feller}) holds.
\end{proof}


\subsection{Proofs for the results in Section \ref{S:CLA}}\label{SS:Proofs}

\begin{proof}[Proof of Theorem \ref{T:Wellposed}]
The reaction network of the TK model contains inflows and outflows of all species \eqref{inout}. Furthermore, the autocatalytic reactions \eqref{auto} satisfies the mass-conserving/mass-dissipating assumption in  \cite[Definition 3.1(a)]{leite2019constrained} with $u=(1,1,\cdots,1)\in \R^d$. Hence \cite[Assumption 3.1]{leite2019constrained} is satisfied.
Strong uniqueness of the CLA follows from \cite[Theorem 6.1]{leite2019constrained}, and weak existence follows from \cite[Section 7]{leite2019constrained}. Now strong existence and weak uniqueness follow from the Yamada-Watanabe-Engelbert theorem.
 \cite{dynkin1956strong}. 
 
The Feller property is given by Proposition \ref{prop:feller}, hence it also has strong Markov property by \cite{dynkin1956strong}.
\end{proof}

\smallskip

The following result, which is a combination of \cite[Theorem 4.2 \& Theorem 4.5]{meyn1993stability} (see also \cite[Theorem 2.2.12]{berglund2021long}), provides a condition on Lyapunov functions that guarantees existence of a unique invariant distribution for Feller diffusions.     Note that a skeleton chain of a Feller diffusion also possesses the Feller property. By \cite[Proposition 2.2]{sarantsev2017reflected},  every compact subset is petite for the skeleton chain. 

\begin{thm}{\rm (\cite[Theorem 4.2 \& Theorem 4.5]{meyn1993stability})}\label{T:Feller_existInv}
Let $X$ be  a 
Feller diffusion. If $U$ is a positive function such that
for some positive constants $c_1,c_2>0$, a function $f:\R^d\to[1,\infty)$,  a compact petite set $K\subset \R^d_+$ such that $U$ is bounded on $K$ and the following inequality holds for $X$,
\begin{equation}\label{E:Feller_existInv}
    \mathcal{L}U(x) \leq -c_1f(x) + c_2\,{\bf 1}_{K}(x),\quad \forall x \in \R^d_+,
\end{equation}
then the diffusion is positive Harris recurrent and there is an invariant probability measure $\pi$ for $X$, also any invariant probability $\pi$ satisfies $\int f(x)\pi(dx)\leq c_2/c_1$.
\end{thm}

The following result from \cite{meyn1993stability} (see also \cite[Theorem 2.2.15]{berglund2021long}) provides a condition on Lyapunov functions that guarantees exponential ergodicity of Feller diffusions.
\begin{thm}{\rm (\cite[Theorem 6.1]{meyn1993stability})} \label{T:Feller_convInv}
Let $X$ be  a 
Feller diffusion. 
Assume there exists a norm-like function $U$, and constants  $c_1>0$ and $c_2\in \R$ such that  $X$ satisfies
\begin{equation}\label{E:Feller_existInv2}
    \mathcal{L}U(x) \leq -c_1\,U(x) + c_2
\end{equation}
for all $x\in\R^d_+$. 
Then $X$ is $f$-exponentially ergodic with $f=U+1$.
\end{thm}

Conditions \eqref{E:Feller_existInv} and
\eqref{E:Feller_existInv2} are called (CD2) and (CD3) respectively in \cite{meyn1993stability}, and they are satisfied for our CLA (process $Z$) thanks to Lemma \ref{lm:lyapunov}.

\begin{proof}
[Proof of Theorem \ref{thm:stationaryExist} ]
    Note that Proposition \ref{prop:feller} implies that $Z$ is a Feller process.  
    Let $U$ be the Lyapunov function defined as (\ref{lyapunov_function}), inequality (\ref{lyapunov_2}) implies that $U$ satisfies the inequality (\ref{E:Feller_existInv}). Therefore, by Theorem \ref{T:Feller_existInv}, $Z$ is positive Harris recurrent has a unique invariant probability measure \cite[Section 4.1]{meyn1993stability}.
    
 It remains to  show that all the moments of  the stationary distribution $\pi$ are finite. By Ito's formula, Lemma \ref{lem:local_time_out} and Lemma \ref{lm:lyapunov}, we have that for each $p \in\mathbb{N}$, there exist some positive constants $c_p$ and $c_p'$, the following inequality holds for $t\in\R_+$:
    \begin{align*}
        \E_x[U^p(Z_t)] &\le U^p(x) + \E_x\left[\int_0^t {\cal L} U^p(Z_s) ds\right]\\
        &\le U^p(x) - \int_0^t c_p \E_x\left[U^p(Z_s)\right]ds + c_p't.
    \end{align*}
By rearranging terms and dividing by $t,c_p$, it follows that 
    \begin{align}\label{pre_ergodic}
        \E_x\left( \frac{1}{t} \int_0^t U^p(Z_s)ds\right) =  \frac{1}{t} \int_0^t \E_x\left[U^p(Z_s)\right]ds \le \frac{U^p(x)}{t\, c_p} + \frac{c_p'}{c_p}.
    \end{align}
    Now, let us define $U^p_M$ as the truncated function of $U^p$ at $M \ge 0$, that is,
    \begin{align*}
        U^p_M(x) = \begin{cases}
        U^p(x) & U^p(x) <M\\
        M& U^p(x) \ge M
        \end{cases}.
    \end{align*}
    Then $U^p_M$ is a bounded continuous function. Since $Z_t$ converges to $\pi$ in law, we have that 
    \begin{align}\label{conv_in_law}
        \lim_{t \to \infty} \E_x[U^p_M(Z_t)] = \int_{\R_+^d} U^p_M(x) \pi(dx).
    \end{align}
    Therefore, by (\ref{conv_in_law}) and (\ref{pre_ergodic})
    \begin{align*}
       \int_{\R_+^d} U^p_M(x) \pi(dx)=\lim_{t\to\infty} \frac{1}{t} \int_0^t \E_x\left[U^p_M(Z_s)\right]ds \le \frac{c_p'}{c_p}.
    \end{align*}
    Now, take $M\to \infty$ on the left hand side and by monotone convergence theorem, we have 
    \begin{align*}
        \int_{\R_+^d} U^p(x) \pi(dx) <\infty.
    \end{align*}
    Since the inequality holds for all $p \in\mathbb{N}$, we may conclude that all moments of $\pi$ are finite.
\end{proof}

\begin{proof}[Proof of Theorem \ref{thm:stationaryconv}]
 Let $U$ be the Lyapunov function defined in (\ref{lyapunov_function}). Inequality (\ref{lyapunov}) says that $U$ satisfies the inequality (\ref{E:Feller_existInv2}), hence we can apply Theorem \ref{T:Feller_convInv} to conclude that $Z$ is $f$-exponentially ergodic with $f=U+1$.

\end{proof}

\begin{proof}[Proof of Proposition \ref{L:stationary1} ]
Following \cite{kang2014characterization}, we let $n(x) = \{\sum_{i \in {\cal I}(x)} \alpha_i e_i, \alpha_i > 0\}$ be the set of interior normal vectors to the domain $\R_+^d$ at  $x \in \partial\R_+^d$, where ${\cal I}(x) = \{1 \le i \le d: x_i = 0\}$. Let 
\begin{align*}
    {\cal U} :=\left\{x \in \partial \R_d^+: \exists n \in n(x) \text{ such that }n\cdot \gamma(x) > 0 \right\}.
\end{align*}
This definition is a bit different from that of \cite[equation (6)]{kang2014characterization} where they define $d(x)$ as a set valued function, but since our reflection $\gamma(x)$ is well defined for all $x \in \partial \R_+^d$ including the non-smooth part, we can set it to be single valued. If $x = (x_1,\cdots,x_d) \in \partial \R_+^d$, then there is some $i$ between $1$ and $d$ such that $x_i = 0$, hence $e_i \in n(x)$, so $b_i(x) = \lambda'$ and  
\begin{align*}
    \langle e_i, \gamma(x)\rangle =\frac{1}{\|b(x)\|} \langle e_i, b(x)\rangle = \frac{1}{\|b(x)\|} \lambda' > 0.
\end{align*}
Hence ${\cal V} := \partial \R_+^d\backslash {\cal U} = \emptyset$.

We prove Proposition (\ref{L:stationary1}) by checking all conditions in \cite[Theorem 2]{kang2014characterization}: note that $\Gamma(x)$ is uniformly elliptic for all $x \in \R_+^d$ by (\ref{EllipticGamma}) and the reflection $\gamma$ is piece-wise $C^2(\partial \R_+^d)$.
\cite[Assumption 2]{kang2014characterization} which is satisfied since ${\cal V} = \emptyset$ by \cite[Remark 3.4]{kang2014characterization}. The wellposedness of the submartingale problem in the statement of \cite[Theorem 2]{kang2014characterization} is given by \cite[Theorem 1]{kang2017submartingale} and Theorem \ref{T:Wellposed}. Now, all assumptions of \cite[Theorem 2]{kang2014characterization} are satisfied, which proves our statement.
\end{proof}

\begin{proof}[Proof of Proposition \ref{L:stationary2} ]

 We prove the statement by checking all conditions in \cite[Theorem 3]{kang2014characterization}:  by the proof of Proposition \ref{L:stationary1} we see that \cite[Assumption 2]{kang2014characterization} is satisfied and the corresponding submartingale problem is well posed . Furthermore, all entries of $\Gamma(\cdot)$ and $b(\cdot)$ are smooth since they are polynomials, so \cite[Theorem 3]{kang2014characterization} holds with $\R_+^d,\gamma(x),b(x),\Gamma(x)$ in the place of $G,d(x),b(x),a(x)$ in equation (12)-(16) of \cite[p. 1341]{kang2014characterization}.
\end{proof}


\end{document}